\documentclass[11pt, a4paper]{article}

\usepackage{fullpage, amsmath, dsfont, amsthm, amsfonts, mathrsfs, color, hyperref, graphicx}
\usepackage[inline]{enumitem}
\usepackage{epic}

\newcommand{\realR}{\mathbb{R}}
\newcommand{\intZ}{\mathbb{Z}}
\newcommand{\Unitary}{\mathrm{U}}
\newcommand{\ie}{i.e.}
\newcommand{\cf}{cf.}
\newcommand{\eg}{e.g.}
\newcommand{\etal}{et al.}
\newcommand{\resp}{resp.}

\newcommand{\normal}{\mathcal{N}}

\newcommand{\hvec}{\mathbf{h}}
\newcommand{\xvec}{\mathbf{x}}
\newcommand{\yvec}{\mathbf{y}}
\newcommand{\wvec}{\mathbf{w}}
\newcommand{\Pineiro}{Pi\~{n}eiro}
\newcommand{\Prahofer}{Pr\"{a}hofer}
\newcommand{\Peche}{P\'{e}ch\'{e}}

\newcommand{\id}{\mathds{1}}
\newcommand{\PartI}{\textbf{Part I}}
\newcommand{\PartII}{\textbf{Part II}}
\newcommand{\iid}{i.i.d.}

\newcommand{\Prob}{\mathbb{P}}
\newcommand{\lHopital}{l'H\^{o}pital}
\newcommand{\bigO}{\mathcal{O}}
\newcommand{\compC}{\mathbb{C}}
\newcommand{\U}{\mathrm{U}}

\DeclareMathOperator{\diag}{diag}
\DeclareMathOperator*{\res}{res}
\DeclareMathOperator*{\GUE}{GUE}
\DeclareMathOperator*{\FF}{F-F}
\DeclareMathOperator*{\JN}{J-N}
\DeclareMathOperator{\ess}{ess}
\DeclareMathOperator{\remainder}{remainder}
\DeclareMathOperator{\prelim}{prelim}
\DeclareMathOperator{\out}{out}
\DeclareMathOperator{\midd}{mid}
\DeclareMathOperator{\quadrant}{quad}

\newtheorem{thm}{Theorem}
\newtheorem{lem}{Lemma}
\newtheorem{prop}{Proposition}
\newtheorem{cor}{Corollary}

\theoremstyle{remark}
\newtheorem{rmk}{Remark}

\author{Mark Adler\thanks{
 Department of Mathematics, Brandeis University,
Waltham, Mass 02454, USA. E-mail: adler@brandeis.edu.
The support of a National Science Foundation grant \#
DMS-07-00782 is gratefully acknowledged.}~
~ Pierre
van Moerbeke\thanks{ Department of Mathematics,
Universit\'e de Louvain, 1348 Louvain-la-Neuve, Belgium
and Brandeis University, Waltham, MA 02454, USA. E-mail: pierre.vanmoerbeke@UCLouvain.be and
vanmoerbeke@brandeis.edu. The support of a National Science
Foundation grant \# DMS-07-00782, FNRS, PAI grants is
gratefully acknowledged.} ~~Dong Wang\thanks{Department of Mathematics, National University of Singapore, Singapore, 119076. E-mail: matwd@nus.edu.sg. The support of an NUS start-up grant \# R-146-000-164-133 is gratefully acknowledged.}}

\title{Random matrix minor processes related to percolation theory}

\begin{document}

\maketitle

\begin{abstract}
  This paper studies a number of matrix models of size $n$ and the
  associated Markov chains for the eigenvalues of the models for
  consecutive $n$'s. They are consecutive principal minors for two of the
  models, GUE with external source and the multiple Laguerre matrix model,
  and merely properly defined consecutive matrices for the third one, the
  Jacobi-Pi\~neiro model; nevertheless the eigenvalues of the consecutive
  models all interlace. We show: (i) For each of those finite models, we
  give the transition probability of the associated Markov chain and the
  joint distribution of the entire interlacing set of eigenvalues; we show
  this is a determinantal point process whose extended kernels share many
  common features. (ii) To each of these models and their set of
  eigenvalues, we associate a last-passage percolation model, either 
  finite percolation or percolation along an infinite strip of
  finite width, yielding a precise relationship between the last passage
  times and the eigenvalues. (iii) Finally it is shown that for
  appropriate choices of exponential distribution on the percolation, with
  very small means, the rescaled last passage times lead to the Pearcey
  process; this should connect the Pearcey statistics with random directed
  polymers.
\end{abstract}


\section{Introduction}

In this paper we are concerned with the (generalized) minor processes associated to random matrix models that are related to very classical orthogonal polynomials. We also show their relation to last-passage percolation models. This project was partially motivated by the following open problem: finding a continuum random directed polymer interpretation for the Pearcey process, and a corresponding stochastic heat equation, in the same way that the Airy process has a KPZ / stochastic heat equation / random polymer interpretation; see the work of Amir, Corwin and Quastel \cite{Amir-Corwin-Quastel11}. A first step in that direction is to show that the Pearcey process appears as a limit of a last-passage percolation model; this is done in the present work.

The first minor process arising from random matrix is the Gaussian Unitary Ensemble (GUE) minor process defined and analyzed in \cite{Johansson-Nordenstam06} by Johansson and Nordenstam. Following this result, other minor processes of classical random matrices are obtained \cite{Forrester-Nagao11}. The minor process of Laguerre Unitary Ensemble, aka complex white Wishart ensemble, leads to the generalized Wishart ensemble that was conjectured in \cite{Borodin-Peche08} and solved in \cite{Dieker-Warren09}. See also \cite{Forrester-Nordenstam09}, \cite{Ferrari-Frings10}, \cite{Adler-Nordenstam-van_Moerbeke10} and \cite{Adler-Nordenstam-van_Moerbeke10a} for other minor processes related to random matrices.

Analogous to complex Wishart ensemble (aka Laguerre Unitary Ensemble with external source) that is a generalization of the classical Laguerre Unitary Ensemble (LUE), the GUE with external source is a generalization to the classical Gaussian Unitary Ensemble. In this paper we consider the minor process associated to this matrix model.

In the same spirit as the complex Wishart ensemble and GUE with external source, there is a new matrix model that generalizes the classical Jacobi Unitary Ensemble (JUE) that we denote as Jacobi-\Pineiro\ ensemble, and another generalization of the LUE that we denote as multiple Laguerre ensemble (to be distinguished with the complex Wishart ensemble).

We describe the minor processes in a systematic way as follows. In all cases below, $W_n,X_n,Y_n$ denote the matrix of the first $n$ columns of the $M\times N$-matrix $W,X$ (\resp\ the $M' \times N$-matrix $Y$) for $1\leq n\leq N$, where $M, M', N$ are positive integers and we assume $M \geq N$ and $M' \geq N$. For square matrices $Z_0$ and $A$ of size $N$, $(Z_0+A)_n$ denotes the $n$-th principal minor of the square matrix $Z_0+A$, also for $1\leq n\leq N$. For the definition of $X$, we need parameters $\alpha_i$ that are nonnegative integers satisfying
 \begin{equation}
   1 + \alpha_1 \leq 2 + \alpha_2 \leq \dots \leq N + \alpha_N.
 \end{equation}

\paragraph{The four (generalized) minor processes}

Consider the spectra $\lambda^{(1)}, \dotsc, \lambda^{(N)}$ of the following consecutive matrices $S_n$, where $\lambda^{(n)} = (\lambda^{(n)}_1, \dots, \lambda^{(n)}_n)$ and $\lambda^{(n)}_i$ are in increasing order.

\begin{enumerate}
   \item
     (GUE with external source)
     The $n\times n$ consecutive matrices\footnote{GUE$(N)$ is standardly defined as Hermitian matrices with independent normal entries, with $\Re (Z_{0})_{ij}={\mathcal N}(0,\frac 12 ), ~\Im (Z_{0})_{ij}={\mathcal N}(0,\frac 12 )$ and $(Z_{0})_{ii}={\mathcal N}(0,1 )$.}    (minors of $S_N$)
     \begin{equation} \label{eq:case_1_defn}
       S_n=(Z_0+A)_n
       ~~ \mbox{with} ~~
       \left\{\begin{array}{l}Z_0=\mathrm{GUE}(N), \\
           A=\diag (a_1,\ldots, a_N). 
         \end{array}\right.
     \end{equation}
   
   \medbreak
   
   \item
     (Wishart)
   \noindent    
   The $n\times n$ consecutive matrices (minors of $S_N$)
   \begin{equation} \label{eq:case_2_defn}
     S_n=W_n^\ast W_n \mbox{   with  }
     \left\{\begin{array}{l}W=M\times N  \mbox{  matrix,   }   ~~~\\
         \Re W_{ij} ={\mathcal N}(0,-\frac 1{2a_j}), ~~\\
         \Im W_{ij} ={\mathcal N}(0,-\frac 1{2a_j}),
       \end{array}\right.
   \end{equation}
   where $a_1, \dotsc, a_N$ are negative parameters.
   
   \item
     (Multiple Laguerre)
   \noindent  
   The $n\times n$ consecutive matrices (minors  of $S_N$)
   \begin{equation} \label{eq:case_3_defn}
     S_n=X_n^\ast X_n
     ~ \mbox{    with }~~\left\{\begin{array}{l}\begin{array}{l}
           X=
           M\times N  \mbox{  matrix,   } \\
           \hspace{-.3cm} \left.\begin{array}{l}\Re X_{ij}={\mathcal N}(0,1/2 ) 
               \\
               \Im   X_{ij}={\mathcal N}(0, 1/2 )\end{array}\right\}\mbox{  for  } 1\leq i\leq j+\alpha_j, \end{array}\\
         \mbox{and $X_{ij}=0$ otherwise.} \end{array}\right.
   \end{equation}
   
 \item
   (Jacobi-\Pineiro)
   \noindent    The 
   $n\times n$ 
   consecutive matrices 
   \begin{equation} \label{eq:case_4_defn}
     S_n=  \frac{X^{\ast}_n X_n}{X^{\ast}_n X_n+Y^{\ast}_n Y_n}
     \mbox{  with   } \left\{\begin{array}{l}\begin{array}{l}
           X = M\times N  \mbox{  matrix, as before,}\\
           Y=
           M'\times N  \mbox{  matrix,   } \\
           \hspace{-.3cm}  \begin{array}{l}\Re Y_{ij}={\mathcal N}(0,1/2 ) 
             ,~~~
             \Im   Y_{ij}={\mathcal N}(0, 1/2 ).\end{array} \end{array}
       \end{array}\right.
   \end{equation}
   
 \end{enumerate}

\begin{rmk}
  Note that in the definitions of the first three minor processes, $S_{n - 1}$ is a minor of $S_n$, so the name ``minor process'' is assigned. Although in the last process, $S_{n - 1}$ is not a minor of $S_n$, by definition the numerator and denominator of $S_{n - 1}$ are minors of those of $S_n$ respectively, and also in that case the eigenvalues of $S_{n - 1}$ and those of $S_n$ are interlaced as for eigenvalues of a Hermitian matrix and its minor. Thus the name minor process is also justified.
\end{rmk}
\begin{rmk}
  The Whishart minor process is a special case of the so called generalized Wishart random-matrix process \cite{Dieker-Warren09}, and its properties has been already known. We include it in this paper for completeness.
\end{rmk}

Previously studies minor processes are shown to be the continuous limits of special Schur processes and are equivalent to continuous last-passage percolation models with properly chosen parameters and possibly taking limit (see \eg\ \cite{Defosseux10}, \cite{Forrester-Rains06}, \cite{Dieker-Warren09} and \cite{Forrester-Nagao11} for the cases most close to ours). The minor processes considered in our paper also have this property.

\begin{figure}[ht]
  \centering
  \includegraphics{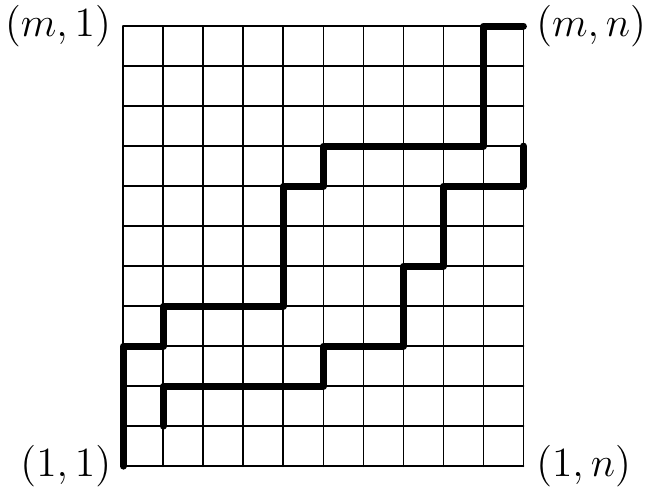}
  \caption{Two up-right paths in $\Pi(m, n)$ that do not intersect.}
  \label{fig:upright_path}
\end{figure}
\bigbreak 
Consider the percolation model on the $\intZ^+ \times \intZ^+$ lattice $\{ (i, j) \mid i, j = 1, 2, \dotsc \}$, where on each site we associate a real weight $x_{ij}$. Let $\Pi(m, n)$ be the set of up-right paths in the rectangle with vertices $(1, 1), (1, n), (m, 1), (m, n)$, see Figure \ref{fig:upright_path}. Then we define the maximum of the total length of $l$ non-intersecting up-right paths in that rectangle to be
\begin{equation} \label{eq:defn_of_L^l(N,n)}
  L^{(l)}(m, n) := \max_{\substack{P_1, \dots, P_l \in \Pi(m,n) \\ P_1, \dots, P_l \text{ non-intersecting}}} \sum^l_{k=1} \sum_{(i,j) \in P_k} x_{ij}.
\end{equation}
In the case $l = 1$, $L^{(l)}(m, n)$ is the length of the longest up-right path from $(1,1)$ to $(m, n)$. Let $p_1, p_2, \dotsc; q_1, q_2, \dotsc$ be two sets of positive parameters. and let the $x_{ij}$'s be independent and exponentially distributed of parameter $\pi_{ij}=p_i + q_j>0$,
\begin{equation} \label{eq:exp_distr_added}
  \Prob(x_{ij}\geq t)=e^{-\pi_{ij} t}.
\end{equation}
Then the $L^{(l)}(m, n)$ become random variables, which for fixed $m, n$ are increasing as $l$ increases. We define the $n$-variate random variables \footnote{In \eqref{eq:defn_mu^mn} the components of the vector $\mu^{(m, n)}$ are weakly decreasing. But later in this paper, some vectors have (weakly) increasing components. Sometimes we say an array of increasing random variables and an array of decreasing random variables have the same distribution, if they have the same distribution after reversing the order of either between them.}
\begin{equation} \label{eq:defn_mu^mn}
  \mu^{(m, n)} = (\mu^{(m, n)}_1, \dotsc, \mu^{(m, n)}_n) \quad \text{where} \quad \sum^l_{i = 1} \mu^{(m, n)}_i = L^{(l)}(m, n).
\end{equation}

Now we state the distribution of the eigenvalues $\lambda^{(n)}_j$ in the minor processes defined above. Before the statement of the theorem, we first define the notation $\mu \preceq \nu$ of interlacing between an $(n-1)$-variate random variable $\mu = (\mu_1, \dotsc, \mu_{n - 1})$ and an $n$-variate random variable $\nu = (\nu_1, \dotsc, \nu_n)$ such that their components are in increasing order and are in a common domain $I$.

\begin{equation} \label{eq:interlacing_defn}
  \mu \preceq \nu \quad \text{if $\nu_1 \leq \mu_1 \leq \nu_2 \leq \mu_2 \leq \dotsc \leq \mu_{n - 1} \leq \nu_n$, where $\mu_i, \nu_i \in I$}.
\end{equation}

\begin{thm} \label{thm:Markov}
  The interlacing sets of spectra $\lambda^{(1)},\ldots,\lambda^{(n)},\ldots,\lambda^{(N)}$ constitute an inhomogeneous Markov chain with transition probability:
  \begin{equation} \label{eq:transition_general}
    P_{n - 1, n}(\lambda^{(n - 1)}, \lambda^{(n)}) = \frac{1}{C_n} \frac{\Delta_n(\lambda^{(n)}) \prod^n_{i = 1} w_n(\lambda^{(n)}_i)}{\Delta_{n - 1}(\lambda^{(n - 1)}) \prod^{n - 1}_{i = 1} w_n(\lambda^{(n - 1)}_i) \psi(\lambda^{(n - 1)}_i)} 1_{\lambda^{(n - 1)}  \preceq 
      \lambda^{(n)}},
  \end{equation}
  where $\Delta_n(\lambda^{(n)}) = \prod_{1 \leq j < k \leq n} (\lambda^{(n)}_k - \lambda^{(n)}_j)$ is the Vandermonde determinant, the range $I$ of the eigenvalues, the weight $w_n(z)$, the function $\psi(z)$ and the constant $C_n$ are given in Table \ref{table:thm}.
\end{thm}
\begin{table}[ht]
    \centering
    \begin{tabular}{|l||l|l|l|l|l|l}
      \hline
      \mbox{model}   & $I$ & $w_n(z)$ & $\psi(z)$ & $C_n$ & $p_1(z)$ \\
      \hline
      1. GUE ext source &  $\mathbb R$ & $e^{-\frac{z^2}{2} + a_n z}$ & $1$ & $\sqrt{2 \pi} e^{\frac{a^2_n}{2}}$ & $\frac{1}{\sqrt{2\pi}} e^{-\frac{1}{2} (z - a_1)^2}$ \\  
      2. Wishart &  $[0,\infty)$ & $z^{M-n} e^{a_n z}$ & $z$ & $\frac{(M - n)!}{(-a_n)^M}$ & $\frac{(-a_1)^M}{(M - 1)!} z^{M - 1} e^{a_1 z}$ \\    
      3. Mult. Laguerre & $[0,\infty)$ & $z^{\alpha_n} e^{-z} $ & $z$ & $\alpha_n !$ & $\frac{z^{\alpha_1}}{\alpha_1 !} e^{-z}$ \\ 
      4. Jacobi-\Pineiro & $[0,1]$ & $z^{\alpha_n} (1-z)^{M' - n}$ & $z(1-z)$ & $\frac{\alpha_n ! (M' - n)!}{(\alpha_n + M')!}$ & $\frac{\prod^{M'}_{i = 1} (\alpha_i + i)}{(M' - 1)!} (1 - z)^{M' - 1} z^{\alpha_1}$ \\
      \hline
    \end{tabular}
    \caption{The range $I$, the weight $w_n(z)$, the function $\psi(z)$, the constant $C_n$, and the probability density function $p_1(z)$ of the eigenvalue $\lambda^{(1)}_1$ of $S_1$ for the four minor processes.}
    \label{table:thm}
  \end{table}

As a corollary to Theorem \ref{thm:Markov}, we have the joint distribution of $\lambda^{(n)}_i$ in the minor processes.
\begin{cor} \label{cor:jpdf}
  In each of the four minor processes associated to the GUE with external source, Wishart ensemble, multiple Laguerre ensemble and Jacobi-\Pineiro\ ensemble respectively, the joint distribution of the eigenvalues $\lambda^{(1)} \in I^1, \lambda^{(2)} \in I^2, \dotsc, \lambda^{(N)} \in I^N$ is 
  \begin{equation}
    P(\lambda^{(1)}, \dotsc, \lambda^{(N)}) = \Delta_N(\lambda^{(N)}) \prod^N_{n = 1} \frac{1}{C_n} w_N(\lambda^{(N)}_n) \prod^{N - 1}_{n = 1} \left( \prod^n_{i = 1} \frac{w_n(\lambda^{(n)}_i)}{w_{n + 1}(\lambda^{(n)}_i) \psi(\lambda^{(n)}_i)} \right) \id_{\lambda^{(n)} \preceq \lambda^{(n + 1)}},
  \end{equation}
  or equivalently, (with the $\lambda^{(n)}_{n + 1}$ virtual variables introduced for convenience)
  \begin{equation} \label{eq:eqivalent_jpdf}
    P(\lambda^{(1)}, \dotsc, \lambda^{(N)}) = \Delta_N(\lambda^{(N)}) \prod^N_{n = 1} \frac{1}{C_n} w_N(\lambda^{(N)}_n) \prod^{N - 1}_{n = 1} \det(\phi_n(\lambda^{(n)}_i, \lambda^{(n + 1)}_j))^{n + 1}_{i, j = 1},
  \end{equation}
  where $\phi_n(\lambda^{(n)}_i, \lambda^{(n+1)}_j)$ is given by
  \begin{equation} \label{eq:phi_JUE_2v}
    \phi_n(x,y) = \frac{w_n(x)}{w_{n + 1}(x) \psi(x)} \id_{x < y}, \quad x \neq \lambda^n_{n+1},
  \end{equation}
  and otherwise
  \begin{equation} \label{eq:phi_LUE_1v}
    \phi_n(\lambda^{(n)}_{n+1}, x) = 1.
  \end{equation}
  Here the domain $I$, constant $C_n$ and functions $w_n, \psi$ are defined in Table \ref{table:thm}.
\end{cor}
To prove Corollary \ref{cor:jpdf}, we use the Markovian property of $\lambda^{(n)}$, apply \eqref{eq:transition_general} inductively, and note the initial condition that the probability density function of $\lambda^{(1)}_1$ in the four minor processes is given by $p_1(z)$ in Table \ref{table:thm}. One can check that the determinant in \eqref{eq:eqivalent_jpdf} encodes the interlacing property.

The next theorem shows the equivalence between the minor processes and the last-passage percolation models. 
\begin{thm} \label{thm:Schur_corrsp}
  To the exponentially distributed percolation model $\Pi_{M, N}$ with  parameter 
  \begin{equation}
    \mbox{$\pi_{ij}:=p_i+q_j$ for $1\leq i\leq M$ and $1\leq j\leq N$, with $N \leq M$,}
  \end{equation}
  we associate the variables $\mu_i^{(M, n)} , ~1\leq i\leq n$ as in \eqref{eq:defn_mu^mn}. Then the following holds:
  \begin{enumerate}
  \item
    For $\pi_{ij}=1-a_j/\sqrt{M}$, the scaling limit $\nu^{(  1)},  \ldots,\nu^{(  N)}$ of the percolation variables $\mu^{(  M,1)},  \ldots,\mu^{( M, N)}$, 
    \begin{equation} \label{eq:nu_sequence_added}
      \nu_i^{(n)}:=\lim_{M\to \infty}\frac{\mu^{(M, n)}_i - M}{\sqrt{M}},~~~~~~1\leq i\leq n \leq N,
    \end{equation}
    has the same joint distribution as $\lambda^{(1)},\ldots,\lambda^{(N)}$ in the GUE with external source minor process defined in \eqref{eq:case_1_defn}.
  \item (A special case of \cite[Theorem 1.1]{Dieker-Warren09})
    For $\pi_{ij}= -a_j $, the percolation variables $\mu^{( M, 1)},  \ldots,\mu^{( M, N)}$ have the same joint distribution as $\lambda^{(1)},\ldots,\lambda^{(N)}$ in the Wishart minor process defined in \eqref{eq:case_2_defn}.
  \item
    For $\pi_{ij}=i+\alpha_j $, the scaling limit $\nu^{(  1)},  \ldots,\nu^{(  N)}$ of the percolation variables $\mu^{(  M,1)},  \ldots,\mu^{( M, N)}$,
    \begin{equation}
      \nu^{(  n)}_i := \lim_{M\to \infty}  M e^{-\mu^{(M, n)}_i},~~~~~~1\leq i\leq n \leq N,
    \end{equation}
    has the same joint distribution as the $\lambda^{(1)},\ldots,\lambda^{(N)}$ in the multiple Laguerre minor process defined in \eqref{eq:case_3_defn}.
  \item
    For $\pi_{ij}=i+\alpha_j $, the variables $\nu^{(1)},  \ldots,\nu^{(N)}$ obtained by exponentiating the percolation variables $\mu^{( M', 1)},  \ldots,\mu^{( M', N)}$,
    \begin{equation}
      \nu^{(n)}_i :=  e^{-\mu^{(M', n)}_i},~~~~~~1\leq i\leq n \leq N\leq M',
    \end{equation}
    have the same joint distribution as the $\lambda^{(1)},\ldots,\lambda^{(N)}$ in the Jacobi-\Pineiro\ minor process defined in \eqref{eq:case_4_defn}.
  \end{enumerate}
\end{thm}


Furthermore, we show that the correlation function of $\lambda^{(n)}_i$ in each minor process has a determinantal formula, and the correlation kernel has a double contour integral formula.
\begin{thm} \label{thm:kernel}
  In each of the four minor processes associated to the GUE with external source, Wishart ensemble, multiple Laguerre ensemble and Jacobi-\Pineiro\ ensemble respectively, the correlation kernel of eigenvalues in $S_{n_1}$ and $S_{n_2}$ is given by
  \begin{enumerate}[label*=(\alph*)]
  \item \label{enu:thm:kernel_GUE}
    GUE with external source:
    \begin{multline} \label{eq:compact_integral_formula_of_K}
      K(n_1,x; n_2,y) =  \frac{-1}{2\pi i} \oint_{\Gamma_a} \frac{e^{(y-x)w}}{\prod^{n_2}_{l = n_1+1} (w - a_l)} dw  \id_{x<y} \id_{n_1 < n_2} \\
      + \frac{1}{(2\pi i)^2} \int_{\Sigma} dz \oint_{\Gamma_a} dw e^{\frac{z^2}{2} - \frac{w^2}{2} -xz + yw} \frac{\prod^{n_1}_{k = 1} (z - a_k)}{\prod^{n_2}_{l = 1} (w - a_l)} \frac{1}{z-w},
    \end{multline}
    where the contour $\Gamma_a$ encloses all poles of the form $a_i$ in the intergrand, and $\Sigma = C + i\realR \uparrow$ lies to the right of $\Gamma_a$.
  \item \label{enu:thm:kernel_Wishart}
    Wishart (a special case of \cite[Formula (15)]{Borodin-Peche08}): 
    \begin{multline} \label{eq:correlation_Wishart}
      K(n_1,x; x_2,y) =  \frac{-1}{2\pi i} \oint_{\Gamma_a} \frac{e^{(y-x)w}}{\prod^{n_2}_{l = n_1+1} (w - a_l)} dw  \id_{x<y} \id_{n_1 < n_2} \\
      + \frac{1}{(2\pi i)^2} \int_{\Sigma} dz \oint_{\Gamma_a} dw e^{-xz + yw} \left( \frac{w}{z} \right)^M \frac{\prod^{n_1}_{k = 1} (z - a_k)}{\prod^{n_2}_{l = 1} (w - a_l)} \frac{1}{w-z},
    \end{multline}
    where the contour $\Gamma_a$ encloses all poles of the form $a_i$ in the intergrand, and $\Sigma$ encloses $0$, and does not cross or contain $\Gamma_a$.
  \item \label{enu:thm:kernel_Laguerre}
    Multiple Laguerre:
    \begin{multline} \label{eq:compact_integral_formula_of_K_LUE}
      K(n_1,x;n_2,y) = \frac{-1}{2\pi i} \oint_{\Gamma_{\alpha}} \frac{x^{-w - 1}y^w}{\prod^{n_2}_{l=n_1+1} (w-\alpha_l)} dw \id_{x < y} \id_{n_1 < n_2} \\
      + \frac{1}{(2\pi i)^2}  \oint_{\Sigma}dz \oint_{\Gamma_{\alpha}}dw \frac{x^{-z-1}y^w \Gamma(z+1)}{(z-w) \Gamma(w+1)} \frac{\prod^{n_1}_{k=1} (z-\alpha_k)}{\prod^{n_2}_{l=1} (w-\alpha_l)},
    \end{multline}
    where $\Gamma_{\alpha}$ is a contour enclosing $\alpha_1, \dots, \alpha_{n_2}$, and $\Sigma$ is a deformed Hankel contour going counterclockwise from $-\infty$ to $-\infty$ that encloses poles $z = -1, -2, \dots$ (see Figure \ref{fig:Hankel}) and the contour $\Gamma_{\alpha}$.
  \item \label{enu:thm:kernel_Jacobi}
    Jacobi-\Pineiro:
    \begin{multline} \label{eq:compact_integral_formula_of_K_JUE}
      K(n_1, x; n_2, y) = \frac{-1}{2\pi i} \oint_{\Gamma_{\alpha}} \frac{x^{-w-1} y^w}{\prod^{n_2}_{l = n_1+1} (w-\alpha_l)} dw \id_{x < y} \id_{n_1 < n_2} \\
      + \frac{1}{(2\pi i)^2} \oint_{\Sigma}dz \oint_{\Gamma_{\alpha}}dw \frac{x^{-z-1}y^w}{z-w} \frac{\Gamma(w + M' + 1) \Gamma(z + 1)}{\Gamma(z + M' + 1) \Gamma(w + 1)} \frac{\prod^{n_1}_{k=1} (z-\alpha_k)}{\prod^{n_2}_{l=1} (w-\alpha_l)},
    \end{multline}
    where the contour $\Gamma_{\alpha}$ is a contour inclosing $\alpha_1, \dots, \alpha_{n_2}$, and $\Sigma$ is a contour going counterclockwise enclosing $-1, -2, \dotsc, -M'$ and the contour $\Gamma_{\alpha}$.
  \end{enumerate}
\end{thm}

\begin{figure}[ht]
  \centering
  \includegraphics{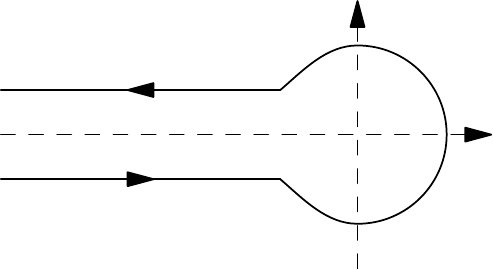}
  \caption{The deformed Hankel contour. It comes from $-\infty$, ends at $-\infty$ and enclose $(-\infty, 0)$ counterclockwise. It keeps a constant distance to the real line at $-\infty$.}
  \label{fig:Hankel}
\end{figure}

\begin{rmk}
  As a special case of the GUE with external source minor process, if all $a_i = 0$, our model becomes the well known GUE minor model. In \cite[Formula (21)]{Ferrari-Frings10}, the correlation kernel of the GUE minor model was obtained as the sum of two contour integrals. Due to a different choice of scaling convention, the kernel in \cite[Formula (21)]{Ferrari-Frings10} is related to our kernel in \eqref{eq:compact_integral_formula_of_K} by
  \begin{equation}
    K^{\GUE}_{\FF}(\xi_1, n_1; \xi_2, n_2) d\xi_2 = \left. 2^{(n_2 - n_1)/2} \sqrt{2} K(n_1, x; n_2, y) dy \right\rvert_{x = \sqrt{2}\xi, y = \sqrt{2}\eta}.
  \end{equation}
  The correlation kernel of the GUE minor model was first discovered in \cite[Definition 1.2]{Johansson-Nordenstam06}, (see \cite[Definition 1.1]{Johansson-Nordenstam06a} for erratum) by Johansson and Nordenstam in terms of Hermite polynomials. Their kernel is related to ours by
  \begin{equation}
    K^{\GUE}_{\JN}(r, \xi; s, \eta) d\xi = \left. e^{\frac{x^2 - y^2}{4}} K(n_1, x; n_2, y) dy \right\rvert_{\substack{n_1 = s, n_2 = r \\ x = \sqrt{2}\eta, y = \sqrt{2}\xi}}.
  \end{equation}
  Although the formula of the correlation kernel in our paper is slightly different from those in previous literature, they define the same correlation function of eigenvalues of minors, since the difference is simply a conjugation and change of variables. Hence Theorem \ref{thm:Markov} is a generalization of previous results. The recent preprint \cite{Ferrari-Frings12} by Ferrari and Frings that appeared shortly before the first preprint version of this paper obtained essentially the same result of Theorem \ref{thm:kernel}\ref{enu:thm:kernel_GUE}. Their kernel \cite[Formula (4)]{Ferrari-Frings12} with $t = 1$ is related to our kernel \eqref{eq:compact_integral_formula_of_K} by
  \begin{equation}
    K_{1, \FF}((x, n), (x', n')) dx' = \left. (-1)^{n_2 - n_1} K(n_1, x; n_2, y) dy \right\rvert_{\substack{n_1 = n, n_2 = n' \\ x = x, y = x'}}.
  \end{equation}
\cite{Ferrari-Frings12} also relates the GUE with external source minor process to Warren's process with drifts \cite{Warren07} and an interacting particle system \cite{Borodin-Ferrari08}, \cite{Borodin-Ferrari08a}.
\end{rmk}
\begin{rmk}
  If we let all $a_n = 0$ in the GUE with external source minor process or all $\alpha_n = M - n$ in the multiple Laguerre minor process and the Jacobi-\Pineiro\ minor process, they are reduced to the GUE, LUE and JUE minor processes respectively, and we can check that the correlation functions in the special cases agree with those obtained in \cite{Forrester-Nagao11} for these three minor processes. 
\end{rmk}

The minor processes studied in this paper are also related to directed polymers. Specialize the exponentially distributed percolation model $\Pi_{M, n}$, for all $1 \leq n \leq N$, as in \eqref{eq:exp_distr_added}, to
\begin{equation}
  \pi_{ij} =
  \begin{cases}
    1 - \frac{\kappa_1}{\sqrt{M}} & \text{for $1 \leq j \leq r_1$}, \\
    1 - \frac{\kappa_2}{\sqrt{M}} & \text{for $r_1 + 1 \leq j \leq n$},
  \end{cases}
\end{equation}
with $L^{(l)}(M, n)$ as in \eqref{eq:defn_of_L^l(N,n)}. As a shorthand notation, we set
\begin{equation}
  \kappa_* =
  \begin{cases}
    \kappa_1 & \text{in the region where $j \in [1, r_1]$}, \\
    \kappa_2 & \text{in the region where $j \in [r_1 + 1, n]$}.
  \end{cases}
\end{equation}
Then from Theorem \ref{thm:Schur_corrsp} it follows that
\begin{cor} \label{cor:polymer_added}
  Given $n$ independent standard Brownian motions $B_j(t)$ run along the vertical lines $j = 1, \dotsc, n$, the first component $\nu^{(n)}_1$ of the vector $\nu^{(n)} = (\nu^{(n)}_1 \geq \dotsb \geq \nu^{(n)}_n)$ as in \eqref{eq:nu_sequence_added}, can be expressed as a directed polymer problem (for continuous times $t_i \in \realR_+$)
  \begin{equation}
    \nu^{(n)}_1 = \lim_{M \to \infty} \frac{L^{(1)}(M, n) - M}{\sqrt{M}} = \sup_{0 = t_0 < t_1 < \dotsb < t_n = 1} \sum^n_{j = 1} [B_j(t_j) - B_j(t_{j - 1}) + \kappa_*(t_j - t_{j - 1})].
  \end{equation}
  More generally, consider non-intersecting right-upper paths $\pi_k$, $1 \leq k \leq l$, leaving from the left-most adjacent points at $t = 0$ and going to the right-most adjacent points at $t = 1$; see Figure \ref{fig:polymer}. The time $0 = t^{(k)}_1 < t^{(k)}_1 < \dotsb < t^{(k)}_{n - l + 1} = 1$ are the associated instants of jump for each path $\pi_k$, with the necessary interlacing of the instants $t^{(k)}_i$ in order to respect the non-intersecting nature of the paths. Also consider independent standard Brownian motions $B^{(k)}_i(t)$, associated with each path $\pi_k$ and each vertical line $k \leq i \leq n - l + k$. Then the other components $v^{(n)}_l$ of the vector $v^{(n)}$ have an interpretation in terms of directed percolation, namely for $1 \leq l \leq n$ one has
  \begin{equation} \label{eq:nu_added}
    \begin{split}
      \sum^l_{i = 1} \nu^{(n)}_i = {}& \lim_{M \to \infty} \frac{L^{(l)}(M, n) - lM}{\sqrt{M}} \\
      = {}& \sup_{\pi_1, \dotsc, \pi_l \in \Pi_n} \sum^l_{k = 1} \sum^{n - l + 1}_{j = 1} \left[ B^{(k)}_{j + k - 1}(t^{(k)}_j) - B^{(k)}_{j + k - 1}(t^{(k)}_{j - 1}) + \kappa_*(t^{(k)}_j - t^{(k)}_{j - 1}) \right],
    \end{split}
  \end{equation}
  where the paths $\pi_1, \dotsc, \pi_l \in \Pi_n$ are non-intersecting right-upper paths.
\end{cor}
See also \cite{Baryshnikov01}, \cite{Glynn-Whitt91} and \cite{Gravner-Tracy-Widom02}.
\begin{figure}[ht]
  \centering
  \includegraphics{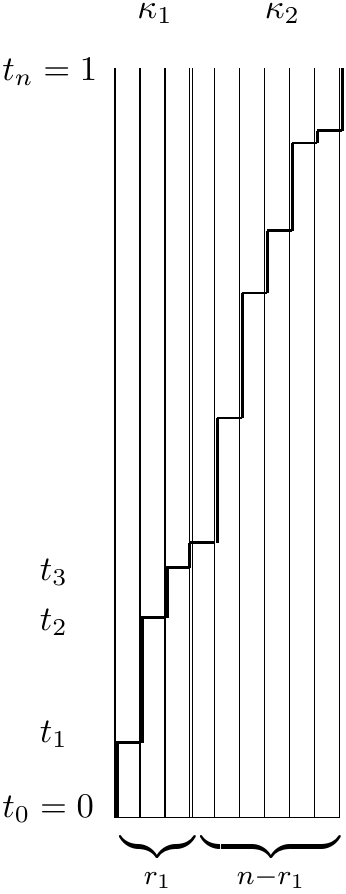} \hspace{0.2\textwidth} \includegraphics{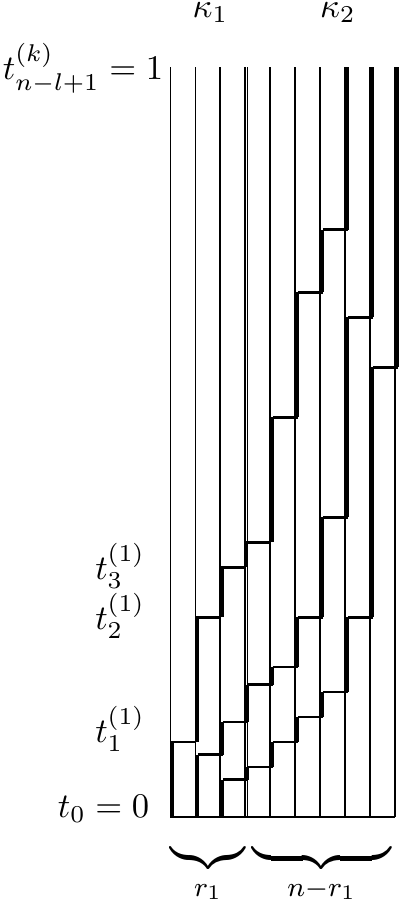}
  \caption{Directed polymers for $l=1$ and $l=3$.}
  \label{fig:polymer}
\end{figure}
  
It is not a coincidence that these four minor processes share so many similarities. They are closely related to the so called very classical multiple orthogonal polynomials, namely the multiple Hermite (to GUE with external source), multiple Laguerre of the second kind (to Wishart), multiple Laguerre of the first kind (to multiple Laguerre) and Jacobi-\Pineiro\ (to Jacobi-\Pineiro) respectively. Here we state without proof the  correlation kernel in the Wishart minor process \eqref{eq:correlation_Wishart} can be written in the form of multiple Laguerre polynomials of the second kind when $n_1 \geq n_2$ (see \eqref{eq:mLaguerre_2nd_I} and \eqref{eq:mLaguerre_2nd_II} for the definition of $P_{(a_{n_1}, a_{n_1 - 1}, \dotsc, a_{n_{k + 1}}; M - n_1)}(x)$ and $Q_{(a_{n_2}, a_{n_2 - 1}, \dotsc, a_{n_k}; M - n_2)}(y)$)
\begin{equation} \label{eq:kernel_Wishart_MOP}
  K(n_1, x; n_2, y) = (-1)^{n_2 - n_1} \frac{x^{M - n_1}}{y^{M - n_2}} \sum^{n_2}_{k = 1} P_{(a_{n_1}, a_{n_1 - 1}, \dotsc, a_{n_{k + 1}}; M - n_1)}(x) Q_{(a_{n_2}, a_{n_2 - 1}, \dotsc, a_{n_k}; M - n_2)}(y).
\end{equation}
When $n_1 < n_2$, the kernel is also related to multiple Laguerre polynomials of the second kind, but the relation is not so simple. Although we do not prove \eqref{eq:kernel_Wishart_MOP}, we indicate that it can be proven based on the joint probability density \eqref{eq:eqivalent_jpdf} and the algebraic result in Lemma \ref{lem:BFPS} of the next section, leading to \eqref{eq:correlation_Wishart}, and we prove similar formulas for the other three minor processes as intermediate steps in the derivation of double contour integral formulas \eqref{eq:compact_integral_formula_of_K}, \eqref{eq:compact_integral_formula_of_K_LUE} and \eqref{eq:compact_integral_formula_of_K_JUE}. 

It is well known that in the Gaussian unitary ensemble, Laguerre unitary ensemble and Jacobi unitary ensemble, if the dimension approaches infinity and we consider the correlation among eigenvalues at the edge of the support of their limiting empirical distribution, then we find the correlation kernel has the limit as the Airy kernel. If we consider the corresponding minor processes, it is shown in \cite{Forrester-Nagao11} that the joint distributions of eigenvalues of consecutive minors around the edge of the limiting empirical distribution has the limit as the extended Airy kernel that defines the Airy process. Since the minor processes discussed here are generalizations to the three minor processes analyzed in \cite{Forrester-Nagao11}, one may expect that they can realize more complicated correlation kernels as their limiting kernels as the dimension of minors approaches infinity and the parameters $a_i$ or $\alpha_i$ are chosen properly. Indeed, in \cite{Borodin-Peche08}, Borodin and \Peche\ shows that the generalized Airy kernel with two sets of parameters can be realized in the limit of the generalized Wishart random-matrix process, which is a generalization of the Wishart minor process in our paper. In this paper, we show that if we choose parameters properly, the Pearcey kernel can be realized as the limit of the correlation kernel of the multiple Laguerre minor process. Similar result should hold for the other three minor processes considered in this paper, but we only analyze the multiple Laguerre case for brevity.

We consider one special case of the multiple Laguerre ensemble, which is the analytically most feasible one (besides the white Wishart ensemble), such that the parameter $\alpha_i$ are of only two values, namely, half of them are $na$ and the other half $nb$. When $\alpha_i$ are chosen in this simplest way, however, the region that the nonzero entries in $X$ occupy is a composition of two trapezoids, aesthetically not of the simplest shape, compared with the composition of two rectangles. In the later case we also observe the Pearcey process when taking limit properly, but we omit the details for brevity.

\paragraph{The trapezoidal multiple Laguerre minor process}

We define the $M \times N$ random matrix $X$ as follows, depending on a large integer $n$ and two parameters $a, b$. Let the left $n$ columns of $X$ be determined by the parameter $a$ and the other $N - n$ columns by $b$, such that in the $i$-th column where $i \leq n$, the top $i + \lfloor na \rfloor$ entries are in \iid\ standard complex normal distribution and the bottom $M - i - \lfloor na \rfloor$ entries are zero, while if $i > n$, the top $i + \lfloor nb \rfloor$ entries are in \iid\ standard normal distribution and the bottom $M - i - \lfloor nb \rfloor$ entries are zero. We shall take the limit $n \to \infty$, and assume that $N$ is large enough, (say, at least greater than $2n$), and $M \geq N$ is large enough so that $M \geq i + na$ and $M \geq i + nb$ in any column. In terms of parameters $\alpha_i$, the $M \times N$ random matrix $X$ is characterized by
\begin{equation} \label{eq:alpha_i_trapezoidal}
  \alpha_1 = \alpha_2 = \dotsb = \alpha_n = \lfloor na \rfloor, \quad \text{and} \quad \alpha_{n + 1} = \alpha_{n + 2} = \dotsb = \lfloor nb \rfloor.
\end{equation}

 For our purpose to analyze the limiting Pearcey process, $a$ and $b$ are chosen in the way that (see Figure \ref{fig:f_and_g})
\begin{equation} \label{eq:defn_of_x_0}
  \begin{gathered}
    \text{the graphs of $f(x) = \frac{1}{(x - a)^2} + \frac{1}{(x - b)^2}$ and $g(x) = \frac{1}{x}$} \\
    \text{intersect at a unique point $x_0$ on the interval $(a, b)$.}
  \end{gathered}
\end{equation}
\begin{figure}[ht]
  \begin{minipage}[t]{0.45\linewidth}
    \centering
    \includegraphics{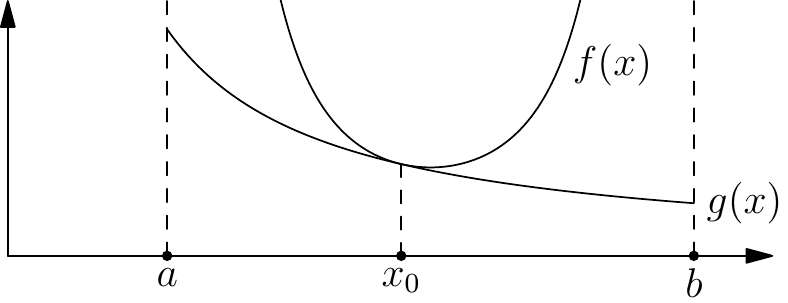}
    \caption{An example of $a$, $b$ and $x_0$, where $a \approx 1.8748$, $b \approx 8.0752$ and $x_0 \approx 4.6305$.}
    \label{fig:f_and_g}
  \end{minipage}
  \hspace{\stretch{1}}
  \begin{minipage}[t]{0.45\linewidth}
    \centering
    \includegraphics{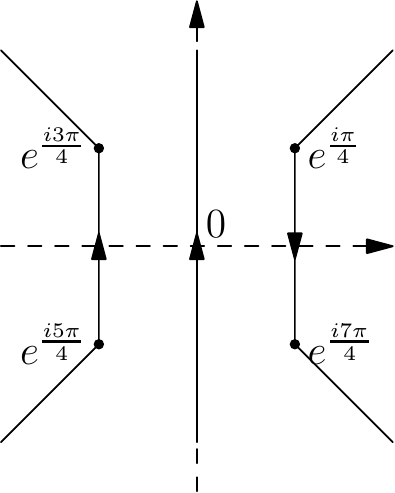}
    \caption{The central part of $\tilde{\Gamma}^{\infty}_a, \tilde{\Sigma}^{\infty}, \tilde{\Gamma}^{\infty}_b$ and $\tilde{\Gamma}_a, \tilde{\Sigma}, \tilde{\Gamma}_b$ around $0$ respectively, where $\tilde{\Gamma}^{\infty} = \tilde{\Gamma}^{\infty}_a \cup \tilde{\Gamma}^{\infty}_b$ (\resp\ $\tilde{\Gamma} = \tilde{\Gamma}_a \cup \tilde{\Gamma}_b$). For $\tilde{\Gamma}^{\infty}_a, \tilde{\Sigma}^{\infty}, \tilde{\Gamma}^{\infty}_b$, the rays go to infinity, while for $\tilde{\Gamma}_a, \tilde{\Sigma}, \tilde{\Gamma}_b$, the contours extend to the ends having magnitude $n^{\frac{1}{4}}$.}
    \label{fig:essential_parts_scaled}
  \end{minipage}
\end{figure}

\begin{thm} \label{thm:Pearcey_kernel}
  Suppose $a$ is large enough and $b$ is determinted by $a$ by \eqref{eq:defn_of_x_0}. Upon the change of scaling
  \begin{equation} \label{eq:rescaling_of_n_1_n_2_x_y_z_w}
    \begin{gathered}
      n_1 = \lfloor n(2 + c_1 n^{-\frac{1}{2}} s) \rfloor, \quad n_2 = \lfloor n(2 + c_1 n^{-\frac{1}{2}} t) \rfloor, \\
      x = cn c^{-c_1 n^{-\frac{1}{2}} s}_2 (1 - c_3 n^{-\frac{3}{4}} u), \quad y = cn c^{-c_1 n^{-\frac{1}{2}} t}_2 (1 - c_3 n^{-\frac{3}{4}} v),
    \end{gathered}
  \end{equation}
  where the constant $c$ is expressed as
  \begin{equation} \label{eq:defn_of_c}
    c = x_0 \exp(\frac{1}{x_0 - a} + \frac{1}{x_0 - b}),
  \end{equation}
  and the constants $c_1, c_2, c_3$ depending on $a, b, x_0$ are defined in \eqref{eq:formulas_of_c_123}, the correlation kernel \eqref{eq:compact_integral_formula_of_K_LUE} of the multiple Laguerre minor process with parameters $\alpha_i$ defined by $a, b$ via \eqref{eq:alpha_i_trapezoidal}, upon conjugation, becomes the extended Pearcey kernel, \ie, for fixed $s, x, t, y$
  \begin{equation}
    \lim_{n \to \infty} \frac{x^{nx_0}((x_0 - b)n)^{-n_1}}{y^{nx_0}((x_0 - b)n)^{-n_2}} K(n_1, x; n_2, y) dy = K^P(s, u; t, v) dv,
  \end{equation}
  with
  \begin{equation} \label{eq:defn_Pearcey_kernel}
    K^P(s, u; t, v) = \frac{1}{(2 \pi i)^2} \oint_{\tilde{\Gamma}^{\infty}} dz \oint_{\tilde{\Sigma}^{\infty}} dw \frac{e^{-\frac{w^4}{4} + \frac{tw^2}{2} - vw}}{e^{-\frac{z^4}{4} + \frac{sz^2}{2} - uz}} \frac{1}{w - z} - \frac{\id_{t > s}}{\sqrt{2 \pi (t - s)}} e^{-\frac{(v - u)^2}{2(t - s)}},
  \end{equation}
  where the contours $\tilde{\Gamma}^{\infty} = \tilde{\Gamma}^{\infty}_a \cup \tilde{\Gamma}^{\infty}_b$ and $\tilde{\Sigma}^{\infty}$ are shown in Figure \ref{fig:essential_parts_scaled}.
\end{thm}
The condition that $a$ is large enough is technical. See the discussion in Remark \ref{rmk:large_enough_a}.

\paragraph{Pearcey distribution in percolation}

In \cite{Johansson03}, Johansson shows that an Airy process appears by taking an appropriate limit $n$ and $m \to \infty$ of a last passage percolation, with $n$ and $m$ as in Figure \ref{fig:upright_path}. The question remained whether the Pearcey process could be found as a limit of percolation problems and also in the directed polymer context. The next theorem answers this question. For the ease of statement, we only consider the one-time distribution of the Pearcey process.
\begin{thm} \label{thm:Pearcey_added}
  Setting $r_1 = n/2$ and letting $n \to \infty$ and the drifts $\kappa_1 = -\sqrt{n}$ and $\kappa_2 = \sqrt{n}$, we obtain a Pearcey distribution for the percolation problem described in Corollary \ref{cor:polymer_added}, namely
  \begin{equation}
    \lim_{n \to \infty} \Prob \left( \text{all $\displaystyle \nu^{(n)}_l \in \frac{E^c}{3n^{\frac{1}{4}}}$, for $1 \leq l \leq n$} \right) = \det(1 - K^P(0, u; 0, v))_{L^2(E)}
  \end{equation}
  where $K^P(0, u; 0, v)$ is the Pearcey kernel defined in \eqref{eq:defn_Pearcey_kernel} evaluated for $s = t = 0$.
\end{thm}
From \eqref{eq:nu_added}, it follows that the paths contributing the most will be in the $\kappa_2 = \sqrt{n}$-region, \ie, in the right-region of the model of size $n/2$. When the number of paths increases, they will tend to fill up that region. When the number of paths exceed $n/2$, more and more paths will be forced in the left-half region where $\kappa_1 = -\sqrt{n}$, a much lower value, thus leading to lower values of the polymer supremum in \eqref{eq:nu_added}. That is to say that around $l = n/2$ the successive increases $\nu^{(n)}_l$ of $\sum^l_{i = 1} \nu^{(n)}_i$, when $l$ increases will be considerably less. That is to say a gap will appear around the values of $\nu^{(n)}_l$, with $l = n/2$. It would be interesting to have such a statement for the O'Connell process, that is to say when the temperature is raised; see \cite{O'Connell12}.

\subsection*{Outline of the paper}

In this paper, the joint distribution function of the eigenvalues in the minor processes are derived in Sections \ref{subsec:GUE_Markov}, \ref{subsec:transition_prob_LUE} and \ref{subsec:transition_prob_JUE} using the idea of corank $1$ projection used by Forrester \etal\ in \cite{Forrester-Nagao11} and \cite{Forrester-Rains05}. The derivation of the determinantal kernel in Sections \ref{subsec:reproducing_kernel_mGUE}, \ref{subsec:reproducing_kernel_LUE} and \ref{subsec:reproducing_kernel_JUE} from the joint distribution function is based on Lemma \ref{lem:BFPS} that was obtained by Borodin, Ferrari, \Prahofer\ and Sasamoto in \cite[Lemma 3.4]{Borodin-Ferrari-Prahofer-Sasamoto07}. Then in Section \ref{sec:Pearcey} we do the asymptotic analysis for a special case of the multiple Laguerre minor process to show the occurrence of the Pearcey process as the limit. We also prove Theorem \ref{thm:Pearcey_added} in Section \ref{sec:Pearcey}. The relation between the minor processes and percolation models and Schur processes via RSK correspondence in Section \ref{sec:Schur_process} follows the argument by Forrester \etal\ in similar models, see \cite{Forrester-Nagao11} and \cite[Appendix A]{Forrester-Rains05}. The proof of Corollary \ref{cor:polymer_added} will be given in the end of Section \ref{sec:Schur_process}. In Appendix \ref{sec:appendix_MOPS} we derive new contour integral formulas for the two kinds of very classical multiple orthogonal polynomials, namely the multiple Laguerre polynomials of the first kind and the Jacobi-\Pineiro\ polynomials, which do not appear in literature according to our limited knowledge. The construction of random matrix models related to the two very classical multiple polynomials brings new types of random matrix models, and the Pearcey kernel is seen in the minor processes for the first time.

\paragraph{Acknowledgment}

The authors thank Ivan Corwin for fruitful discussions at an early stage of this work.

\section{Joint distribution of eigenvalues of minors in GUE with external source} \label{sec:GUE_w_ext}

\subsection{Transition probability of the Markov chain $\lambda^{(n)}$} \label{subsec:GUE_Markov}

The results in this subsection depend on the following technical lemma.
\begin{lem} \label{lem:GUE_w_ext}
  Suppose $H'$ is an $(n-1) \times (n-1)$ fixed Hermitian matrix with distinct eigenvalues $\mu = (\mu_1, \dots, \mu_{n-1})$ in increasing order. Let $H$ be the $n \times n$ random Hermitian matrix defined by
  \begin{itemize}
  \item 
    The upper-left $(n-1)$ minor of $H$ is equal to $H'$.
  \item
    Denote the $(n-1)$ dimensional column vector $\hvec := (h_{1,n},
    \dots, h_{n-1,n})^{\perp}$, the last column of $H$ without the last component. The components of $\hvec$ are independent complex random variables in standard normal distribution, \ie, $p(\Re h_{in}) = \normal(0, \frac{1}{2})$ and $p(\Im h_{in}) = \normal(0, \frac{1}{2})$.
  \item
    $h_{nn}$, the lower-right entry of $H$, is a real random variable independent of $h_{in}$ ($i < n$), and it is in normal distribution: $p(h_{nn}) = \normal(a, 1)$, where $a$ is a real constant.
  \end{itemize}
  Then the distribution of eigenvalues of $H$, denoted by $\lambda = (\lambda_1, \dots, \lambda_n)$ with $\lambda_1 \leq \dots \leq \lambda_n$, satisfies the interlacing property $\mu \preceq \lambda$ given by \eqref{eq:interlacing_defn} and their joint distribution is
  \begin{equation} \label{eq:distr_of_lambda_fix_mu_GUE}
    p_{\mu}(\lambda) = C \Delta_n(\lambda) e^{e^{\sum^n_{i=1} -\frac{\lambda^2_i}{2} + a\lambda_i}} \id_{\mu \preceq \lambda}, \quad \textnormal{where} \quad C = \frac{e^{-a^2/2}}{\sqrt{2\pi}} \frac{e^{\sum^{n-1}_{i=1} \frac{\mu^2_i}{2} - a\mu_i}}{\Delta_{n-1}(\mu)}.
 \end{equation}
\end{lem}
\begin{rmk}
  The distribution of eigenvalues of $H$ depends only on the eigenvalues of $H'$.
\end{rmk}

\begin{proof}
  The proof of Lemma \ref{lem:GUE_w_ext} is analogous to the proof of \cite[Section 4]{Baryshnikov01}. Suppose $U \in \Unitary(n-1)$ is an $(n-1)$ dimensional unitary matrix depending on $H'$, such that $UH'U^{-1} = \diag(\mu)$. Then the $n \times n$ random Hermitian matrix
  \begin{equation}
    \tilde{H} =
    \begin{pmatrix}
      \diag(\mu) & \tilde{\hvec} \\
      \tilde{\hvec}^{\dagger} & h_{nn}
    \end{pmatrix},
    \quad \textnormal{where} \quad \tilde{\hvec} := (\tilde{h}_1, \dots, \tilde{h}_{n-1}) = U\hvec,
  \end{equation}
  is conjugate to $H$, and they have the same eigenvalues and the same characteristic polynomials $p_H(z) = p_{\tilde{H}}(z)$. Note that components of $\tilde{\hvec}$ are also independent random variables in standard complex normal distribution. By direct computation, we find the relations between the characteristic polynomials $p_H(z) = p_{\tilde{H}}(z) = \prod^n_{i=1} (z - \lambda_i)$ of $H$ and $\tilde{H}$ and  $p_{H'}(z) = \prod^{n-1}_{i=1} (z - \mu_i)$ of $H'$: 
  \begin{align} 
     \frac{p_H(z)}{p_{H'}(z)} = {}& z - h_{nn} - \sum^{n-1}_{i=1} \frac{\lvert \tilde{h}_i \rvert^2}{z  - \mu_i}, \label{eq:char_poly_in_GUE_w_ext:1} \\
     \frac{p_H(z)}{p_{H'}(z)} = {}& z - (\sigma_1(\lambda) - \sigma_1(\mu)) - \frac{1}{z} (\sigma_1(\lambda) \sigma_1(\mu) + \sigma_2(\mu) -\sigma_2(\lambda) - \sigma^2_1(\mu)) + \bigO(\frac{1}{z^2}), \label{eq:char_poly_in_GUE_w_ext:2}
 \end{align}
  where $\sigma_1(\lambda) = \sum_i \lambda_i$, $\sigma_2(\lambda) = \sum_{i<j} \lambda_i \lambda_j$ are elementary symmetric polynomials. If $\lvert \tilde{h}_i \rvert^2$ are all positive, the roots of $p_H(\lambda)$, which are eigenvalues of $\tilde{H}$ and $H$, satisfy the strict interlacing condition $\mu \prec \lambda$, which means
  \begin{equation} \label{eq:interlacing_of_mu_and_lambda}
    \lambda_1 < \mu_1 < \lambda_2 < \dots < \mu_{n-1} < \lambda_n.
  \end{equation}
  To see that, we notice the limiting behavior of $p_H(z) / p_{H'}(z)$
  \begin{equation} \label{eq:limiting_behavior_of_p_tiledeH/p_H'}
    \lim_{\lambda \to \pm\infty} \frac{p_H(\lambda)}{p_{H'}(\lambda)} = \pm\infty, \quad \lim_{\lambda \to (\mu_i)_{\pm}} \frac{p_H(\lambda)}{p_{H'}(\lambda)} = \mp\infty,
  \end{equation}
  and find that there is one root in each of the $n$ intervals with $\mu_i, \pm\infty$ as endpoints.
  
  Given $\lvert \tilde{h}_j \rvert^2 \in \realR_+$ and $h_{nn} \in \realR$, there is a unique $\lambda = (\lambda_1, \dots, \lambda_n)$ satisfying \eqref{eq:interlacing_of_mu_and_lambda}. On the other hand, given $\lambda$ satisfying \eqref{eq:interlacing_of_mu_and_lambda}, there is a unique array of $\lvert \tilde{h}_i \rvert^2 \in \realR^+, h_{nn} \in \realR$ such that $\lambda_1, \dots, \lambda_n$ are roots of $p_H(z)$. To see that, we identify the residues at $z = \mu_j$ of the right-hand side of \eqref{eq:char_poly_in_GUE_w_ext:1} with that of the left-hand side, and obtain
  \begin{equation} \label{eq:expression_of_|h|^2}
    \lvert \tilde{h}_i \rvert^2 = -\frac{p_H(\mu_i)}{p'_{H'}(\mu_i)}, \quad i = 1, \dots, n-1;
  \end{equation}
  then we identify the constant terms in the right-hand sides of \eqref{eq:char_poly_in_GUE_w_ext:1} and \eqref{eq:char_poly_in_GUE_w_ext:2}, and obtain
  \begin{equation} \label{eq:expression_of_h_nn}
    h_{nn} = \sigma_1(\lambda) - \sigma_1(\mu).
  \end{equation}
  The argument above also shows that the probability that some eigenvalues of $H$ coincide with eigenvalues of $H'$ is the same as the probability that some $\tilde{h}_i =0$, which is $0$. Thus in order to find the probability density of $\lambda$, we need only to find the probability density of $\lvert \tilde{h}_j \rvert^2 \in \realR_+$ and $h_{nn} \in \realR$ and the Jacobian determinant of the map from $\lambda$ to $\lvert \tilde{h}_j \rvert^2, h_{nn}$, since
  \begin{equation} \label{eq:relations_between_densities}
    p(\lambda) = p(\lvert \tilde{h}_1 \rvert^2, \dots, \lvert \tilde{h}_{n-1} \rvert^2, h_{nn}) \left\lvert \frac{\partial(\lvert \tilde{h}_1 \rvert^2, \dots, \lvert \tilde{h}_{n-1} \rvert^2, h_{nn})}{\partial(\lambda)} \right\rvert.
  \end{equation}
  By \eqref{eq:expression_of_|h|^2} and \eqref{eq:expression_of_h_nn} we obtain
  \begin{equation} \label{eq:derivatives_of_h_lambda_GUE}
    \frac{\partial \lvert \tilde{h}_j \rvert^2}{\partial \lambda_i} = \frac{\lvert \tilde{h}_j \rvert^2}{\lambda_i - \mu_j}, \quad \frac{\partial h_{nn}}{\partial \lambda_i} = 1,
  \end{equation}
  \begin{equation} \label{Jacobian_in_GUE}
    \frac{\partial(\lvert \tilde{h}_1 \rvert^2, \dots, \lvert \tilde{h}_{n-1} \rvert^2, h_{nn})}{\partial(\lambda)} = \prod^{n-1}_{j=1} \lvert \tilde{h}_j \rvert^2 \det(C) = (-1)^{n-1} \frac{\Delta_n(\lambda)}{\Delta_{n-1}(\mu)},
  \end{equation}
  where
  \begin{equation} \label{eq:defn_of_the_matrix_C}
    C =
    \begin{pmatrix}
      \frac{1}{\lambda_1 - \mu_1} & \cdots & \frac{1}{\lambda_n - \mu_1} \\
      \vdots & \vdots & \vdots \\
      \frac{1}{\lambda_1 - \mu_{n-1}} & \cdots & \frac{1}{\lambda_n - \mu_{n-1}} \\
      1 & \cdots & 1 
    \end{pmatrix}.
  \end{equation}
  Here we use the formula \eqref{eq:expression_of_|h|^2} of $\lvert \tilde{h}_i \rvert^2$ and
  \begin{equation}
    \det(C) = (-1)^{\frac{(n-2)(n-1)}{2}} \frac{\Delta_n(\lambda) \Delta_{n-1}(\mu)}{\prod^n_{i=1} \prod^{n-1}_{j=1} (\lambda_i - \mu_j)}
  \end{equation}
  from the Cauchy determinant formula.
  
  On the other hand, by the distributions of $\lvert \tilde{h}_i \rvert^2$ and $h_{nn}$,
  \begin{equation} \label{eq:joint_pdf_of_h_tilde_terms}
    \begin{split}
      p(\lvert \tilde{h}_1 \rvert^2, \dots, \lvert \tilde{h}_{n-1} \rvert^2, h_{nn}) = {}& \frac{1}{\sqrt{2\pi}} e^{-\sum^{n-1}_{i=1} \lvert \tilde{h}_i \rvert^2 - \frac{(h_{nn} - a)^2}{2}} \\
      = {}& \frac{1}{\sqrt{2\pi}} e^{-\frac{1}{2} \sum^n_{i=1} \lambda^2_i + \frac{1}{2} \sum^{n-1}_{i=1} \mu^2_1 + a(\sum^n_{i=1} \lambda_i - \sum^{n-1}_{i=1} \mu_i) - \frac{a^2}{2}},
    \end{split}
  \end{equation}
  where the second identity follows from identifying the $-\sum^{n-1}_{i=1} \lvert \tilde{h}_i \rvert^2$ with the residue of the right-hand side of \eqref{eq:char_poly_in_GUE_w_ext:1} at $\infty$, and calculating the residue from the right-hand side of \eqref{eq:char_poly_in_GUE_w_ext:2}. Thus Lemma \ref{lem:GUE_w_ext} is proved.
\end{proof}

\begin{proof}[Proof of Theorem \ref{thm:Markov} (GUE with external source)]
  To prove that the random variables $\lambda^{(1)}, \dotsc, \lambda^{(n)}$ constitute the Markov chain with transition density given in \eqref{eq:transition_general}, it suffices to show the identity of limiting conditional probability density of $\lambda^{(n)}$ that
  \begin{equation} \label{eq:limiting_identity_for_Markov}
    \lim_{\epsilon \to \infty} P(\lambda^{(n)} \mid \lambda^{(1)} \in I^1_{\epsilon}(\mu^{(1)}), \dots, \lambda^{(n-1)} \in I^{n-1}_{\epsilon}(\mu^{(n-1)})) = P_{\mu^{(n-1)}}(\lambda^{(n)}),
  \end{equation}
  where $P_{\mu^{(n-1)}}(\lambda^{(n)})$ is defined in \eqref{eq:distr_of_lambda_fix_mu_GUE} and
  \begin{equation} \label{eq:defn_of_box}
    I^{j}_{\epsilon}(\mu^{(j)}) = \{ (x_1, \dotsc, x_j) \in \realR^j \mid x_1 \leq \dotsb \leq x_j \text{ and } \lvert x_i - \mu^{(j)}_i \rvert \leq \epsilon \}.
  \end{equation}
  Without loss of generality, here we assume that for all $j = 1, \dotsc, n-1$, $\mu^{(j)}_1, \dotsc, \mu^{(j)}_j$ are distinct real numbers in increasing order. Under the condition $\lambda^{(j)} \in I^{j}_{\epsilon}(\mu^{(j)})$, $j = 1, \dotsc, n-1$, we assume the conditional distribution of $S_{n-1}$, the $n-1$ minor of the matrix $Z_0 + A$, is given by $f_{\epsilon}(H_{n-1}) dH_{n-1}$. The distribution function $f_{\epsilon}(H_{n-1})$ satisfies a property that $f_{\epsilon}(H_{n-1}) = 0$ if the eigenvalues of $H_{n-1}$, ordered increasingly, is not in $I^{n-1}_{\epsilon}(\mu^{(n-1)})$. This is a direct consequence of $\lambda^{(n-1)} \in I^{n-1}_{\epsilon}(\mu^{(n-1)})$.

  Then using the general formula for conditional probability density functions
  \begin{equation}
    f_X(x \mid Y \in A) = \int f_X(x \mid Y \in A \text{ and } Z = z) f_Z(z \mid Y \in A) dz,
  \end{equation}
  we have
  \begin{equation} \label{eq:new_integral_Markov_GUE}
    P(\lambda^{(n)} \mid \lambda^{(j)} \in I^j_{\epsilon}(\mu^{(j)}), \ j = 1, \dotsc, n-1) = \int P(\lambda^{(n)} \mid S_{n-1} = H_{n-1}) f_{\epsilon}(H_{n-1}) dH_{n-1}.
  \end{equation}
  Here $P(\lambda^{(n)} \mid S_{n-1} = H_{n-1})$ denotes the conditional probability density function of $\lambda^{(n)}$ that are the eigenvalues of $S_n$ (in increasing order), the $n$-minor of $Z_0+A$, where the condition is that $S_{n-1}$ is equal to $H_{n-1}$. The value of the density function $P(\lambda^{(n)} \mid S_{n-1} = H_{n-1})$ depends only on the values of the eigenvalues of $H_{n-1}$, as a consequence of Lemma \ref{lem:GUE_w_ext}. From \eqref{eq:distr_of_lambda_fix_mu_GUE}, we know that $P(\lambda^{(n)} \mid S_{n-1} = H_{n-1})$ depends on  the eigenvalues of $H_{n-1}$ in a continuous way, given that the eigenvalues are distinct. Then we know that for any $\lambda^{(n)}$ and $\delta > 0$, there is an $\epsilon$ such that if the eigenvalues of $H_{n-1}$ is given by $\mu \in I^{n-1}_{\epsilon}(\mu^{(n-1)})$, then
  \begin{equation} \label{eq:continuity_of_conditional_prob_GUE}
    \lvert P(\lambda^{(n)} \mid S_{n-1} = H_{n-1}) - P_{\mu^{(n-1)}}(\lambda^{(n)}) \rvert < \delta.
  \end{equation}
  Then \eqref{eq:new_integral_Markov_GUE} and \eqref{eq:continuity_of_conditional_prob_GUE} imply that if $\epsilon$ is small enough for \eqref{eq:continuity_of_conditional_prob_GUE} to hold, then
  \begin{equation}
    \lvert P(\lambda^{(n)} \mid \lambda^{(j)} \in I^j_{\epsilon}(\mu^{(j)}), \ j = 1, \dotsc, n-1) - P_{\mu^{(n-1)}}(\lambda^{(n)}) \rvert < \delta.
  \end{equation}
  Taking the limit $\epsilon \to 0$, we obtain \eqref{eq:limiting_identity_for_Markov}. Thus we prove Theorem \ref{thm:Markov} in the GUE with external source case.
\end{proof}

\subsection{Correlation kernel of the Markov chain $\lambda^{(n)}$} \label{subsec:reproducing_kernel_mGUE}

For the ease of derivation in this subsection, we re-express the formula \eqref{eq:eqivalent_jpdf}, the joint probability density of $\lambda^{(n)}$:
\begin{equation} \label{eq:alternative_jpdf_mGUE}
   P(\lambda^{(1)}, \dots, \lambda^{(N)}) = \frac{e^{-\frac{1}{2} \sum^N_{n=1} a^2_n}}{(2\pi)^{N/2}}\left( \prod^{N-1}_{n=1} \det(\phi_n(\lambda^{(n)}_i, \lambda^{(n+1)}_j))^{n+1}_{i,j=1} \right) \det(\Psi^N_{N-i}(\lambda^{(N)}_j))^N_{i,j=1},
\end{equation}
where the function $\phi_n(\lambda^{(n)}_i, \lambda^{(n + 1)}_j)$ ($i \neq n + 1$) has explicit formula
  \begin{equation} \label{phi_n:GUE}
    \phi_n(x, y) = e^{(a_n - a_{n + 1})x} \id_{x < y},
  \end{equation}
  and
\begin{equation} \label{eq:Psi_GUE}
  \Psi^N_i(x) = e^{-\frac{x^2}{2} + a_N x} P_{(a_N, a_{N-1}, \dots, a_{N-i+1})}(x) = \frac{e^{a_N x}}{\sqrt{2\pi} i} \int_{C+i\realR \uparrow} e^{\frac{s^2}{2} - xs} \prod^j_{k=1} (s- a_{N-k+1}) ds,
\end{equation}
and $P_{(a_N, a_{N-1}, \dots, a_{N-i+1})}(x)$ is the $i$-th degree (monic) multiple Hermite polynomial of type II (see Appendix \ref{sec:appendix_MOPS}).

In the proof of Theorem \ref{thm:kernel}\ref{enu:thm:kernel_GUE}, we first express the correlation kernel in terms of multiple Hermitian polynomials, and then write it in the form of double contour integral.

\subsubsection{The correlation kernel and multiple Hermite polynomials}

The expression of the correlation kernel in terms of multiple Hermite polynomials, and even the fact that the Markov chain $\lambda^{(n)}$ is determinantal such that its properties are captured by the correlation kernel, is based on the following general lemma of Borodin, Ferrari, \Prahofer\ and Sasamoto \cite{Borodin-Ferrari-Prahofer-Sasamoto07}.

Below we state the lemma. First we fix $I$ as a subset of $\realR$ (in our paper, $I$ is taken as $(-\infty, \infty)$, $[0, \infty)$ or $[0, 1]$). Suppose the functions $a(x, y), b(x, y), c(x)$ are defined on $I \times I$ or $I$. Then the operator $*$ used in the lemma is defined as
\begin{equation} \label{eq:binary_star}
  a*b(x,y) := \int_I a(x,z)b(z,y) dz, \quad  a*c(x) := \int_I a(x,z)c(z) dz, \quad c*a(x) = \int_I c(z)a(z,x) dz.
\end{equation}

\begin{lem}[{\cite[Lemma 3.4]{Borodin-Ferrari-Prahofer-Sasamoto07}}] \label{lem:BFPS} 
  Suppose we have a signed measure on $I \times I^2 \times \dotsb \times I^N$ given in the form
  \begin{equation}
    \frac{1}{Z_N} \prod^{N-1}_{n=1} \det(\phi_n(\lambda^{(n)}_i, \lambda^{(n+1)}_j))^{n+1}_{i,j = 1} \det(\Psi^N_{N-i}(\lambda^{(N)}_j))^N_{i,j = 1}, \quad (\lambda^{(n)}_1, \dotsc, \lambda^{(n)}_n) \in I^n,
  \end{equation}
  where $\phi_n(\lambda^{(n)}_i, \lambda^{(n+1)}_j)$ stands for the value of the two-variable function $\phi_n(x,y)$ at $x = \lambda^{(n)}_i, y = \lambda^{(n+1)}_j$ if $i \leq n$ and otherwise the value of the one-variable function $\phi_n(\lambda^{(n)}_{n+1}, x)$, and $Z_N$ is a normalization constant. If $Z_N \neq 0$, then the correlation functions are determinantal.

    To write down the kernel we let
  \begin{equation} \label{eq:defn_of_phi^n1n2}
    \phi^{(n_1, n_2)}(x,y) :=
    \begin{cases}
      \phi_{n_1}*(\phi_{n_1+1}*(\dots *\phi_{n_2-1}))(x,y) & n_1 < n_2, \\
      0 & n_1 \geq n_2,
    \end{cases}
  \end{equation}
  and for $1 \leq n < N$,
  \begin{equation} \label{eq:defn_of_Psi^n_j}
    \Psi^n_{n-j}(x) := \phi^{(n,N)}*\Psi^N_{N-j}(x), \quad j = 1, 2, \dots, N.
  \end{equation}
  Set $\phi_0(\lambda^{(0)}_1,x) := 1$. Suppose the functions defined on $I$
  \begin{equation} \label{eq:basis_of_V_n}
    \phi_0*\phi^{(1,n)}(\lambda^{(0)}_1,x),\ \dots,\ (\phi_{n-2}*\phi^{(n-1,n)})(\lambda^{(n-2)}_{n-1},x),\ \phi_{n-1}(\lambda^{(n-1)}_n, x),
 \end{equation}
 are well defined, then they are linearly independent and generate an $n$-dimensional space $V_n$. Define a set of functions $\Phi^n_j(x)$ ($j = 0, \dots, n-1$) spanning $V_n$ by the orthogonality relations
  \begin{equation} \label{eq:orthonormality_of_Psi_and_Phi}
    \int_I \Phi^n_i(x)\Psi^n_j(x) = \delta_{ij}, \quad \text{for $0 \leq i,j \leq n-1$.}
  \end{equation}

  Under Assumption A:
  \begin{equation} \label{eq:formula_of_Assumption_A}
    \phi_n(\lambda^{(n)}_{n+1}, x) = c_n \Phi^{(n+1)}_0(x), \quad n = 1, \dots, N-1
  \end{equation}
  for some $c_n \neq 0$, the kernel takes the simple form
  \begin{equation} \label{eq:general_formula_of_kernel}
    K(n_1, x; n_2, y) := -\phi^{(n_1, n_2)}(x,y) + \sum^{n_2}_{k=1}\Psi^{n_1}_{n_1-k}(x) \Phi^{n_2}_{n_2-k}(y).
  \end{equation}
\end{lem}
\begin{rmk}
  Our version of the lemma is slightly different from that in \cite[Lemma 3.4]{Borodin-Ferrari-Prahofer-Sasamoto07}, in the sense that they consider discrete random variables but we consider continuous ones supported on $I$. The proof in their paper can be used for Lemma \ref{lem:BFPS} with little change.
\end{rmk}

We first compute the correlation kernel under the assumption that $a_1 > a_2 > \dots > a_N$ by the application of Lemma \ref{lem:BFPS} with $I = (-\infty, \infty)$. We need the following identity that if $\phi_n(x,y)$ is given by \eqref{phi_n:GUE}, $\phi_n(\lambda^{(n)}_{n+1}, x) = 1$ and $\Psi^N_{N-j}(x)$ is given by \eqref{eq:Psi_GUE}, then
\begin{equation} \label{eq:evaluation_of_phi^n1n2}
  \phi^{(n_1, n_2)}(x,y) = \frac{e^{a_{n_1}x - a_{n_2}y}}{2\pi i} \oint_{\Gamma_a} \frac{e^{(y-x)z}}{\prod^{n_2}_{j = n_1+1} (z - a_j)} dz \id_{x<y} \id_{n_1 < n_2},
\end{equation}
where $\Gamma_a$ is a large enough contour enclosing all poles $a_{n_1+1}, \dots, a_{n_2}$, and \footnote{In the evaluation of $\Psi^{(n)}_0(x)$ in \eqref{eq:evaluation_of_Psi^n_j} and below, we take the notational convention that $\prod^0_{j=1}( \cdot ) = 1$.}
\begin{equation} \label{eq:evaluation_of_Psi^n_j}
  \Psi^n_j(x) = 
   \begin{cases} 
      {\displaystyle
        \begin{gathered}[b]
          \frac{e^{a_n x}}{\sqrt{2\pi} i} \int_{C+i\realR \uparrow} e^{\frac{s^2}{2} - xs} \prod^j_{k=1} (s- a_{n-k+1}) ds = \\
          e^{-\frac{x^2}{2} + a_n x} P_{(a_n, a_{n-1}, \dots, a_{n-j+1})}(x),
        \end{gathered}
        } & j = 0, 1, \dots, n-1, \\
      {\displaystyle \frac{e^{a_n x}}{\sqrt{2\pi} i} \int_{C+i\realR \uparrow} e^{\frac{s^2}{2} - xs} \frac{ds}{\prod^{-j}_{k=1} (s-a_{n+k})},} & j = -1, -2, \dots, n-N,
    \end{cases}
\end{equation}
where $C$ is a large enough real number such that $C + i\realR \uparrow$ lies to the right all possible poles of the integrand, and 
\begin{equation} \label{eq:evaluation_of_phi*phi}
  \phi_j*\phi^{(j+1,n)}(\lambda^{(j)}_{j+1}, x) = \left( \prod^n_{k=j+2} \frac{1}{a_{j+1} - a_k} \right) e^{(a_{j+1}-a_n)x}, \quad j = 0, 1, \dots, n-2.
\end{equation}

Note that from \eqref{eq:basis_of_V_n} and \eqref{eq:evaluation_of_phi*phi}, the vector space $V_n$ is spanned by $e^{(a_1 - a_n)x}, e^{(a_2 - a_n)x}, \dots, e^{(a_n - a_n)x}$, and $\Phi^n_j(x)$ are vectors in $V_n$ defined by the orthonormality \eqref{eq:orthonormality_of_Psi_and_Phi}. Since $\Psi^n_j(x)$ are defined by multiple Hermite polynomials of type II, $\Phi^n_j(x)$ are hence defined by multiple Hermite polynomials of type I. Using Proposition \ref{prop:mult_Hermite} in Appendix \ref{sec:appendix_MOPS}, we have
\begin{equation} \label{eq:expression_of_Phi^n_j}
  \Phi^n_j(x) = e^{\frac{x^2}{2}-a_n(x)}Q_{(a_n, a_{n-2}, \dots, a_{n-j})}(x) = \frac{e^{-a_n x}}{\sqrt{2\pi}2\pi i} \oint_{\Gamma_a} e^{-\frac{t^2}{2} + xt} \frac{dt}{\prod^j_{k=0} (t-a_{n-k})},
\end{equation}
where $Q_{(a_n, a_{n-1}, \dotsc, a_{n-j})}(x)$ is the multiple Hermite polynomial of type I.

It is clear that the Assumption A is satisfied, as both functions $\phi_n(\lambda^{(n)}_{n+1}, x)$ and $\Phi^{(n+1)}_0(x)$ in \eqref{eq:formula_of_Assumption_A} are constants. Therefore the correlation kernel is given by  \eqref{eq:general_formula_of_kernel} with $\phi^{(n_1, n_2)}(x,y)$ expressed in \eqref{eq:evaluation_of_phi^n1n2}, $\Psi^{n_1}_{n_1-k}(x)$ expressed in \eqref{eq:evaluation_of_Psi^n_j} and $\Phi^{n_2}_{n_2-k}(y)$ expressed in \eqref{eq:expression_of_Phi^n_j}.

Note that the correlation kernel $K(n_1, x; n_2, y)$ is  analytic in $a_1, \dots, a_N$, by \eqref{eq:general_formula_of_kernel}. By analytic continuation we can remove the restriction $a_1 > a_2 > \dots > a_N$.

\begin{proof}[Proof of \eqref{eq:evaluation_of_phi^n1n2}, \eqref{eq:evaluation_of_Psi^n_j} and \eqref{eq:evaluation_of_phi*phi}]
  We prove \eqref{eq:evaluation_of_phi^n1n2} by induction. If $n_1 = n_2 - 1$, it is clear that \eqref{eq:evaluation_of_phi^n1n2} holds. If it holds for $n_1 = n_2 - k$, $k \geq 1$, then for $n_1 = n_2 - k - 1$,

 \begin{equation}
    \begin{split}
       \phi^{(n_1, n_2)}(x,y)  = {}& \phi_{n_1} * \phi^{(n_1+1, n_2)}(x,y) \\
       = {}& \int^{\infty}_{-\infty} e^{(a_{n_1} - a_{n_1+1})x} \frac{e^{a_{n_1+1}w - a_{n_2}y}}{2\pi i} \oint_{\Gamma_a} \frac{e^{(y-w)z}}{\prod^{n_2}_{n_1+2} (z - a_j)} dz \id_{x<w<y} dw \\
      = {}& \frac{e^{a_{n_1}x - a_{n_2}y}}{2\pi i} \oint_{\Gamma_a} \frac{e^{yz - a_{n_1+1}x}}{\prod^{n_2}_{j=n_1+2}(z-a_j)} \int^y_x e^{(a_{n_1+1}-z)w} dw dz \id_{x<y} \\
      = {}& \frac{e^{a_{n_1}x - a_{n_2}y}}{2\pi i} \oint_{\Gamma_a} \frac{e^{(y-x)z} - e^{(y-x)a_{n_1+1}}}{\prod^{n_2}_{n_1+1} (z - a_j)} dz \id_{x<y} \\
      = {}& \frac{e^{a_{n_1}x - a_{n_2}y}}{2\pi i} \left( \oint_{\Gamma_a} \frac{e^{(y-x)z}}{\prod^{n_2}_{n_1+1} (z - a_j)} dz + e^{(y-x)a_{n_1+1}} \oint_{\Gamma_a} \frac{1}{\prod^{n_2}_{n_1+1} (z - a_j)} dz \right) \id_{x<y} \\
      = {}& \frac{e^{a_{n_1}x - a_{n_2}y}}{2\pi i} \oint_{\Gamma_a} \frac{e^{(y-x)z}}{\prod^{n_2}_{n_1+1} (z - a_j)} dz \id_{x<y} \id_{n_1 < n_2},
    \end{split}
  \end{equation}
  where in the last step we use the vanishing of the contour integral $\oint_{\Gamma_a} \frac{1}{\prod^{n_2}_{n_1+1} (z - a_j)} dz$, which can be seen by deforming $\Gamma_a$ to a contour about $\infty$, and using $n_2 - n_1 \geq 2$.
 
  The proof of \eqref{eq:evaluation_of_Psi^n_j} is based on the contour integral formula \eqref{eq:mult_Hermite_type_II} of multiple Hermite polynomials of type II in Appendix \ref{sec:appendix_MOPS}. From \eqref{eq:defn_of_Psi^n_j}, \eqref{eq:evaluation_of_phi^n1n2}, \eqref{eq:Psi_GUE} and \eqref{eq:mult_Hermite_type_II}, we have, where $C + i\realR \uparrow$ is to the right of $\Gamma_a$, that
  \begin{equation} \label{eq:proof_of_Psi_GUE}
    \begin{split}
      \Psi^n_{n-j}(x) = {}& \int^{\infty}_x \frac{e^{a_nx - a_Ny}}{2\pi i} \oint_{\Gamma_a} \frac{e^{(y-x)z}}{\prod^N_{k=n+1}(z-a_k)} dz \frac{e^{-\frac{y^2}{2} + a_Ny}}{\sqrt{2\pi}i} \int_{C +  i\realR \uparrow} e^{\frac{(s-y)^2}{2}} \prod^N_{l=j+1} (s-a_l) ds dy \\
      = {}& \frac{e^{a_nx}}{-(2\pi)^{3/2}} \int_{C +  i\realR \uparrow} ds \oint_{\Gamma_a} dz \frac{e^{-xz}}{\prod^N_{k=n+1}(z-a_k)} e^{\frac{s^2}{2}} \prod^N_{l=j+1}(s-a_l) \int^{\infty}_x e^{(z-s)y} dy \\
      = {}& \frac{e^{a_nx}}{-(2\pi)^{3/2}} \int_{C + i\realR \uparrow} ds e^{\frac{s^2}{2} - sx} \prod^N_{l=j+1}(s-a_l) \oint_{\Gamma_a} dz \frac{1}{(s-z) \prod^N_{k=n+1}(z-a_k)} \\
      = {}& \frac{e^{a_nx}}{\sqrt{2\pi}i} \int_{C + i \realR \uparrow} ds e^{\frac{s^2}{2} - sx} \frac{\prod^N_{l=j+1}(s-a_l)}{\prod^N_{k=n+1}(s-a_k)},
    \end{split}
  \end{equation}
  and from \eqref{eq:proof_of_Psi_GUE}, \eqref{eq:evaluation_of_Psi^n_j} is easy to obtain.

  The proof of \eqref{eq:evaluation_of_phi*phi} is straightforward, at least if $a_1, \dotsc, a_n$ are in strictly descending order. Using $\phi_n(\lambda^{(n)}_{n+1}, x) = 1$ and \eqref{eq:evaluation_of_phi^n1n2}, we conclude that
  \begin{equation}
    \begin{split}
      \phi_j*\phi^{(j+1,n)}(\lambda^{(j)}_{j+1}, x) = {}& \int^x_{-\infty} \frac{e^{a_{j+1}y - a_nx}}{2\pi i} \oint_{\Gamma_a} \frac{e^{(x-y)z}}{\prod^n_{k=j+2}(z-a_k)} dz dy \\
      = {}& \frac{e^{-a_nx}}{2\pi i} \oint_{\Gamma_a} dz \frac{e^{xz}}{\prod^n_{k=j+2}(z-a_k)} \int^x_{-\infty} e^{(a_{j+1} - z)y} dy \\
      = {}& \frac{e^{(a_{j+1} - a_n)x}}{2\pi i} \oint_{\Gamma_a} \frac{-1}{\prod^n_{k=j+1}(z-a_k)} dz \\
      = {}& \left( \prod^n_{k=j+2} \frac{1}{a_{j+1} - a_k} \right) e^{(a_{j+1}-a_n)x}.
    \end{split}
  \end{equation}
In the above derivation we assumed that the contour $\Gamma_a$ encloses $a_{j+2}, a_{j+3}, \dots, a_n$ but lies to the left of $a_{j+1}$, which enables us to do the $y$ integration and finally deform $\Gamma_a$ to a large contour about $\infty$, picking up the residue just at $z = a_{j+1}$. The argument above relies on the assumption that $a_{j+1} > a_{j+2} > \dots > a_n$, but it is clear that \eqref{eq:evaluation_of_phi*phi} holds when the assumption is removed, due to the analytic continuation.
\end{proof}

\subsubsection{Double contour integral formula of the correlation kernel}

From \eqref{eq:proof_of_Psi_GUE} and \eqref{eq:expression_of_Phi^n_j},
\begin{multline} \label{eq:expansion_of_sum_Psi*Phi_GUE}
  \sum^{n_2}_{j=1}\Psi^{n_1}_{n_1-j}(x) \Phi^{n_2}_{n_2-j}(y) = \\
  \frac{e^{a_{n_1}x - a_{n_2}y}}{(2\pi i)^2} \int_{C + i\realR \uparrow} ds \oint_{\Gamma_a} dt e^{\frac{s^2}{2} - \frac{t^2}{2} -xs + yt} \frac{\prod^N_{k = n_2+1} (t - a_k)}{\prod^N_{l = n_1+1} (s - a_l)} \sum^{n_2}_{j=1} \frac{\prod^N_{k = j+1} (s - a_k)}{\prod^N_{l = j} (t - a_l)}.
\end{multline}
Using the identity that for any $j < N$
\begin{equation} \label{eq:telescope_1}
  \begin{split}
    (s - t) \frac{\prod^N_{k=j+1}(s-a_k)}{\prod^N_{l=j}(t-a_l)} = {}& (s - a_j) \frac{\prod^N_{k=j+1}(s-a_k)}{\prod^N_{l=j}(t-a_l)} - (t - a_j) \frac{\prod^N_{k=j+1}(s-a_k)}{\prod^N_{l=j}(t-a_l)} \\
    = {}& \prod^N_{k=j} \left( \frac{s - a_k}{t - a_k} \right) - \prod^N_{k=j+1} \left( \frac{s - a_k}{t - a_k} \right)
  \end{split}
\end{equation}
and the telescoping trick, we find
\begin{equation} \label{eq:telescope_identity}
  \sum^{n_2}_{j=1} \frac{\prod^N_{k = j+1} (s - a_k)}{\prod^N_{l = j} (t - a_l)} = \frac{1}{s-t} \prod^N_{k=1} \left( \frac{s - a_k}{t - a_k} \right) - \frac{1}{s-t} \prod^N_{k=n_2+1} \left( \frac{s - a_k}{t - a_k} \right).
\end{equation}
Substituting \eqref{eq:telescope_identity} into \eqref{eq:expansion_of_sum_Psi*Phi_GUE}, we obtain, letting the contour $C + i\realR \uparrow$ of $s$ be to the right of $\Gamma_a$,
\begin{equation} \label{eq:kernel_of_GUE_ext_2nd_part}
  \sum^{n_2}_{k=1}\Psi^{n_1}_{n_1-k}(x) \Phi^{n_2}_{n_2-k}(y) = \frac{e^{a_{n_1}x - a_{n_2}y}}{(2\pi i)^2} \int_{C + i\realR \uparrow} ds \oint_{\Gamma_a} dt e^{\frac{s^2}{2} - \frac{t^2}{2} -xs + yt} \frac{\prod^{n_1}_{k = 1} (s - a_k)}{\prod^{n_2}_{l = 1} (t - a_l)} \frac{1}{s-t}.
\end{equation}
Hence we prove Theorem \ref{thm:kernel}\ref{enu:thm:kernel_GUE} from \eqref{eq:evaluation_of_phi^n1n2}, \eqref{eq:kernel_of_GUE_ext_2nd_part} and \eqref{eq:general_formula_of_kernel} in Lemma \ref{lem:BFPS}, and the fact that we may conjugate any Fredholm determinant without changing its value.
  
\section{Joint distribution of eigenvalues of minors in multiple
LUE} \label{sec:multiple_LUE}

\subsection{Transition probability of the Markov chain $\lambda^{(n)}$} \label{subsec:transition_prob_LUE}

Analogous to Lemma \ref{lem:GUE_w_ext} in Section \ref{sec:GUE_w_ext}, we need the following technical lemma in this subsection. In the statement of the lemma and later in this section, we denote $l_1, l_2, \dots$ be nondecreasing integers such that $l_n \geq n$, and denote the nonnegative integers $\alpha_n = l_n - n$.

\begin{lem} \label{lem:LUE_w_ext}
  Suppose $X'$ is an $l_{n-1} \times (n-1)$ fixed rectangular matrix with distinct singular values $\mu = (\sqrt{\mu_1}, \dots, \sqrt{\mu_{n-1}})$ where $\mu = (\mu_1, \dots, \mu_{n-1})$ is assumed to be $0 < \mu_1 < \dots < \mu_{n-1}$. Let $X$ be the $l_n \times n$ rectangular random matrix defined by
  \begin{itemize}
  \item 
    The upper-left $l_{n-1} \times (n-1)$ block of $X$ is equal to $X'$.
  \item
    Denote the $l_n$ dimensional column vector $\xvec := (x_{1,n}, \dots, x_{l_n,n})^{\perp}$, the last column of $X$. Then components of $\xvec$ are independent complex random variables in standard normal distribution, \ie, $p(\Re x_{i,n}) = \normal(0, 1/2)$ and $p(\Im x_{i,n}) = \normal(0, 1/2)$.
  \item
    All entries in the lower-left $(l_n - l_{n-1}) \times (n-1)$ block of $X$ are $0$.
    \end{itemize}
  Denote the singular values of $X$ by $\sqrt{\lambda_1}, \dots, \sqrt{\lambda_n}$ where $\lambda = (\lambda_1, \dots, \lambda_n)$ is assumed to be $0 \leq \lambda_1 \leq \dots \leq \lambda_n$. Then the joint distribution of $\lambda$ is
  \begin{equation} \label{eq:prob_distr_LUE_lemma}
    p_{\mu}(\lambda) = C^{-1} \Delta_n(\lambda) \left( \prod^n_{i=1} \lambda^{\alpha_n}_i e^{-\lambda_i} \right) \id_{\mu \preceq \lambda}, \quad \textnormal{where} \quad C = \alpha_n ! \Delta_{n-1}(\mu) \prod^{n-1}_{i=1} \mu^{\alpha_i+1}_i e^{-\mu_i}.
 \end{equation}
\end{lem}

\begin{proof}
  For the convenience of the proof, we denote $S = X^*X$ and $S' = (X')^*X'$, and note that the eigenvalues of $S$ and $S'$ are $\lambda_1, \dots, \lambda_n$ and $\mu_1, \dots, \mu_{n-1}$ respectively.
  
  Performing the singular value decomposition, we find the unitary matrices $U \in \U(l_{n-1})$ and $V \in \U(n-1)$ such that
  \begin{equation} \label{eq:singular_value_decomposition_of_X_m-1}
    X' = U
    \begin{pmatrix}
      \sqrt{\mu_1} & 0 & \cdots & 0 \\
      0 & \sqrt{\mu_2} & & 0 \\
      \vdots & & \ddots & \vdots \\
      0 & \cdots & \cdots & \sqrt{\mu_{n-1}} \\
      0 & \cdots & \cdots & 0 \\
      \vdots & & & \vdots \\
      0 & \cdots & \cdots & 0
    \end{pmatrix}
    V^*.
  \end{equation}
  Let $\tilde{U} = U \oplus I_{(l_n - l_{n-1}) \times (l_n - l_{n-1})} \in \U(l_n)$ and $\tilde{V} = V \oplus I_{1 \times 1} \in \U(n)$, we then have
  \begin{equation}
    X = \tilde{U} \tilde{X} \tilde{V}^*,
  \end{equation}
  where
  \begin{equation} \label{eq:partially_diagonalized_tilde_X}
    \tilde{X} =
    \begin{pmatrix}
      \sqrt{\mu_1} & 0 & \cdots & 0 & \tilde{x}_{1,n} \\
      0 & \sqrt{\mu_2} & & 0 & \vdots \\
      \vdots & & \ddots & \vdots & \vdots \\
      0 & \cdots & \cdots & \sqrt{\mu_{n-1}} & \tilde{x}_{n-1,n} \\
      0 & \cdots & \cdots & 0 & \tilde{x}_{n,n} \\
      \vdots & & & \vdots & \vdots \\
      0 & \cdots & \cdots & 0 & \tilde{x}_{l_n, n}
    \end{pmatrix},
  \end{equation}
  and in the last column of $\tilde{X}$ all components are in \iid\ standard complex normal distribution. Below we use the fact that the eigenvalue distribution of $S$ is the same of that of
  \begin{equation}
    \tilde{S} := \tilde{X}^* \tilde{X} =
    \begin{pmatrix}
      \mu_1 & \cdots & 0 & \sqrt{\mu_1}\tilde{x}_{1,n} \\
      \vdots & \ddots & \vdots & \vdots \\
      0 & \cdots & \mu_{n-1} & \sqrt{\mu_{n-1}}\tilde{x}_{n-1,n} \\
      \sqrt{\mu_1}\overline{\tilde{x}_{1,n}} & \cdots & \sqrt{\mu_{n-1}}\overline{\tilde{x}_{n-1,n}} & \sum^{l_n}_{k=n} \lvert \tilde{x}_{k,n} \rvert^2
    \end{pmatrix}.
  \end{equation}
  Now we consider the relation between the characteristic polynomials $p_S(z)$ of $S$ (that is the same as $p_{\tilde{S}}(z)$ of $\tilde{S}$) and $p_{S'}(z)$ of $S'$, analogous to \eqref{eq:char_poly_in_GUE_w_ext:1} and \eqref{eq:char_poly_in_GUE_w_ext:2}. By direct calculation, we find that 
  \begin{align} 
    \frac{p_S(z)}{p_{S'}(z)} = {}& z - \xi_n - \sum^{n-1}_{i=1} \frac{z}{z  - \mu_i}\xi_i, \label{eq:char_poly_in_LUE_w_ext:1} \\
    \frac{p_S(z)}{p_{S'}(z)} = {}& z - (\sigma_1(\lambda) - \sigma_1(\mu)) - \frac{1}{z} (\sigma_1(\lambda) \sigma_1(\mu) + \sigma_2(\mu) -\sigma_2(\lambda) - \sigma^2_1(\mu)) + \bigO(\frac{1}{z^2}), \label{eq:char_poly_in_LUE_w_ext:2}
  \end{align}
  where $\sigma_1, \sigma_2$ are defined the same as in \eqref{eq:char_poly_in_GUE_w_ext:2}, and in \eqref{eq:char_poly_in_LUE_w_ext:1}
  \begin{equation} \label{eq:defn_of_xi_j_LUE}
    \xi_i =
    \begin{cases}
      \lvert \tilde{x}_{i,n} \rvert^2 & i = 1, \dots, n-1, \\
      \sum^{l_n}_{l=n} \lvert \tilde{x}_{l,n} \rvert^2 & i=n.
    \end{cases}
  \end{equation}
  
  Noticing the limiting behavior of $p_S(z) / p_{S'}(z)$ analogous to \eqref{eq:limiting_behavior_of_p_tiledeH/p_H'} yields the strict interlacing condition $\mu \prec \lambda$ when $\xi_1, \dots, \xi_{n-1}$ are all positive, that is,
  \begin{equation} \label{eq:interlacing_of_mu_and_lambda_LUE}
    0 < \lambda_1 < \mu_1 < \lambda_2 < \dots < \mu_{n-1} < \lambda_n.
  \end{equation}
  On the other hand, given $\lambda$ satisfying \eqref{eq:interlacing_of_mu_and_lambda_LUE}, there is a unique array of $\xi_1, \dots, \xi_{n-1} \in \realR_+, \xi_n \in \realR$. To see that, by the calculations of residues like in \eqref{eq:expression_of_|h|^2} and \eqref{eq:expression_of_h_nn}, we have 
  \begin{align}
    \xi_i = {}& -\frac{p_S(\mu_i)}{\mu_i p'_{S'}(\mu_i)} = - \res_{z = \mu_i} \frac{\prod^{n}_{k=1} (z-\lambda_k)}{z \prod^{n-1}_{k=1} (z-\mu_k)}, \quad i = 1, \dots, n-1, \label{eq:expression_of_xi_j} \\
    \xi_n = {}& \sigma_1(\lambda) - \sigma_1(\mu) - \sum^{n-1}_{i=1} \xi_i 
    \begin{aligned}[t]
      {}= {}& \res_{z = \infty} \frac{\prod^{n}_{k=1} (z-\lambda_k)}{z \prod^{n-1}_{k=1} (z-\mu_k)} + \sum^{n-1}_{i=1} \res_{z = \mu_i} \frac{\prod^{n}_{k=1} (z-\lambda_k)}{z \prod^{n-1}_{k=1} (z-\mu_k)} \\
      = {}&  -\res_{z = 0} \frac{\prod^{n}_{k=1} (z-\lambda_k)}{z \prod^{n-1}_{k=1} (z-\mu_k)} = \frac{\prod^n_{i=1} \lambda_i}{\prod^{n-1}_{i=1} \mu_i}.
    \end{aligned} \label{eq:expression_of_xi_n}
 \end{align}
  Then we can express $p(\lambda)$ analogous to \eqref{eq:relations_between_densities}
  \begin{equation} \label{eq:relations_between_densities_LUE}
    p(\lambda) = p(\xi_1, \dots, \xi_n) \left\lvert \frac{\partial(\xi_1, \dots, \xi_n)}{\partial(\lambda)} \right\rvert.
  \end{equation}
  Like \eqref{eq:derivatives_of_h_lambda_GUE},
  \begin{equation}
    \frac{\partial \xi_j}{\partial \lambda_i} = \frac{\xi_j}{\lambda_i - \mu_j}, \quad \text{for $j = 1, \dots, n-1$, and} \quad \frac{\partial \xi_n}{\partial \lambda_i} = 1 - \sum^{n-1}_{j=1} \frac{\partial \xi_j}{\partial \lambda_i},
  \end{equation}
  and like \eqref{Jacobian_in_GUE},
  \begin{equation} \label{Jacobian_in_LUE}
    \frac{\partial (\xi_1, \dots, \xi_n)}{\partial (\lambda_1, \dots, \lambda_n)} = \prod^{n-1}_{j=1} \xi_j \det(C) = (-1)^{n-1} \frac{\Delta_n(\lambda)}{\Delta_{n-1}(\mu) \prod^{n-1}_{j=1} \mu_j},
  \end{equation}
  where the $n \times n$ matrix $C$ is defined in \eqref{eq:defn_of_the_matrix_C}, and in the last identity of \eqref{Jacobian_in_LUE} we use the formula \eqref{eq:expression_of_xi_j}. 

  On the other hand, by the definition \eqref{eq:defn_of_xi_j_LUE} of $\xi_1, \dots, \xi_n$, they are independent, $2\xi_n$ is $\chi^2_{2(\alpha_n+1)}$ in distribution and $2\xi_1, \dots, 2\xi_{n-1}$ are all $\chi^2_2$ in distribution, and so by expressions \eqref{eq:expression_of_xi_j} and \eqref{eq:expression_of_xi_n}, (note $\alpha_n = l_n - n$)
  \begin{equation} \label{eq:joint_pdf_of_xi_LUE}
    p(\xi_1, \dots, \xi_n) = e^{-\sum^{n-1}_{i=1} \xi_i} \frac{\xi^{\alpha_n}_n}{\alpha_n!}e^{-\xi_n} = \frac{1}{\alpha_n!} \left( \frac{\prod^n_{j=1} \lambda_j}{\prod^{n-1}_{k=1} \mu_k} \right)^{\alpha_n} e^{ \left( \sum^{n-1}_{k=1} \mu_k - \sum^n_{j=1} \lambda_j \right)}. 
  \end{equation}
  Substituting \eqref{Jacobian_in_LUE} and \eqref{eq:joint_pdf_of_xi_LUE} in to \eqref{eq:relations_between_densities_LUE}, we prove Lemma \ref{lem:LUE_w_ext}.
\end{proof}

\begin{proof}[Proof of Theorem \ref{thm:Markov} (multiple Laguerre case)]
  Recall in the proof in the GUE with external source case, the Markov property of the process relies on the fact that the distribution of the eigenvalues of $H$ in Lemma \ref{lem:GUE_w_ext} depends on only the eigenvalues of $H'$, but not its eigenvectors. In Lemma \ref{lem:LUE_w_ext}, we also have that the distribution of the eigenvalue of $S = X^* X$ depends only on the eigenvalues of $S' = (X')^* X'$, but not its eigenvectors. 

  Like in the proof in the GUE with external source case, suppose for all $j$, $\mu^{(j)}_1, \dotsc, \mu^{(j)}_j$ are distinct real numbers in increasing order, and define $I^j_{\epsilon}(\mu^{(j)})$ as in \eqref{eq:defn_of_box}. Then analogous to \eqref{eq:new_integral_Markov_GUE},
  \begin{equation}
    P(\lambda^{(n)} \mid \lambda^{(j)} \in I^j_{\epsilon}(\mu^{(j)}), \ j = 1, \dotsc, n-1) = \int P(\lambda^{(n)} \mid S_{n-1} = H_{n-1}) f_{\epsilon}(H_{n-1}) dH_{n-1},
  \end{equation}
  where the distribution function $f_{\epsilon}(H_{n-1})$ is defined in the same way as that in the proof in the GUE with external source case. Then after the same argument, we derive the limiting identity analogous to \eqref{eq:limiting_identity_for_Markov}
  \begin{equation} \label{eq:multi_conditional_LUE}
    P(\lambda^{(n)} \mid \lambda^{(j)} = \mu^{(j)}, \ j = 1, \dotsc, n-1) = P_{\mu^{(n-1)}}(\lambda^{(n)}),
  \end{equation}
  and we prove Theorem \ref{thm:Markov} in the multiple Laguerre case.
\end{proof}

\subsection{Correlation kernel of the Markov chain $\lambda^{(n)}$} \label{subsec:reproducing_kernel_LUE}

We give the proof of Theorem \ref{thm:kernel}\ref{enu:thm:kernel_Laguerre} when $\alpha_1, \dotsc, \alpha_N$ are distinct. Note that the kernel depends on $\alpha_i$ in an analytical way, when some $\alpha_i$ are identical, we obtain the kernel by analytic continuation.

Analogous to \eqref{eq:alternative_jpdf_mGUE}, we write the joint probability density formula \eqref{eq:eqivalent_jpdf} of the positive random variables $\lambda^{(n)}_i$ into
\begin{equation} \label{eq:alternative_jpdf_LUE}
  P(\lambda^{(1)}, \dots, \lambda^{(N)}) = \frac{1}{\prod^N_{n=1} \alpha_n!} \left( \prod^{N-1}_{n=1} \det(\phi_n(\lambda^{(n)}_i, \lambda^{(n+1)}_j))^{n+1}_{i,j=1} \right) \det(\Psi^N_{N-i}(\lambda^{(N)}_j))^N_{i,j=1},
\end{equation}
where the function $\phi_n(\lambda^{(n)}_i, \lambda^{(n + 1)}_j)$ ($i \neq n + 1$) has explicit formula
  \begin{equation} \label{phi_n:LUE}
    \phi_n(x, y) = x^{\alpha_n - \alpha_{n + 1} - 1} \id_{x < y},
  \end{equation}
  and
\begin{equation} \label{eq:expr_of_P_N_LUE}
  \Psi^N_j(x) = x^{\alpha_N}e^{-x} P_{(\alpha_N, \alpha_{N-1}, \dots, \alpha_{N-j+1})}(x) = \frac{x^{\alpha_N}}{2\pi i} \oint_{\Sigma} \frac{\Gamma(z+1) \prod^j_{k=1} (z-\alpha_{N-k+1})}{x^{z+1}} dz.
\end{equation}
Here $P_{(\alpha_N, \alpha_{N-1}, \dots, \alpha_{N-j+1})}(x)$ is the $j$-th degree (monic) multiple Laguerre polynomial of the first kind, type II (see \eqref{eq:contour_integral_of_P_Laguerre} in Appendix \ref{sec:appendix_MOPS}). Note that in this section, the functions $\phi_n$, $\Psi^n_j$ and later $\Phi^n_j$ are all defined on positive variables.

We apply Lemma \ref{lem:BFPS} in Section \ref{subsec:reproducing_kernel_mGUE} with $I = [0, \infty)$. Analogous to \eqref{eq:evaluation_of_phi^n1n2}, we have by a similar argument that if $n_1 < n_2$ 
\begin{equation} \label{eq:evaluation_of_phi^n1n2_LUE}
  \phi^{(n_1, n_2)}(x,y) := \phi_{n_1}*(\phi_{n_1+1}*(\dotsb *\phi_{n_2-1}))(x,y) = \frac{x^{-1}}{2\pi i} \oint_{\Gamma_{\alpha}} \frac{x^{\alpha_{n_1}-z} y^{z-\alpha_{n_2}}}{\prod^{n_2}_{k=n_1+1} (z-\alpha_k)} dz \id_{x<y} \id_{n_1 < n_2},
\end{equation}
where $\Gamma_{\alpha}$ is a large enough contour enclosing all poles $\alpha_{n_1+1}, \dots, \alpha_{n_2}$, and similar to \eqref{eq:evaluation_of_Psi^n_j}
\begin{equation} \label{eq:evaluation_of_Psi^n_j_LUE}
  \Psi^n_j(x) =
  \begin{cases}
    \begin{gathered}[b]
      \frac{x^{\alpha_n}}{2\pi i} \oint_{\Sigma} \frac{\Gamma(z+1) \prod^j_{k=1} (z - \alpha_{n-k+1})}{x^{z+1}} dz \\
      = x^{\alpha_n}e^{-x} P_{(\alpha_n, \dots, \alpha_{n-j+1})}(x),
    \end{gathered} & j = 0, 1, \dots, n-1, \\
    {\displaystyle \frac{x^{\alpha_n}}{2\pi i} \oint_{\Sigma} \frac{\Gamma(z+1)}{x^{z+1} \prod^{-j}_{k=1} (z - \alpha_{n+k})} dz,} & j = -1, -2, \dots, n-N,
  \end{cases}
\end{equation}
where the contour $\Sigma$ is the deformed Hankel contour from $-\infty$ to $-\infty$ that encloses all possible poles of the integrand, $z = -1, -2, \dots$ and $z = \alpha_{n+1}, \dots, \alpha_N$ counterclockwise. The proof of \eqref{eq:evaluation_of_Psi^n_j_LUE} will be given in the end of this subsection.

Also similar to \eqref{eq:evaluation_of_phi*phi}, we have
\begin{equation} \label{eq:func_forming_basis_V_n_LUE}
  \phi_j*\phi^{(j+1,n)}(\lambda^{(j)}_{j+1}, x) = \left( \prod^n_{k=j+2} \frac{1}{\alpha_{j+1} - \alpha_k} \right) x^{\alpha_{j+1}-\alpha_n}, \quad j = 0, 1, \dots, n-2.
\end{equation}
Note that the vector space $V_n$ is spanned by $x^{\alpha_1 - \alpha_n}, x^{\alpha_2 - \alpha_n}, \dots, x^{\alpha_n - \alpha_n}$. Hence by the orthogonality \eqref{eq:orthonormality_of_Psi_and_Phi}, we have similar to \eqref{eq:expression_of_Phi^n_j} (using \eqref{eq:orthogonality_of_Laguerre_kind1_type1}, \eqref{eq:orthogonality_of_Laguerre_kind1_type1:2} and \eqref{eq:contour_integral_of_Q_Laguerre})
\begin{equation} \label{eq:expression_of_Phi^n_j_LUE}
  \Phi^n_j(x) = x^{-\alpha_n} e^x Q_{(\alpha_n, \alpha_{n-1}, \dots, \alpha_{n-j})}(x) = \frac{x^{-\alpha_n}}{2\pi i} \oint_{\Gamma_{\alpha}} \frac{x^z}{\Gamma(z+1) \prod^j_{k=0} (z - \alpha_{n-k})} dz,
\end{equation}
where $Q_{(\alpha_n, \alpha_{n-1}, \dots, \alpha_{n-j})}(x)$ is the multiple Laguerre polynomial of the first kind, type I (see Appendix \ref{sec:appendix_MOPS}), and the contour $\Gamma_{\alpha}$ encloses the poles $\alpha_n, \dotsc, \alpha_{n-j}$.

It is clear again that the Assumption A is satisfied, and the correlation kernel is given by \eqref{eq:general_formula_of_kernel} with $\phi^{(n_1, n_2)}(x,y)$ expressed in \eqref{eq:evaluation_of_phi^n1n2_LUE}, $\Psi^{n_1}_{n_1-k}(x)$ in \eqref{eq:evaluation_of_Psi^n_j_LUE} and $\Phi^{n_2}_{n_2-k}(y)$ in \eqref{eq:expression_of_Phi^n_j_LUE}.

Next we express the kernel in the double contour integral form. From the formulas \eqref{eq:evaluation_of_Psi^n_j_LUE}, \eqref{eq:expression_of_Phi^n_j_LUE},  we compute, analogous to  \eqref{eq:expansion_of_sum_Psi*Phi_GUE}
\begin{multline} \label{eq:long_sum_form_of_the_kernel_LUE}
  \sum^{n_2}_{k=1} \Psi^{n_1}_{n_1-k}(x) \Phi^{n_2}_{n_2-k}(y) = \\
  \frac{x^{\alpha_{n_1}} y^{-\alpha_{n_2}}}{(2\pi i)^2} \oint_{\Sigma}dz \oint_{\Gamma_{\alpha}}dw \frac{y^w \Gamma(z+1)}{x^{z+1} \Gamma(w+1)} \frac{\prod^N_{l=n_2+1} (w-\alpha_l)}{\prod^N_{j=n_1+1} (z-\alpha_j)} \sum^{n_2}_{k=1} \frac{\prod^N_{j=k+1} (z-\alpha_j)}{\prod^N_{l=k} (w-\alpha_l)},
\end{multline}
where $\Gamma_{\alpha}$ encloses the poles $\alpha_1, \dotsc, \alpha_{n_2}$, and $\Sigma$, the deformed Hankel contour, encloses all the poles $-1, -2, \dotsc$ and $\Gamma_{\alpha}$. Using the telescope trick \eqref{eq:telescope_1} and \eqref{eq:telescope_identity}, we have
\begin{multline} \label{eq:sum_Psi_Phi_Laguerre}
  \sum^{n_2}_{k=1} \Psi^{n_1}_{n_1-k}(x) \Phi^{n_2}_{n_2-k}(y) = \\
  \frac{x^{\alpha_{n_1}} y^{-\alpha_{n_2}}}{(2\pi i)^2} \oint_{\Sigma}dz \oint_{\Gamma_{\alpha}}dw \frac{x^{-z-1}y^w \Gamma(z+1)}{(z-w) \Gamma(w+1)} \left( \frac{\prod^{n_1}_{j=1} (z-\alpha_j)}{\prod^{n_2}_{l=1} (w-\alpha_l)} - \frac{\prod^N_{j=n_2+1} (z-\alpha_j)}{\prod^N_{l=n_1+1} (z-\alpha_l)} \right).
\end{multline}
Note that
\begin{equation} \label{eq:the_2nd_part_in_telescope_vanishes_Laguerre}
  \begin{split}
    & \oint_{\Gamma_{\alpha}}dw \frac{x^{-z-1}y^w \Gamma(z+1)}{(z-w) \Gamma(w+1)} \frac{\prod^N_{j=n_2+1} (z-\alpha_j)}{\prod^N_{l=n_1+1} (z-\alpha_l)} \\
    ={}& \frac{\prod^N_{j=n_2+1} (z-\alpha_j)}{\prod^N_{l=n_1+1} (z-\alpha_l)} x^{-z-1} \Gamma(z+1) \oint_{\Gamma_{\alpha}} \frac{y^w}{\Gamma(w+1)} \frac{dw}{(z-w)} \\
    ={}& \frac{\prod^N_{j=n_2+1} (z-\alpha_j)}{\prod^N_{l=n_1+1} (z-\alpha_l)} x^{-z-1} \Gamma(z+1) \cdot 0 = 0.
  \end{split}
\end{equation}
Thus \eqref{eq:sum_Psi_Phi_Laguerre} can be simplified as
\begin{equation} \label{eq:simplified_formula_of_sum_Psi_Phi_Laguerre}
  \sum^{n_2}_{k=1} \Psi^{n_1}_{n_1-k}(x) \Phi^{n_2}_{n_2-k}(y) = \frac{x^{\alpha_{n_1}} y^{-\alpha_{n_2}}}{(2\pi i)^2} \oint_{\Sigma}dz \oint_{\Gamma_{\alpha}}dw \frac{x^{-z-1}y^w \Gamma(z+1)}{(z-w) \Gamma(w+1)} \frac{\prod^{n_1}_{j=1} (z-\alpha_j)}{\prod^{n_2}_{l=1} (w-\alpha_l)}.
\end{equation}
Hence the formulas \eqref{eq:simplified_formula_of_sum_Psi_Phi_Laguerre} for $\sum^{n_2}_{k=1} \Psi^{n_1}_{n_1-k}(x) \Phi^{n_2}_{n_2-k}(y)$ and \eqref{eq:evaluation_of_phi^n1n2_LUE} for $\phi^{(n_1, n_2)}(x,y)$ and \eqref{eq:general_formula_of_kernel} yield the double contour representation of the correlation kernel
\begin{multline}
  K(n_1,x;n_2,y) = \frac{x^{\alpha_{n_1}}}{y^{\alpha_{n_2}}} \left[ \frac{-1}{2\pi i} \oint_{\Gamma_{\alpha}} \frac{x^{-w-1}y^w}{\prod^{n_2}_{l=n_1+1} (w-\alpha_l)} dw \id_{x > y} \id_{n_2 > n_1} \right. \\
  + \left. \frac{1}{(2\pi i)^2} \oint_{\Sigma}dz \oint_{\Gamma_{\alpha}}dw \frac{x^{-z-1}y^w \Gamma(z+1)}{(z-w) \Gamma(w+1)} \frac{\prod^{n_1}_{j=1} (z-\alpha_j)}{\prod^{n_2}_{l=1} (w-\alpha_l)} \right],
\end{multline}
and yield Theorem \ref{thm:kernel}\ref{enu:thm:kernel_Laguerre} after conjugating out $x^{\alpha_{n_1}}/y^{\alpha_{n_2}}$, upon proving \eqref{eq:evaluation_of_Psi^n_j_LUE}.

\begin{proof}[Proof of \eqref{eq:evaluation_of_Psi^n_j_LUE}]
  From \eqref{eq:expr_of_P_N_LUE}, \eqref{eq:evaluation_of_phi^n1n2_LUE} and \eqref{eq:defn_of_Psi^n_j}, we have
\begin{equation}
  \begin{split}
    \Psi^n_j(x) = {}& \int^{\infty}_0 \phi^{(n, N)}(x,y) \Psi^N_{N-n+j}(y) dy \\
    = {}& \frac{1}{(2\pi i)^2} \int^{\infty}_x \oint_{\Gamma_{\alpha}} \frac{x^{\alpha_n - z - 1}y^z}{\prod^N_{l=n+1} (z - \alpha_l)} dz \oint_{\Sigma'} \frac{\Gamma(w+1) \prod^N_{k=n-j+1} (w - \alpha_k)}{y^{w+1}} dw dy.
  \end{split}
\end{equation}

Here the contour $\Gamma_{\alpha}$ is defined as in \eqref{eq:evaluation_of_phi^n1n2_LUE}, and the contour $\Sigma'$ is a deformed Hankel contour, such that they satisfy that for all $z \in \Gamma_{\alpha}$ and $w \in \Sigma'$, $\Re w < \Re z$. This property will be used in \eqref{eq:defn_of_partII_Laguerre} to make $\int^x_0 y^{z-w-1} dy$ well defined. Then using the decomposition $\int^{\infty}_x = \int^{\infty}_0 - \int^x_0$, we express $\Psi^n_j(x) = x^{\alpha_n} (\PartI - \PartII)$, such that 
\begin{equation} \label{eq:partI_in_Psi_LUE}
  \begin{split}
    \PartI = {}& \int^{\infty}_0 \frac{1}{(2\pi i)^2} \oint_{\Gamma_{\alpha}} \frac{x^{-z-1} y^z}{\prod^N_{l=n+1} (z - \alpha_l)} dz \oint_{\Sigma'} \frac{\Gamma(w+1) \prod^N_{k=n-j+1} (w - \alpha_k)}{y^{w+1}} dw dy \\
    = {}& \frac{1}{2\pi i} \oint_{\Gamma_{\alpha}} \frac{x^{-z-1}}{\prod^N_{l=n+1} (z - \alpha_l)} \int^{\infty}_0 y^z e^{-y} P_{(\alpha_N, \alpha_{N-1}, \dots, \alpha_{n-j+1})}(y) dy dz,
  \end{split}
\end{equation}
where the second identity is the consequence of \eqref{eq:expr_of_P_N_LUE}, and
\begin{equation} \label{eq:defn_of_partII_Laguerre}
  \begin{split}
    \PartII = {}& \int^x_0 \frac{1}{(2\pi i)^2} \oint_{\Gamma_{\alpha}} \frac{x^{-z-1} y^z}{\prod^N_{l=n+1} (z - \alpha_l)} dz \oint_{\Sigma'} \frac{\Gamma(w+1) \prod^N_{k=n-j+1} (w - \alpha_k)}{y^{w+1}} dw dy \\
    = {}& \frac{1}{(2\pi i)^2} \oint_{\Gamma_{\alpha}} dz \oint_{\Sigma'} dw x^{-z-1}\Gamma(w+1) \frac{\prod^N_{k=n-j+1} (w - \alpha_k)}{\prod^N_{l=n+1} (z - \alpha_l)} \int^x_0 y^{z-w-1} dy \\
  = {}& \frac{1}{(2\pi i)^2} \oint_{\Gamma_{\alpha}} dz \oint_{\Sigma'} dw \frac{x^{-w-1}\Gamma(w+1)}{z-w} \frac{\prod^N_{k=n-j+1} (w - \alpha_k)}{\prod^N_{l=n+1} (z - \alpha_l)}.
  \end{split}
\end{equation}
Note that because $P_{(\alpha_N, \alpha_{N-1}, \dots, \alpha_{n-j+1})}(y)$ is a monic polynomial of degree $N - n + j$, the integral with respect to $y$ in \eqref{eq:partI_in_Psi_LUE} is $\Gamma(z+1)$ times a monic polynomial in $z$ of degree $N - n + j$; because of the orthogonal property of $P_{(\alpha_N, \alpha_{N-1}, \dots, \alpha_{n-j+1})}(y)$, the integral vanishes as $z = \alpha_N, \dotsc, \alpha_{n-j+1}$. Hence we find that 
\begin{equation} \label{eq:simplified_Part_I_LUE}
  \PartI = \frac{1}{2\pi i} \oint_{\Gamma_{\alpha}} \frac{\Gamma(z+1)}{x^{z+1}} \frac{\prod^N_{k=n-j+1} (z-\alpha_k)}{\prod^N_{l=n+1} (z - \alpha_l)} dz.
\end{equation}
In $\PartII$, if we integrate $z$ first, by calculation of residues we have 
\begin{equation} \label{eq:simplified_part_II_LUE}
  \begin{split}
    \PartII ={}& \frac{1}{2\pi i} \oint_{\Sigma'} x^{-w-1} \Gamma(w+1) \prod^N_{k=n-j+1} (w - \alpha_k) \left( \sum^N_{l = n+1} \res_{z = \alpha_l} \frac{1}{(z - w) \prod^N_{l=n+1} (z - \alpha_l)} \right) \\
    ={}& \frac{1}{2\pi i} \oint_{\Sigma'} x^{-w-1} \Gamma(w+1) \prod^N_{k=n-j+1} (w - \alpha_k) \left(-\res_{z = w} \frac{1}{(z - w) \prod^N_{l=n+1} (z - \alpha_l)} \right) \\
    ={}& \frac{-1}{2\pi i} \oint_{\Sigma'} \frac{\Gamma(w+1)}{x^{w+1}} \frac{\prod^N_{k=n-j+1} (w - \alpha_k)}{\prod^N_{l=n+1} (w - \alpha_l)} dw,
  \end{split}
\end{equation}
where we use the identity $(\res_{z=w} + \sum^N_{l = n+1} \res_{z = \alpha_l}) (z - w)^{-1} \prod^N_{l=n+1} (z - \alpha_l)^{-1} = 0$. Since as integration contours, $\Gamma_{\alpha} + \Sigma' = \Sigma$ where $\Sigma$ is the contour in \eqref{eq:evaluation_of_Psi^n_j_LUE}, we have
\begin{equation}
  \PartI - \PartII = \frac{1}{2\pi i} \oint_{\Sigma} \frac{\Gamma(w+1)}{x^{w+1}} \frac{\prod^N_{k=n-j+1} (w - \alpha_k)}{\prod^N_{l=n+1} (w - \alpha_l)} dw,
\end{equation}
and we prove \eqref{eq:evaluation_of_Psi^n_j_LUE}.
\end{proof}

\section{Joint distribution of eigenvalues of minor quotients in
multiple JUE}

\subsection{Transition probability of the Markov chain $\lambda^{(n)}$} \label{subsec:transition_prob_JUE}

Analogous to Lemma \ref{lem:GUE_w_ext} in Section \ref{sec:GUE_w_ext} and Lemma \ref{lem:LUE_w_ext} in Section \ref{sec:multiple_LUE}, we need the following technical lemma. Like in Lemma \ref{lem:LUE_w_ext}, $l_1, l_2, \dots$ are nondecreasing integers with $l_n \geq n$, and $\alpha_n = l_n - n$.
\begin{lem} \label{lem:JUE_w_ext}
  Let $X'$ be an $l_{n-1} \times (n-1)$ fixed rectangular matrix and $Y'$ be an $M' \times (n-1)$ fixed rectangular matrix with $M' \geq n$ such that $R' = ((X')^* X' + (Y')^* Y')^{-1} (X')^* X'$ has eigenvalues $\tilde{\mu} = (\tilde{\mu}_1, \dots, \tilde{\mu}_{n-1})$ in increasing order. Let $X$ and $Y$ be $l_n \times n$ and $M' \times n$ rectangular random matrices such that
  \begin{itemize}
  \item
    The upper-left $l_{n-1} \times (n-1)$ block of $X$ and the left $M' \times (n-1)$ block of $Y$ are fixed and are equal to $X'$ and $Y'$.
  \item
    Denote the $l_n$ dimensional column vector $\xvec := (x_{1,n}, \dots, x_{l_n,n})^{\perp}$, the last column of $X$, and the $M'$ dimensional column vector $\yvec := (y_{1,n}, \dots, y_{M',n})^{\perp}$, the last column of $Y$. Then components of $\xvec$ and $\yvec$ are independent complex random variables in standard normal distribution, \ie, $p(\Re x_{i,n}) = p(\Re y_{i,n}) = \normal(0, 1/2)$ and $p(\Im x_{i,n}) = p(\Im y_{i,n}) = \normal(0, 1/2)$.
  \item
    All entries in the lower-left $(l_n - l_{n-1}) \times (n-1)$ block of $X$ are $0$.
  \end{itemize}
  Denote the eigenvalues of $R = (X^* X + Y^* Y)^{-1} X^* X$ by $\tilde{\lambda} = (\tilde{\lambda}_1, \dots, \tilde{\lambda}_n)$ where $0 \leq \tilde{\lambda}_1 \leq \dots \leq \tilde{\lambda}_n \leq 1$. Then the joint distribution of $\tilde{\lambda}$ is
  \begin{equation} \label{eq:trans_prob_JUE}
    p_{\tilde{\mu}}(\tilde{\lambda}) = \frac{(M' + \alpha_n)!}{\alpha_n! (M'-n)!} \frac{\Delta_n(\tilde{\lambda}) \prod^n_{i=1} \tilde{\lambda}^{\alpha_n}_i (1-\tilde{\lambda}_i)^{M'-n}}{\Delta_{n-1}(\tilde{\mu}) \prod^{n-1}_{i=1} \tilde{\mu}^{\alpha_n+1}_i (1-\tilde{\mu}_i)^{M'-n+1}} \id_{\tilde{\mu} \preceq \tilde{\lambda}}.
  \end{equation}
\end{lem}
\begin{proof}
  Instead of $R'$, $R$, $\mu$ and $\lambda$, we consider $T' = ((Y')^* Y')^{-1} (X')^* X'$ that has distinct eigenvalues $\mu = (\mu_1, \dots, \mu_{n-1})$, where $0 < \mu_1 < \dots < \mu_{n-1}$, and $T = (Y^* Y)^{-1} X^* X$ that has eigenvalues $\lambda = (\lambda_1, \dots, \lambda_n)$ where $0 \leq \lambda_1 \leq \dots \leq \lambda_n$. Below we prove
  \begin{equation} \label{eq:trans_prob_lemma_JUE}
    p_{\mu}(\lambda) = \frac{(M' + \alpha_n)!}{\alpha_n! (M'-n)!} \frac{\Delta_n(\lambda) \prod^n_{i=1} \lambda^{\alpha_n}_i (1+\lambda_i)^{-(M'+\alpha_n+1)}}{\Delta_{n-1}(\mu) \prod^{n-1}_{i=1} \mu^{\alpha_n+1}_i (1+\mu_i)^{-(M'+\alpha_n)}} \id_{\mu \preceq \lambda},
  \end{equation}
  and \eqref{eq:trans_prob_JUE} is a direct consequence of substituting $\lambda_i = \tilde{\lambda}_i/(1-\tilde{\lambda}_i)$ and $\mu_i = \tilde{\mu}_i / (1 - \tilde{\mu}_i)$ into the probability measure defined by \eqref{eq:trans_prob_lemma_JUE}.
  
  From the assumption of the lemma, $(X')^*X'$, $(Y')^* Y'$ and $(X')^*X' + (Y')^* Y'$ are invertible, and from the randomness of $\xvec$ and $\yvec$, $X^* X$, $Y^* Y$ and $X^* X + Y^* Y$ are almost surely invertible. Below we assume the invertibility of them.
  
  By QR decomposition of $X'$, we have $U_1 \in \U(l_{n-1})$ such that
  \begin{equation} \label{eq:QR_of_X'}
    X' = U_1
    \begin{pmatrix}
      C \\
      0_{(l_{n-1} - n+1) \times (n-1)}
    \end{pmatrix},
  \end{equation}
  where $C$ is an $(n-1) \times (n-1)$ upper-triangular matrix. Let $U_2 = U_1 \oplus I_{l_n - l_{n-1}} \in \U(l_n)$, we have
  \begin{equation} \label{eq:tilde_x_1_and_2}
    X = U_2
    \begin{pmatrix}
      C & \tilde{\xvec}_1 \\
      0_{(l_n - n+1) \times (n-1)} & \tilde{\xvec}_2
    \end{pmatrix}, 
  \end{equation}
  where $\tilde{\xvec}_1 = (\tilde{x}_1, \dots, \tilde{x}_{n-1})^{\perp}$ and $\tilde{\xvec}_2 = (\tilde{x}_n, \dots, \tilde{x}_{l_n})^{\perp}$ are vectors of dimensions $n-1$ and $l_n-n+1$ respectively, such that if we concatenate them into a $l_n$ dimensional vector $\tilde{\xvec} = (\tilde{x}_1, \dots, \tilde{x}_{l_n})$, then $U_2 \tilde{\xvec} = \xvec$. We choose an $U_3 \in \U(l_n - l_{n-1})$ such that 
  \begin{equation}
    U^{-1}_3\tilde{\xvec}_2 = (\xi, 0, \dots, 0)^{\perp}, \quad \text{where} \quad \xi = \sqrt{\sum^{l_n}_{k=n} \lvert \tilde{x}_{k,n} \rvert^2},
  \end{equation}
  and denote $U_4 = U_2(I_{n-1} \oplus U_3)$, then we have
  \begin{equation} \label{eq:X_and_tilde_X}
    X = U_4
    \begin{pmatrix}
       \tilde{X} \\
      0_{(l_n-n) \times n}
    \end{pmatrix},
    \quad \text{where} \quad \tilde{X} =
    \begin{pmatrix}
      C & \tilde{\xvec}_1 \\
      0_{1 \times (n-1)} & \xi
    \end{pmatrix}.
  \end{equation}
  On the other hand, by the QR decomposition of the $M' \times (n-1)$ matrix $Y'$, we have $U_5 \in \U(M')$ such that
  \begin{equation} \label{eq:QR_of_Y'}
    Y' = U_5
    \begin{pmatrix}
      B \\
      0_{(M'-n+1) \times (n-1)}
    \end{pmatrix},
  \end{equation}
  where $B$ is an $(n-1) \times (n-1)$ upper-triangular matrix. Then we have
  \begin{equation}
    Y = U_5
    \begin{pmatrix}
      B & \tilde{\yvec}_1 \\
      0_{(M'-n+1) \times (n-1)} & \tilde{\yvec}_2
    \end{pmatrix},
  \end{equation}
  where analogous to $\tilde{\xvec}_1$ and $\tilde{\xvec}_2$ in \eqref{eq:tilde_x_1_and_2}, $\tilde{\yvec}_1 = (\tilde{y}_1, \dots, \tilde{y}_{n-1})^{\perp}$ and $\tilde{\yvec}_2 = (\tilde{y}_n, \dots, \tilde{y}_{M'})^{\perp}$ are vectors of dimensions $n-1$ and $M'-n+1$ respectively, and if we concatenate them into $\tilde{\yvec} = (\tilde{y}_1, \dots, \tilde{y}_{M'})$, then $\tilde{\yvec} = U^{-1}_5 \yvec$. We choose an $U_6 \in \U(M'-n+1)$ such that
  \begin{equation}
    U^{-1}_6 \tilde{\yvec}_2 = (\eta, 0, \dots, 0), \quad \text{where} \quad \eta = \sqrt{\sum^{M'}_{k=n} \lvert \tilde{y}_{k,n} \rvert^2}.
  \end{equation}
  Analogous to $U_4$ in \eqref{eq:X_and_tilde_X}, let $U_7 = U_5(I_{n-1} \oplus U_6)$, we have
  \begin{equation} \label{eq:Y_and_tilde_Y}
    Y = U_7
    \begin{pmatrix}
      \tilde{Y} \\
      0_{(M'-n) \times n}
    \end{pmatrix},
    \quad \text{where} \quad \tilde{Y} =
    \begin{pmatrix}
      B & \tilde{\yvec}_1 \\
      0_{1 \times (n-1)} & \eta
    \end{pmatrix}.
  \end{equation}

  From \eqref{eq:QR_of_X'}, \eqref{eq:X_and_tilde_X}, \eqref{eq:QR_of_Y'}, \eqref{eq:Y_and_tilde_Y}, we have
  \begin{align}
    (X')^* X' = {}& C^* C, & (Y')^* Y' = {}& B^* B, & T' = {}& (B^* B)^{-1} C^* C, \\
    X^* X = {}& \tilde{X}^* \tilde{X}, & Y^* Y = {}& \tilde{Y}^* \tilde{Y}, & T = {}& (\tilde{Y}^* \tilde{Y})^{-1} \tilde{X}^* \tilde{X}.
  \end{align}
  From the invertibility of $(Y')^* Y'$ and $Y^* Y$, we find that $B$ and $\tilde{Y}$ are invertible, and then $T'$ and $T$ are similar to
  \begin{equation}
    \tilde{T}' = (CB^{-1})^* (CB^{-1}), \quad \tilde{T} = (\tilde{X} \tilde{Y}^{-1})^* (\tilde{X} \tilde{Y}^{-1})
  \end{equation}
  respectively, which implies the relation between characteristic polynomials, $p_{T'}(z) = p_{\tilde{T}'}(z)$ and $p_{T}(z) = p_{\tilde{T}}(z)$.
  
  We have
  \begin{equation}
    \tilde{X} \tilde{Y}^{-1} =
    \begin{pmatrix}
      CB^{-1} & \eta^{-1} (\tilde{\xvec}_1 - CB^{-1} \tilde{\yvec}_1) \\
      0_{1 \times (n-1)} & \xi \eta^{-1}
    \end{pmatrix}.
  \end{equation}
  Since $T'$ and hence $\tilde{T}' = (CB^{-1})^* (CB^{-1})$ has eigenvalues $\mu_1, \dots, \mu_{n-1}$, we have the singular value decomposition that for $U_8, V_1 \in \U(n-1)$,
  \begin{equation}
    CB^{-1} = U_8 D V^*_1, \quad \text{where} \quad D = \diag(\sqrt{\mu_1}, \dots, \sqrt{\mu_{n-1}}).
  \end{equation}
  Let $U_9 = U_8 \oplus I_1 \in \U(n)$ and $V_2 = V_1 \oplus I_1 \in \U(n)$, we have
  \begin{equation}
    \tilde{X}\tilde{Y}^{-1} = U_9
    \begin{pmatrix}
      D & \eta^{-1} \wvec \\
      0_{1 \times (n-1)} & \xi\eta^{-1}
    \end{pmatrix}
    V^*_2, \quad \text{where} \quad \wvec = U^{-1}_8 \tilde{\xvec}_1 - DV^*_1 \tilde{\yvec}_1.
  \end{equation}

  We consider the relationship between the characteristic polynomials of $T$ and $T'$ (remembering $p_{T}(z) = p_{\tilde{T}}(z)$), in the same way as \eqref{eq:partially_diagonalized_tilde_X}--\eqref{eq:defn_of_xi_j_LUE}, and find
  \begin{align}
    \frac{p_T(z)}{p_{T'}(z)} = {}& z - \zeta_n - \sum^{n-1}_{i=1} \frac{z}{z  - \mu_i}\zeta_i, \label{eq:char_poly_in_JUE_w_ext:1} \\
    \frac{p_T(z)}{p_{T'}(z)} = {}& z - (\sigma_1(\lambda) - \sigma_1(\mu)) - \frac{1}{z} (\sigma_1(\lambda) \sigma_1(\mu) + \sigma_2(\mu) -\sigma_2(\lambda) - \sigma^2_1(\mu)) + \bigO(\frac{1}{z^2}), \label{eq:char_poly_in_JUE_w_ext:2}
  \end{align}
  where $\sigma_1, \sigma_2$ in \eqref{eq:char_poly_in_JUE_w_ext:2} are defined the same as in \eqref{eq:char_poly_in_GUE_w_ext:2} and \eqref{eq:char_poly_in_LUE_w_ext:2}, and $\zeta_i$ in \eqref{eq:char_poly_in_JUE_w_ext:1} are defined similarly to $\xi_i$ in \eqref{eq:char_poly_in_LUE_w_ext:1}
  \begin{equation} \label{eq:defn_of_xi_j_JUE}
    \zeta_i =
    \begin{cases}
      \lvert w_i \rvert^2 / \eta^2 & i = 1, \dots, n-1, \\
      \xi^2 / \eta^2 & i=n.
    \end{cases}
  \end{equation}
From \eqref{eq:char_poly_in_JUE_w_ext:1} we find that the eigenvalues of $T$ and $T'$ interlace, that is, $\mu \preceq \lambda$. Like \eqref{eq:interlacing_of_mu_and_lambda_LUE}--\eqref{eq:expression_of_xi_n} in Section \ref{subsec:transition_prob_LUE} that there is a homeomorphism between $\lambda$ that satisfies the strict interlacing condition \eqref{eq:interlacing_of_mu_and_lambda_LUE} and $\zeta_1, \dots, \zeta_n \in \realR_+$, and
\begin{align}
    \zeta_i = {}& -\frac{p_T(\mu_i)}{\mu_i p'_{T'}(\mu_i)} = - \res_{z = \mu_i} \frac{\prod^{n}_{k=1} (z-\lambda_k)}{z \prod^{n-1}_{k=1} (z-\mu_k)}, \quad i = 1, \dots, n-1, \label{eq:expression_of_zeta_j} \\
    \zeta_n = {}& \sigma_1(\lambda) - \sigma_1(\mu) - \sum^{n-1}_{i=1} \zeta_i = -\res_{z = 0} \frac{\prod^{n}_{k=1} (z-\lambda_k)}{z \prod^{n-1}_{k=1} (z-\mu_k)} = \frac{\prod^n_{i=1} \lambda_i}{\prod^{n-1}_{i=1} \mu_i}. \label{eq:expression_of_zeta_n}
 \end{align}
 Hence like \eqref{Jacobian_in_LUE}, we have
\begin{equation} \label{Jacobian_in_JUE}
    \frac{\partial (\zeta_1, \dots, \zeta_n)}{\partial (\lambda_1, \dots, \lambda_n)} = (-1)^{n-1} \frac{\Delta_n(\lambda)}{\Delta_{n-1}(\mu) \prod^{n-1}_{j=1} \mu_j}.
  \end{equation}
 
  We note that the random variables $\zeta_1, \dots, \zeta_n$ are not independent. However, from the definition, components of $\tilde{\xvec}$ and $\tilde{\yvec}$, and $\eta$ and $\xi$ are independent, such that $\tilde{x}_{i,n}$ and $\tilde{y}_{i,n}$ are in standard complex normal distribution. Thus $w_i$ are independent normal distribution such that $\Re w_i$ and $\Im w_i$ are independent random variables in $\normal(0, \frac{1}{2} (1+\mu_i))$ distribution, $\sqrt{2}\eta$ is in $\chi_{2(M'-n+1)}$ distribution and $\sqrt{2}\xi$ is in $\chi_{2(\alpha_n+1)}$ distribution. Hence $\lvert w_1 \rvert^2, \dots, \lvert w_{n-1} \rvert^2, \eta^2, \xi^2$ are independent, and their distribution functions are
  \begin{gather}
    p_{\lvert w_i \rvert^2}(x) =  \frac{1}{1+\mu_k} e^{-\frac{x}{1+\mu_k}}, \quad \text{for} \quad i = 1, \dots, n-1, \\
    p_{\eta^2}(x) = \frac{1}{(M' - n)!}x^{M' - n}e^{-x}, \quad p_{\xi^2}(x) = \frac{1}{\alpha_n!}x^{\alpha_n}e^{-x}.
  \end{gather}
  By the relation \eqref{eq:defn_of_xi_j_JUE} between $\zeta_i$ and $\lvert w_1 \rvert^2, \dots, \lvert w_{n-1} \rvert^2, \eta^2, \xi^2$, we compute the Jacobian
  \begin{equation}
    \frac{\partial(\zeta_1, \dotsc, \zeta_n, \eta^2)}{\partial(\lvert w_1 \rvert^2, \dotsc, \lvert w_{n-1} \rvert^2, \xi^2, \eta^2)} = (\eta^2)^n,
  \end{equation}
  we have that
  \begin{equation} \label{eq:jpdf_of_zeta}
    \begin{split}
      p(\zeta_1, \dots, \zeta_n) = {}& \int^{\infty}_0 p(\zeta_1, \dotsc, \zeta_n, \eta^2 = r) dr \\
      = {}& \int^{\infty}_0 r^n p_{\eta^2}(r) \left( \prod^{n-1}_{i=1} p_{\lvert w_i \rvert^2}(r\zeta_i) \right) p_{\xi^2}(r\zeta_n) dr \\
      = {}& \frac{(M' + \alpha_n)!}{\alpha_n! (M'-n)! \prod^{n-1}_{i=1} (1+\mu_i)} \frac{\zeta^{\alpha_n}_n}{\left( 1 + \sum^{n-1}_{i=1} \frac{\zeta_i}{1+\mu_i} + \zeta_n \right)^{M' + \alpha_n + 1}} \\
      = {}&  \frac{(M' + \alpha_n)!}{\alpha_n! (M'-n)!} \frac{\prod^n_{i=1} \lambda^{\alpha_n}_i (1+\lambda_i)^{-(M'+\alpha_n+1)}}{\prod^{n-1}_{i=1} \mu^{\alpha_n}_i (1+\mu_i)^{-(M'+\alpha_n)}}.
    \end{split}
  \end{equation}
  Here we use in the last step identity \eqref{eq:expression_of_zeta_n} and the following identity which is a consequence of \eqref{eq:expression_of_zeta_j} and \eqref{eq:expression_of_zeta_n}:
  \begin{equation}
    \begin{split}
      1 + \sum^{n-1}_{i=1} \frac{\zeta_i}{1+\mu_i} + \zeta_n = {}& 1 - \sum^{n-1}_{i=1} \res_{z = \mu_i} \frac{\prod^{n}_{k=1} (z-\lambda_k)}{z(z+1) \prod^{n-1}_{k=1} (z-\mu_k)} - \res_{z = 0} \frac{\prod^{n}_{k=1} (z-\lambda_k)}{z(z+1) \prod^{n-1}_{k=1} (z-\mu_k)} \\
      = {}& \res_{z = -1} \frac{\prod^{n}_{k=1} (z-\lambda_k)}{z(z+1) \prod^{n-1}_{k=1} (z-\mu_k)} \\
      = {}& \frac{\prod^n_{i=1} (1 + \lambda_i)}{\prod^{n-1}_{i=1} (1 + \mu_i)}.
    \end{split}
  \end{equation}
  Analogous to \eqref{eq:relations_between_densities_LUE}, we prove \eqref{eq:trans_prob_lemma_JUE} from \eqref{Jacobian_in_JUE} and \eqref{eq:jpdf_of_zeta}.
\end{proof}

\begin{proof}[Proof of Theorem \ref{thm:Markov} (Jacobi-\Pineiro\ case)]
  Like in the proofs in the GUE with external source and the multiple Laguerre cases, the Markov property of the process relies on that the distribution of the eigenvalues of $R$ in Lemma \ref{lem:JUE_w_ext} depends on the eigenvalues of $R'$, but not its eigenvectors. By the result of Lemma \ref{lem:JUE_w_ext}, we have like \eqref{eq:limiting_identity_for_Markov} and \eqref{eq:multi_conditional_LUE} that
  \begin{equation} \label{eq:multi_conditional_JUE}
    p(\lambda^{(n)} \mid \lambda^{(i)} = \mu^{(i)}, \ i = 1, \dotsc, n-1) = p_{\mu^{(n-1)}}(\lambda^{(n)}), 
  \end{equation}
  where the distribution function $p_{\mu^{(n-1)}}$ in \eqref{eq:multi_conditional_JUE} is defined in \eqref{eq:trans_prob_JUE}. Thus we prove the Markov property, while \eqref{eq:trans_prob_JUE} yields \eqref{eq:transition_general} in the Jacobi-\Pineiro\ case, and hence we finish the proof.
\end{proof}

\subsection{Correlation kernel of the Markov chain $\lambda^{(n)}$} \label{subsec:reproducing_kernel_JUE}

Like in Section \ref{subsec:reproducing_kernel_LUE}, we prove Theorem \ref{thm:kernel}\ref{enu:thm:kernel_Jacobi} for distinct $\alpha_1, \dotsc, \alpha_N$, and obtain the general case by analytic continuation.

Analogous to \eqref{eq:alternative_jpdf_mGUE} and \eqref{eq:alternative_jpdf_LUE}, we write the joint probability density formula \eqref{eq:eqivalent_jpdf} of the random variables $\lambda^{(n)}_i$ whose range is $[0,1]$ as follows:
\begin{multline}
  P(\lambda^{(1)}, \dotsc, \lambda^{(N)}) = \\
  \prod^N_{n=1} \frac{(M'-N+n+\alpha_{N-n+1})!}{\alpha_n!}  \left( \prod^{N-1}_{n=1} \det(\phi_n(\lambda^{(n)}_i, \lambda^{(n+1)}_j))^{n+1}_{i,j=1} \right) \det(\Psi^N_{N-i}(\lambda^{(N)}_j))^N_{i,j=1},
\end{multline}
where the function $\phi_n(\lambda^{(n)}_i, \lambda^{(n + 1)}_j)$ ($i \neq n + 1$) has explicit formula
  \begin{equation} \label{phi_n:JUE}
    \phi_n(x, y) = x^{\alpha_n - \alpha_{n + 1} - 1} \id_{x < y},
  \end{equation}
  and
\begin{equation} \label{eq:P_N_in_JUE}
  \begin{split}
    \Psi^N_j(x) ={}& \frac{\prod^j_{k=1} (M'-N+j+1+\alpha_{N-k+1})}{(M'-N+j)!} x^{\alpha_N}(1-x)^{M'-N} P_{(\alpha_N, \alpha_{N-1}, \dots, \alpha_{N-j+1}; M'-N)}(x) \\
    ={}& \frac{x^{\alpha_N}}{2\pi i} \oint_{\Sigma} \frac{x^{-z-1} \Gamma(z+1) \prod^j_{k=1}(z-\alpha_{N-k+1})}{\Gamma(z+j+2+M'-N)} dz.
  \end{split}
\end{equation}
Here $P_{(\alpha_N, \alpha_{N-1}, \dots, \alpha_{N-j+1}; M'-N)}(x)$ is the $j$-th degree (monic) Jacobi-\Pineiro\ polynomial of type II, and we use the contour integral representation \eqref{eq:contour_integral_of_P} in Appendix \ref{sec:appendix_MOPS}, and the contour $\Sigma$ encloses the poles $z = -1, -2, \dotsc, -(j+1+M'-N)$. Note that in this section, the functions $\phi_n$, $\Psi^n_j$ and later $\Phi^n_j$ are all defined on variables in $(0,1)$.

We apply Lemma \ref{lem:BFPS} in Section \ref{subsec:reproducing_kernel_mGUE} with $I = [0, 1]$. Since our $\phi_n(x,y)$ defined in \eqref{phi_n:JUE} is very similar to the $\phi_n(x,y)$ defined in \eqref{phi_n:JUE}, with only the domain different, we have similar to \eqref{eq:evaluation_of_phi^n1n2_LUE} that
\begin{equation} \label{eq:evaluation_of_phi^n1n2_JUE}
  \phi^{(n_1, n_2)}(x,y) = \frac{x^{-1}}{2\pi i} \oint_{\Gamma_{\alpha}} \frac{x^{\alpha_{n_1}-z} y^{z-\alpha_{n_2}}}{\prod^{n_2}_{k=n_1+1} (z-\alpha_k)} dz \id_{x<y} \id_{n_1 < n_2}.
\end{equation}
where the contour $\Gamma_{\alpha}$ is the same as the $\Gamma_{\alpha}$ in \eqref{eq:evaluation_of_phi^n1n2_LUE}. Then similar to \eqref{eq:evaluation_of_Psi^n_j} and \eqref{eq:evaluation_of_Psi^n_j_LUE},
\begin{equation} \label{eq:evaluation_of_Psi^n_j_JUE}
  \Psi^n_j(x) =
  \begin{cases}
    \begin{aligned}[b]
      \frac{x^{\alpha_n}}{2\pi i} & \oint_{\Sigma} \frac{x^{-z-1} \Gamma(z+1) \prod^j_{k=1} (z - \alpha_{n-k+1})}{\Gamma(z+j+2+M'-n)} dz \\
      ={}& \frac{\prod^j_{k=1}(j+1+M'-n+\alpha_{n-k+1})}{(M'-n+j)!} \\
      & \times x^{\alpha_n}(1-x)^{M'-n} P_{(\alpha_n, \dots, \alpha_{n-j+1}; M'-n)}(x),
    \end{aligned} & j = 0, 1, \dots, n-1, \\
    {\displaystyle \frac{x^{\alpha_n}}{2\pi i} \oint_{\Sigma} \frac{x^{-z-1} \Gamma(z+1)}{\Gamma(z+j+2+M'-n) \prod^{-j}_{k=1} (z - \alpha_{n+k})} dz,} & j = -1, -2, \dots, n-N,
  \end{cases}
\end{equation}
where the contours $\Sigma$ encloses all poles of the integrand. The proof of \eqref{eq:evaluation_of_Psi^n_j_JUE} will be given in the end of this subsection.

Like \eqref{eq:func_forming_basis_V_n_LUE}, we have
\begin{equation} \label{eq:func_forming_basis_V_n_JUE}
  \phi_j*\phi^{(j+1,n)}(\lambda^{(j)}_{j+1}, x) = \left( \prod^n_{k=j+2} \frac{1}{\alpha_{j+1} - \alpha_k} \right) x^{\alpha_{j+1}-\alpha_n}, \quad j = 0, 1, \dots, n-2.
\end{equation}
Then the vector space $V_n$ is also spanned by $x^{\alpha_1-\alpha_n}, x^{\alpha_2-\alpha_n}, \dotsc, x^{\alpha_n-\alpha_n}$. Hence like \eqref{eq:expression_of_Phi^n_j} and \eqref{eq:expression_of_Phi^n_j_LUE}, by the orthogonality \eqref{eq:orthonormality_of_Psi_and_Phi} and \eqref{eq:vanishing_defn_of_Q}, \eqref{eq:non-vanishing_defn_of_Q} and \eqref{eq:contour_integral_of_Q}, we have  
\begin{equation} \label{eq:expression_of_Phi^n_j_JUE}
  \begin{split}
    \Phi^n_j(x) ={}& \frac{(M'-n+j)!}{\prod^j_{k=1} (j+1+M'-n+\alpha_{n-k+1})} x^{-\alpha_n}(1-x)^{n-M'} Q_{(\alpha_n, \alpha_{n-1}, \dotsc, \alpha_{n-j}; M'-n)}(x) \\
    ={}& (j+1+M'-n+\alpha_{n-j}) \frac{x^{-\alpha_n}}{2\pi i} \oint_{\Gamma_{\alpha}} \frac{x^z \Gamma(z+j+1+M'-n)}{\Gamma(z+1) \prod^j_{k=0}(z-\alpha_{n-k})} dz,
  \end{split}
\end{equation}
where $Q_{(\alpha_n, \alpha_{n-1}, \dotsc, \alpha_{n-j}; M'-n)}(x)$ is the Jacobi-\Pineiro\ polynomial of type I (see Appendix \ref{sec:appendix_MOPS}), and the contour $\Gamma_{\alpha}$ is the same as the $\Gamma_{\alpha}$ in \eqref{eq:expression_of_Phi^n_j_LUE}.

It is clear that the Assumption A is satisfied, and the correlation kernel is given by \eqref{eq:general_formula_of_kernel} with $\phi^{(n_1, n_2)}(x,y)$ expressed in \eqref{eq:evaluation_of_phi^n1n2_JUE}, $\Psi^{n_1}_{n_1-k}(x)$ in \eqref{eq:evaluation_of_Psi^n_j_JUE} and $\Phi^{n_2}_{n_2-k}(y)$ in \eqref{eq:expression_of_Phi^n_j_JUE}. To express the kernel in the double contour integral form, we write like \eqref{eq:expansion_of_sum_Psi*Phi_GUE} and \eqref{eq:long_sum_form_of_the_kernel_LUE} that
\begin{equation} \label{eq:long_sum_form_of_the_kernel_JUE}
  \begin{split}
     \sum^{n_2}_{k=1} \Psi^{n_1}_{n_1-k}(x) \Phi^{n_2}_{n_2-k}(y) ={}& \frac{x^{\alpha_{n_1}} y^{-\alpha_{n_2}}}{(2\pi i)^2} \oint_{\Sigma}dz \oint_{\Gamma_{\alpha}}dw \frac{y^w \Gamma(z+1)}{x^{z+1} \Gamma(w+1)} \frac{\prod^N_{l=n_2+1} (w-\alpha_l)}{\prod^N_{j=n_1+1} (z-\alpha_j)} \\
     & \times \sum^{n_2}_{k=1} (M'-k+1+\alpha_k) \frac{\prod^N_{j=k+1} (z-\alpha_j)}{\prod^N_{l=k} (w-\alpha_l)} \frac{\Gamma(w+M'-k+1)}{\Gamma(z+M'-k+2)} \\
     ={}& \frac{x^{\alpha_{n_1}} y^{-\alpha_{n_2}}}{(2\pi i)^2} \oint_{\Sigma}dz \oint_{\Gamma_{\alpha}}dw \frac{y^w \Gamma(z+1) \Gamma(w+M'+1)}{x^{z+1} \Gamma(w+1) \Gamma(z+M'+1)} \frac{\prod^N_{l=n_2+1} (w-\alpha_l)}{\prod^N_{j=n_1+1} (z-\alpha_j)} \\
     & \times  \sum^{n_2}_{k=1} (M'-k+1+\alpha_k) \frac{\prod^N_{j=k+1} (z-\alpha_j)}{\prod^N_{l=k} (w-\alpha_l)} \frac{\prod^{k-1}_{j'=1} (z+M'-j'+1)}{\prod^k_{l'=1} (w+M'-l'+1)},
  \end{split}
\end{equation}
where the contour $\Sigma$ enclose the poles $-1, -2, \dotsc, -M'$ and the contour $\Gamma_{\alpha}$ (Compare it with the $\Sigma$ in \eqref{eq:sum_Psi_Phi_Laguerre}). By the telescope trick like \eqref{eq:telescope_1} and \eqref{eq:telescope_identity}, writing $(z - w)(M'-k+1+\alpha_k) = (z-\alpha_k)(w+M'-k+1) - (w-\alpha_k)(z+M'-k+1)$, we have
\begin{multline}
  (z-w) \sum^{n_2}_{k=1} (M'-k+1+\alpha_k) \frac{\prod^N_{j=k+1} (z-\alpha_j)}{\prod^N_{l=k} (w-\alpha_l)} \frac{\prod^{k-1}_{j'=1} (z+M'-j'+1)}{\prod^k_{l'=1} (w+M'-l'+1)} = \\
  \prod^N_{j=1} \frac{z-\alpha_j}{w-\alpha_j} - \prod^N_{j=n_2+1} \frac{z-\alpha_j}{w-\alpha_j} \prod^{n_2}_{l=1} \frac{z+M'-l+1}{w+M'-l+1},
\end{multline}
and then
\begin{multline} \label{eq:formula_of_Psi_Phi_telescope}
  \sum^{n_2}_{k=1} \Psi^{n_1}_{n_1-k}(x) \Phi^{n_2}_{n_2-k}(y) = \frac{x^{\alpha_{n_1}} y^{-\alpha_{n_2}}}{(2\pi i)^2} \left[ \oint_{\Sigma}dz \oint_{\Gamma_{\alpha}}dw \frac{x^{-z-1}y^w}{z-w} \prod^{M'}_{k=1} \frac{w+k}{z+k} \frac{\prod^{n_1}_{j=1} (z-\alpha_j)}{\prod^{n_2}_{l=1} (w-\alpha_l)}
  \vphantom{\prod^{M'-n_2}_{k=1} \frac{w+k}{z+k} \frac{\prod^{n_1}_{j=1} (z-\alpha_j)}{\prod^{n_2}_{j=1} (z-\alpha_j)}} \right. \\
  - \left. \oint_{\Sigma}dz \oint_{\Gamma_{\alpha}}dw \frac{x^{-z-1}y^w}{z-w} \prod^{M'-n_2}_{k=1} \frac{w+k}{z+k} \frac{\prod^{n_1}_{j=1} (z-\alpha_j)}{\prod^{n_2}_{j=1} (z-\alpha_j)} \right]
\end{multline}
Like \eqref{eq:the_2nd_part_in_telescope_vanishes_Laguerre}, by integrating over $\Gamma_{\alpha}$ we have
\begin{equation} \label{eq:partII_n1>=n_2_JUE}
  \oint_{\Sigma}dz \oint_{\Gamma_{\alpha}}dw \frac{x^{-z-1}y^w}{z-w} \prod^{M'-n_2}_{k=1} \frac{w+k}{z+k} \frac{\prod^{n_1}_{j=1} (z-\alpha_j)}{\prod^{n_2}_{j=1} (z-\alpha_j)} = \oint_{\Sigma} 0 \cdot dz = 0.
\end{equation}
Hence the formulas \eqref{eq:formula_of_Psi_Phi_telescope} for $\sum^{n_2}_{k=1} \Psi^{n_1}_{n_1-k}(x) \Phi^{n_2}_{n_2-k}(y)$ and \eqref{eq:partII_n1>=n_2_JUE} and \eqref{eq:evaluation_of_phi^n1n2_JUE} for $\phi^{(n_1, n_2)}(x,y)$ yield the double contour representation of the correlation kernel
\begin{multline}
  K(n_1,x;n_2,y) = \frac{x^{\alpha_{n_1}}}{y^{\alpha_{n_2}}} \left[ \frac{-1}{2\pi i} \oint_{\Gamma_{\alpha}} \frac{x^{-w-1}y^w}{\prod^{n_2}_{l=n_1+1} (w-\alpha_l)} dw \id_{x < y} \id_{n_1 <  n_2} \right. \\
  + \left. \frac{1}{(2\pi i)^2} \oint_{\Sigma}dz \oint_{\Gamma_{\alpha}}dw \frac{x^{-z-1}y^w}{(z-w)} \prod^{M'}_{k=1} \frac{w+k}{z+k} \frac{\prod^{n_1}_{j=1} (z-\alpha_j)}{\prod^{n_2}_{l=1} (w-\alpha_l)} \right],
\end{multline}
which yields Theorem \ref{thm:kernel}\ref{enu:thm:kernel_Jacobi} after conjugating out $x^{\alpha_{n_1}} / y^{\alpha_{n_2}}$.

\begin{proof}[Proof of \eqref{eq:evaluation_of_Psi^n_j_JUE}]
  This proof is parallel to that of \eqref{eq:evaluation_of_Psi^n_j_LUE}. From \eqref{eq:P_N_in_JUE}, \eqref{eq:evaluation_of_phi^n1n2_JUE} and \eqref{eq:defn_of_Psi^n_j}, we have
  \begin{equation}
    \begin{split}
      \Psi^n_j(x) = {}& \int^1_0 \phi^{(n, N)}(x,y) \Psi^N_{N-n+j}(y) dy \\
      = {}& \frac{1}{(2\pi i)^2} \int^1_x \oint_{\Gamma_{\alpha}} \frac{x^{\alpha_n - z - 1}y^z}{\prod^N_{l=n+1} (z - \alpha_l)} dz \oint_{\Sigma'} \frac{y^{-w-1} \Gamma(w+1) \prod^N_{k=n-j+1} (w - \alpha_k)}{\Gamma(w + j + 2 + M' - n)} dw dy.
    \end{split}
  \end{equation}
  Here the contour $\Gamma_{\alpha}$ is defined as in \eqref{eq:evaluation_of_phi^n1n2_JUE}, and the contour $\Sigma'$ encloses the poles $-1, \dotsc, -(j + 1+ M' - n)$ such that $\Re w < \Re z$ for all $w \in \Sigma'$ and $z \in \Gamma_{\alpha}$. Then using the decomposition $\int^1_x = \int^1_0 - \int^x_0$, we express $\Psi^n_j(x) = x^{\alpha_n} (\PartI - \PartII)$, such that  
\begin{equation} \label{eq:partI_in_Psi_JUE}
  \begin{split}
    \PartI = {}& \int^1_0 \frac{1}{(2\pi i)^2} \oint_{\Gamma_{\alpha}} \frac{x^{-z-1} y^z}{\prod^N_{l=n+1} (z - \alpha_l)} dz \oint_{\Sigma'} \frac{y^{-w-1} \Gamma(w+1) \prod^N_{k=n-j+1} (w - \alpha_k)}{\Gamma(w + j + 2 + M' - n)} dw dy \\
    = {}&\frac{\prod^{N-n+j}_{k=1} (M'-n+j+1+\alpha_{N-k+1})}{(M'-n+j)!} \\
    & \times \frac{1}{2\pi i} \oint_{\Gamma_{\alpha}} \frac{x^{-z-1}}{\prod^N_{l=n+1} (z - \alpha_l)} \int^1_0 y^z (1 - y)^{M' - N} P_{(\alpha_N, \alpha_{N-1}, \dots, \alpha_{n-j+1}; M' - N)}(y) dy dz,
  \end{split}
\end{equation}
where the second identity is the consequence of \eqref{eq:P_N_in_JUE}, and
\begin{equation} \label{eq:defn_of_partII_Jacobi}
  \begin{split}
    \PartII = {}& \int^x_0 \frac{1}{(2\pi i)^2} \oint_{\Gamma_{\alpha}} \frac{x^{-z-1} y^z}{\prod^N_{l=n+1} (z - \alpha_l)} dz \oint_{\Sigma'} \frac{y^{-w-1} \Gamma(w+1) \prod^N_{k=n-j+1} (w - \alpha_k)}{\Gamma(w + j + 2 + M' - n)} dw dy \\
    = {}& \frac{1}{(2\pi i)^2} \oint_{\Gamma_{\alpha}} dz \oint_{\Sigma'} dw \frac{x^{-z-1} \Gamma(w+1)}{\Gamma(w + j + 2 + M' - n)} \frac{\prod^N_{k=n-j+1} (w - \alpha_k)}{\prod^N_{l=n+1} (z - \alpha_l)} \int^x_0 y^{z-w-1} dy \\
  = {}& \frac{1}{(2\pi i)^2} \oint_{\Gamma_{\alpha}} dz \oint_{\Sigma'} dw \frac{x^{-w-1}\Gamma(w+1)}{\Gamma(w + j + 2 + M' - n)} \frac{\prod^N_{k=n-j+1} (w - \alpha_k)}{\prod^N_{l=n+1} (z - \alpha_l)} \frac{1}{z-w}.
  \end{split}
\end{equation}
Note that because $P_{(\alpha_N, \alpha_{N-1}, \dots, \alpha_{n-j+1}; M' - N)}(y)$ is a monic polynomial of degree $N - n + j$,
\begin{equation}
  \int^1_0 y^z (1 - y)^{M' - N} P_{(\alpha_N, \alpha_{N-1}, \dots, \alpha_{n-j+1}; M' - N)}(y) dy = \frac{p(z) \Gamma(z + 1)}{\Gamma(z + j + 2 + M' - n)},
\end{equation}
where $p(z)$ is a polynomial of degree $M' - n + j$, and the residue of $p(z) \Gamma(z + 1)/\Gamma(z + M' - n + j + 2)$ at $-(j + 1 + M' - n)$ is $1$. On the other hand, because of the orthogonal property of $P_{(\alpha_N, \alpha_{N-1}, \dots, \alpha_{n-j+1}; M' - N)}(y)$, the integral vanishes as $z = \alpha_N, \dotsc, \alpha_{n-j+1}$. Hence we find that similar to \eqref{eq:simplified_Part_I_LUE}
\begin{equation}
  \PartI = \frac{1}{2\pi i} \oint_{\Gamma_{\alpha}} \frac{x^{-z-1}\Gamma(z+1)}{\Gamma(z + j + 2 + M' - n)} \frac{\prod^N_{k=n-j+1} (z-\alpha_k)}{\prod^N_{l=n+1} (z - \alpha_l)} dz.
\end{equation}
In $\PartII$, if we integrate $z$ first, by calculation of residues we have like \eqref{eq:simplified_part_II_LUE}
\begin{equation}
  \begin{split}
    \PartII ={}& \frac{1}{2\pi i} \oint_{\Sigma'} \frac{x^{-w-1} \Gamma(w+1)}{\Gamma(w + j + 2 + M' - n)} \prod^N_{k=n-j+1} (w - \alpha_k) \left( \sum^N_{l = n+1} \res_{z = \alpha_l} \frac{1}{(z - w) \prod^N_{l=n+1} (z - \alpha_l)} \right) \\
    ={}& \frac{1}{2\pi i} \oint_{\Sigma'} \frac{x^{-w-1} \Gamma(w+1)}{\Gamma(w + j + 2 + M' - n)} \prod^N_{k=n-j+1} (w - \alpha_k) \left(-\res_{z = w} \frac{1}{(z - w) \prod^N_{l=n+1} (z - \alpha_l)} \right) \\
    ={}& \frac{-1}{2\pi i} \oint_{\Sigma'} \frac{x^{-w - 1}\Gamma(w+1)}{\Gamma(w + j + 2 + M' - n)} \frac{\prod^N_{k=n-j+1} (w - \alpha_k)}{\prod^N_{l=n+1} (w - \alpha_l)} dw,
  \end{split}
\end{equation}
Since as integration contours, $\Gamma_{\alpha} + \Sigma' = \Sigma$ where $\Sigma$ is the contour in \eqref{eq:evaluation_of_Psi^n_j_JUE}, we have
\begin{equation}
  \PartI - \PartII = \frac{1}{2\pi i} \oint_{\Sigma} \frac{x^{-w - 1}\Gamma(w+1)}{\Gamma(w + j + 2 + M' - n)} \frac{\prod^N_{k=n-j+1} (w - \alpha_k)}{\prod^N_{l=n+1} (w - \alpha_l)} dw,
\end{equation}
and we prove \eqref{eq:evaluation_of_Psi^n_j_JUE}.
\end{proof}
\section{Asymptotic Pearcey process} \label{sec:Pearcey}

\subsection{Proof of Theorem \ref{thm:Pearcey_kernel}}

For the multiple Laguerre minor process defined by parameters $\alpha_i$ specified in \eqref{eq:alpha_i_trapezoidal}, the correlation kernel for the distribution of the eigenvalues of $S_{n_1}, S_{n_2}$, the minor of $X^* X$, is given by (suppose $n_1, n_2 > n$)
\begin{multline} \label{eq:kernel_with_2_alphas_Laguerre}
  K(n_1, x; n_2, y) = \frac{-1}{2\pi i} \oint_{\Gamma_{\lfloor na \rfloor, \lfloor nb \rfloor}} \frac{x^{-w - 1} y^w}{(w - \lfloor nb \rfloor)^{n_2 - n_1}} dw \id_{x < y} \id_{n_1 < n_2} \\
  + \frac{1}{(2\pi i)^2} \oint_{\Sigma_{\out}} dz \oint_{\Gamma_{\lfloor na \rfloor, \lfloor nb \rfloor}} dw \frac{x^{-z - 1} y^w \Gamma(z + 1)}{(z - w) \Gamma(w + 1)} \left( \frac{z - \lfloor na \rfloor}{w - \lfloor na \rfloor} \right)^n \frac{(z - \lfloor nb \rfloor)^{n_1 - n}}{(w - \lfloor nb \rfloor)^{n_2 - n}},
\end{multline}
where $\Gamma_{\lfloor na \rfloor, \lfloor nb \rfloor}$ is a contour enclosing the poles $\lfloor na \rfloor$ and $\lfloor nb \rfloor$, and $\Sigma_{\out}$ is a deformed Hankel contour as the $\Sigma$ in \eqref{eq:compact_integral_formula_of_K_LUE} that encloses $\Gamma_{\lfloor na \rfloor, \lfloor nb \rfloor}$.

For the saddle point analysis later in this section, we need to deform the contour $\Gamma_{\lfloor na \rfloor, \lfloor nb \rfloor}$ into the sum of two disjoint contour $\Gamma_{\lfloor na \rfloor}$ and $\Gamma_{\lfloor nb \rfloor}$ that enclose $\lfloor na \rfloor$ and $\lfloor nb \rfloor$ respectively. Then we use the contour $\Sigma_{\midd}$ instead of $\Sigma_{\out}$, where $\Sigma_{\midd}$ is, a deformed Hankel contour similar to $\Sigma_{\out}$ enclosing $\Gamma_{\lfloor na \rfloor}$ but not $\Gamma_{\lfloor nb \rfloor}$.

A simple calculation of residue yields
\begin{multline} \label{eq:kernel_with_Sigma_mid}
  \frac{1}{(2\pi i)^2} \oint_{\Sigma_{\out} - \Sigma_{\midd}} dz \oint_{\Gamma_{\lfloor na \rfloor} \cup \Gamma_{\lfloor nb \rfloor}} dw \frac{x^{-z - 1} y^w \Gamma(z + 1)}{(z - w) \Gamma(w + 1)} \left( \frac{z - \lfloor na \rfloor}{w - \lfloor na \rfloor} \right)^n \frac{(z - \lfloor nb \rfloor)^{n_1 - n}}{(w - \lfloor nb \rfloor)^{n_2 - n}} \\
  = \frac{1}{2\pi i} \oint_{\Gamma_{\lfloor nb \rfloor}} \frac{x^{-w - 1} y^w}{(w - \lfloor nb \rfloor)^{n_2 - n_1}} dw.
\end{multline}

From \eqref{eq:kernel_with_2_alphas_Laguerre} and \eqref{eq:kernel_with_Sigma_mid}, we have that
\begin{equation} \label{eq:kernel_degenerate_without_scaling_Laguerre}
  K(n_1, x; n_2, y) = K^{(1)}(n_1, x; n_2, y) + K^{(2)}(n_1, x; n_2, y),
\end{equation}
where
\begin{align}
  K^{(1)}(n_1, x; n_2, y) = {}&
  \begin{cases}
    {\displaystyle \frac{-1}{2\pi i} \oint_{\Gamma_{\lfloor na \rfloor}} \frac{x^{-w - 1} y^w}{(w - \lfloor nb \rfloor)^{n_2 - n_1}} dw \id_{n_1 < n_2}} = 0 & \text{if $x < y$}, \\
    {\displaystyle \frac{1}{2\pi i} \oint_{\Gamma_{\lfloor nb \rfloor}} \frac{x^{-w - 1} y^w}{(w - \lfloor nb \rfloor)^{n_2 - n_1}} dw \id_{n_1 < n_2}} & \text{if $x \geq y$},
  \end{cases} \\
  K^{(2)}(n_1, x; n_2, y) = {}& \frac{1}{(2\pi i)^2} \oint_{\Sigma_{\midd}} dz \oint_{\Gamma_{\lfloor na \rfloor} \cup \Gamma_{\lfloor nb \rfloor}} dw \frac{x^{-z - 1} y^w \Gamma(z + 1)}{(z - w) \Gamma(w + 1)} \left( \frac{z - \lfloor na \rfloor}{w - \lfloor nb \rfloor} \right)^n \frac{(z - \lfloor nb \rfloor)^{n_1 - n}}{(w - \lfloor nb \rfloor)^{n_2 - n}}.
\end{align}
Although generically $na$ and $nb$ are not both integers, below we use $na$ and $nb$ in place of $\lfloor na \rfloor$ and $\lfloor nb \rfloor$ as though they are integers.  Readers can verify that it does not affect the asymptotic result.

For the limiting Pearcey kernel to appear, we consider the scaling that $n_1$ and $n_2$ are $2n + \bigO(n^{1/2})$, and $x$ and $y$ are $nc + \bigO(n^{1/4})$ where
 $c$ is defined in \eqref{eq:defn_of_c}. After the change of scaling \eqref{eq:rescaling_of_n_1_n_2_x_y_z_w}, we write the kernel of the minor eigenvalue process as (assuming $n_1 = n(2 + c_1 n^{-\frac{1}{2}} s), n_2 = n(2 + c_1 n^{-\frac{1}{2}} t)$)
\begin{equation} \label{eq:kernel_rescaled_and_split}
  K(n_1, x; n_2, y) dy = \frac{c_3 n^{\frac{1}{4}} (-n)^{c_1 n^{\frac{1}{2}} (s - t)}}{(1 + c_3 n^{-\frac{3}{4}} u) c^{-c_1 n^{-\frac{1}{2}} (s - t)}} (K^{(1)}(s, u; t, v) + K^{(2)}(s, u; t, v)) dv,
\end{equation}
where by \eqref{eq:kernel_degenerate_without_scaling_Laguerre} (noting that we change variables $z \mapsto nz, w \mapsto nw$)
\begin{align}
   K^{(1)}(s, u; t, v) = {}& \left\{
  \begin{aligned}
   & {\displaystyle \frac{1}{2 \pi i} \oint_{\Gamma_a} \left( \frac{1 - c_3 n^{-\frac{3}{4}} v}{1 - c_3 n^{-\frac{3}{4}} u} \right)^{nw} [c^w_2 (b - w)]^{c_1 n^{\frac{1}{2}} (s - t)} dw \id_{s < t}}, \\
   & \hphantom{\displaystyle \frac{-1}{2 \pi i} \oint_{\Gamma_a}} \text{if $c^{-c_1 n^{-\frac{1}{2}} s}_2 (1 + c_3 n^{-\frac{3}{4}} u) < c^{-c_1 n^{-\frac{1}{2}} t}_2 (1 + c_3 n^{-\frac{3}{4}} v)$,} \\
   & {\displaystyle \frac{-1}{2 \pi i} \oint_{\Gamma_b} \left( \frac{1 - c_3 n^{-\frac{3}{4}} v}{1 - c_3 n^{-\frac{3}{4}} u} \right)^{nw} [c^w_2 (b - w)]^{c_1 n^{\frac{1}{2}} (s - t)} dw \id_{s < t}}, \\
   & \hphantom{\displaystyle \frac{-1}{2 \pi i} \oint_{\Gamma_a}} \text{if $c^{-c_1 n^{-\frac{1}{2}} s}_2 (1 + c_3 n^{-\frac{3}{4}} u) \geq c^{-c_1 n^{-\frac{1}{2}} t}_2 (1 + c_3 n^{-\frac{3}{4}} v)$,} \\
  \end{aligned} \right.
  \label{eq:rescaled_K_Laguerre:1} \\
  K^{(2)}(s, u; t, v) ={}& \frac{1}{(2 \pi i)^2} \oint_{\Sigma} dz \oint_{\Gamma_a \cup \Gamma_b} dw \frac{\left( 1 - c_3 n^{-\frac{3}{4}} v \right)^{nw}}{\left( 1 - c_3 n^{-\frac{3}{4}} u \right)^{nz}} \frac{[c^z_2 (b - z)]^{c_1 n^{\frac{1}{2}} s}}{[c^w_2 (b - w)]^{c_1 n^{\frac{1}{2}} t}} \frac{e^{F_n(z)}}{e^{F_n(w)}} \frac{1}{(w - z)}. \label{eq:rescaled_K_Laguerre:2} 
\end{align}
Here $\Gamma_a$ and $\Gamma_b$ are contours enclosing $a$ and $b$ respectively, $\Sigma$ is a deformed Hankel contour that encloses $\Gamma_a$ but not $\Gamma_b$, and
\begin{equation}
  F_n(z) = -\log(cn)n z + \log \Gamma(nz + 1) + n\log(z - a) + n\log (b - z).
\end{equation}

Now we apply the saddle point method to the double contour integral $K^{(2)}(s, u; t, v)$ in \eqref{eq:rescaled_K_Laguerre:2}. Note that by Stirling's formula, for any complex number $z$ satisfying $\lvert z \rvert > \epsilon$ and $\lvert \arg z \rvert < \pi - \epsilon$,
\begin{equation} \label{eq:relation_between_F_n_and_F}
  F_n(z) = nF(z) + \frac{1}{2} \log(\pi nz) + \bigO(n^{-1}),
\end{equation}
where (taking the principal branch of logarithm)
\begin{equation} \label{eq:formula_of_F}
  F(z) = z (\log z - \log c - 1) + \log (z - a) + \log(z - b).
\end{equation}

By the definition formulas \eqref{eq:defn_of_x_0} and \eqref{eq:defn_of_c}, we see that the derivatives
\begin{align}
  F'(z) ={}& \log z - \log c + \frac{1}{z - a} + \frac{1}{z - b}, \label{eq:x_0_zero_of_F'} \\
  F''(z) ={}& \frac{1}{z} - \frac{1}{(z - a)^2} - \frac{1}{(z - b)^2}, \label{eq:x_0_zero_of_F''} \\
  F'''(z) ={}& -\frac{1}{z^2} + \frac{2}{(z - a)^3} + \frac{2}{(z - b)^3} \label{eq:x_0_zero_of_F'''}
\end{align}
vanish simultaneously at $z = x_0$. Thus around $x_0$, $F(z) = F(x_0) + \frac{1}{24} F^{(4)}(x_0)(z - x_0)^4 + \bigO((z - x_0)^5)$. We also have that $F^{(4)}(x_0) < 0$. To see this, we denote $u_0 = x_0(x_0 - b)^{-1}$ and $v_0 = x_0(x_0 - a)^{-1}$. Then the fact that $x_0$ is the zero of both \eqref{eq:x_0_zero_of_F''} and \eqref{eq:x_0_zero_of_F'''} implies that
\begin{equation} \label{eq:relation_between_u_and_v}
  u_0^2 + v_0^2 = x_0 = 2(u_0^3 + v_0^3).
\end{equation}
Let
\begin{equation} \label{eq:defn_of_r}
r = u_0/v_0.
\end{equation}

Then \eqref{eq:relation_between_u_and_v} yields
\begin{equation} \label{eq:expression_of_u_v_x_0}
  u_0 = \frac{r(1 + r^2)}{2(1 + r^3)}, \quad v_0 = \frac{1 + r^2}{2(1 + r^3)}, \quad x_0 = \frac{(1 + r^2)^3}{4(1 + r^3)^2},
\end{equation}
and further
\begin{equation} \label{eq:expression_of_a_b_in_r}
  \begin{gathered}
    x_0 - a = \frac{x_0}{v_0} = \frac{(1 + r^2)^2}{2(1 + r^3)}, \quad x_0 - b = \frac{x_0}{u_0} = \frac{(1 + r^2)^2}{2r(1 + r^3)}, \\
    a = \frac{(1 + r^2)^2}{4(1 + r^3)^2} (r^2 - 2r^3 - 1), \quad b = \frac{(1 + r^2)^2}{4r(1 + r^3)^2} (r - r^3 - 2).
  \end{gathered}
\end{equation}
By the relation $a < x_0 < b$, we have $u_0 < 0$ and $v_0 > 0$, and then $r = u_0/v_0 < 0$; from the relation $x_0 > 0$ and its expression in $r$ in \eqref{eq:expression_of_u_v_x_0}, we find that $r > -1$; and further from the property $a > 0$ and the expression of $a$ in \eqref{eq:expression_of_a_b_in_r}, we find that $r \in (-1, r_0 \approx -0.657)$. Now we can express $F^{(4)}(x_0)$ as
\begin{equation}
  F^{(4)}(x_0) = \frac{2}{x^3_0} - \frac{6}{(x_0 - a)^4} - \frac{6}{(x_0 - a)^4} = \frac{16(1 + r^3)^4}{(1 + r^2)^9} (r^6 - 3r^4 + 8r^3 - 3r^2 + 1) < 0,
\end{equation}
where the inequality holds for $r \in (-1, -0.4) \supset (-1, r_0)$ by numerical computation.

From the formula \eqref{eq:expression_of_a_b_in_r}, we see that if any one of the three points $a, x_0, b$ is fixed, then $r$ is determined and so are the other two points. As $r \to r_0$, $a \to 0$, $x_0 \to 1 + r^2_0$ and $b \to (1 - r^{-1}_0)(1 + r^2_0)$. When $r$ decreases, $a, x_0, b$ all increase. If $r$ is close to $-1$ such that $r = -1 + \epsilon$, then
\begin{equation}
  x_0 = \frac{2}{9}\epsilon^{-2} + \bigO(\epsilon^{-1}), \quad x_0 - a = \frac{2}{3}\epsilon^{-1} + \bigO(1), \quad b - x_0 = \frac{2}{3}\epsilon^{-1} + \bigO(1).
\end{equation}

To apply the steepest-descent method to \eqref{eq:rescaled_K_Laguerre:2}, we need the contours $ \Gamma_a, \Gamma_b, \Sigma$ to satisfy the following conditions

\paragraph{Conditions satisfied by contours $\Gamma_a, \Gamma_b, \Sigma$}

\begin{enumerate}
\item
  The contours are of the shapes largely depicted in Figure \ref{fig:Sigma_and_Gamma}.
\item
  The point $x_0$ is the maximum of $\Re F(x)$ for $z \in \Sigma$ and the minimum for $z \in \Gamma_a$ and $z \in \Gamma_b$. As $z$ moves to $\infty$ along any direction of $\Sigma$, $\Re F(x) \to -\infty$.
\item
  Denoting the contours $\Sigma$, $\Gamma_b$ and $\Gamma_a$ near $x_0$ by $\Sigma^{\ess}$, $\Gamma^{\ess}_b$ and $\Gamma^{\ess}_a$ respectively, such that (see Figure \ref{fig:essential_parts})
  \begin{equation}
    X^{\ess} = \{ z \in X \mid \lvert z - x_0 \rvert < c^{-1}_3 n^{-\frac{1}{5}} \}, \quad \text{where $X$ stands for $\Sigma$, $\Gamma_b$ or $\Gamma_a$,}
  \end{equation}
  we require that the shape of $X^{\ess}$ is depicted as in Figure \ref{fig:essential_parts}.
\end{enumerate}
 Below we are going to construct contours as described above.

\begin{figure}[ht]
  \begin{minipage}[t]{0.45\linewidth}
    \centering
    \includegraphics{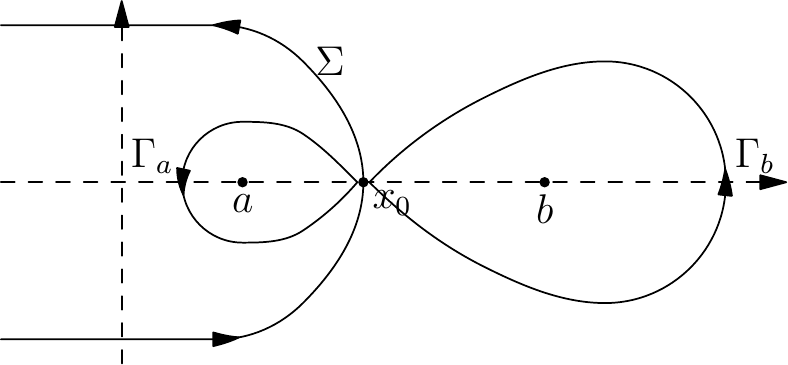}
    \caption{The schematic shape of $\Sigma$ and $\Gamma_{a, b} = \Gamma_a \cup \Gamma_b$.}
    \label{fig:Sigma_and_Gamma}
  \end{minipage}
  \hspace{\stretch{1}}
  \begin{minipage}[t]{0.45\linewidth}
    \centering
    \includegraphics{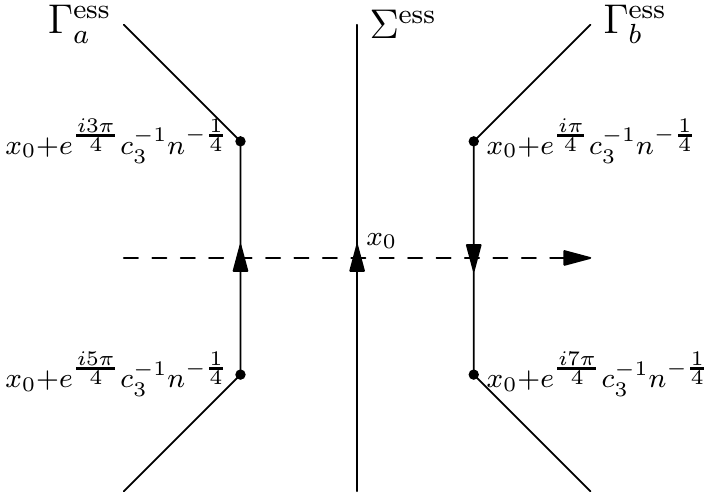}
    \caption{The central part of $\Gamma^{\ess}_a$, $\Sigma^{\ess}$ and $\Gamma^{\ess}_b$ around $x_0$ that are the left, the middle and the right respectively. The contours extend to the ends having distance $c^{-1}_3 n^{-\frac{1}{5}}$ to $x_0$.}
    \label{fig:essential_parts}
  \end{minipage}
\end{figure}

The explicit construction of the contours turns out to be tricky. We relegate the construction to Appendix \ref{sec:constr_of_contours}, and assume the properties above for the contours in the remaining part of this section.

Below we do the steepest-descent analysis of \eqref{eq:rescaled_K_Laguerre:2} around $x_0$. First we specify the values of the constants $c_1, c_2, c_3$ appearing in \eqref{eq:rescaling_of_n_1_n_2_x_y_z_w} as
\begin{equation} \label{eq:formulas_of_c_123}
  c_2 = e^{\frac{1}{b - x_0}}, \quad c_3 = 6^{-\frac{1}{4}} F^{(4)}(x_0)^{\frac{1}{4}}, \quad c_1 = (b - x_0)^2 c^2_3.
\end{equation}
For the asymptotic analysis, we write
\begin{equation} \label{eq:K^2_divided_into_ess_and_remainder}
  K^{(2)}(s, u; t, v) = K^{(2)}_{\ess}(s, u; t, v) + K^{(2)}_{\remainder}(s, u; t, v),
\end{equation}
where $K^{(2)}_{\ess}(s, u; t, v)$ is expressed by the integral formula on the right-hand side of \eqref{eq:rescaled_K_Laguerre:2} with the contours $\Gamma_a$, $\Gamma_b$ and $\Sigma$ replaced by $\Gamma^{\ess}_a$, $\Gamma^{\ess}_b$ and $\Sigma^{\ess}$ respectively, and $K^{(2)}_{\remainder}(s, u; t, v)$ is expressed by the same integral formula with $(z, w) \in \Sigma \times (\Gamma_a \cup \Gamma_b) \setminus \Sigma^{\ess} \times (\Gamma^{\ess}_a \cup \Gamma^{\ess}_b)$. 

Taking the substitution
\begin{equation} \label{eq:change_of_variable_w_z_again}
  w \mapsto x_0 + c^{-1}_3 n^{-\frac{1}{4}} w, \quad z \mapsto x_0 + c^{-1}_3 n^{-\frac{1}{4}} z,
\end{equation}
such that the contours $\Sigma^{\ess}$ and $\Gamma^{\ess}_a$, $\Gamma^{\ess}_b$ are scaled into $\tilde{\Sigma}$ and $\tilde{\Gamma}_a$, $\tilde{\Gamma}_b$ respectively (see Figure \ref{fig:essential_parts_scaled}), we write
\begin{equation} \label{eq:change_of_variable_K^2_ess}
  K^{(2)}_{\ess}(s, u; t, v) = c^{-1}_3 n^{-\frac{1}{4}} \left( \frac{1 - c_3 n^{-\frac{3}{4}} v}{1 - c_3 n^{-\frac{3}{4}} u} \right)^{n x_0} [c^{x_0}_2 (b - x_0)]^{c_1 n^{\frac{1}{2}} (s - t)} \tilde{K}^{(2)}_{\ess}(s, u; t, v)),
\end{equation}
where
\begin{multline} \label{eq:double_countour_defn_of_K^2_ess}
  \tilde{K}^{(2)}_{\ess}(s, u; t, v) = \frac{1}{(2 \pi i)} \oint_{\tilde{\Sigma}} dz \oint_{\tilde{\Gamma}} dw \frac{(1 - c_3 n^{-\frac{3}{4}} v)^{c^{-1}_3 n^{\frac{3}{4}} w}}{(1 - c_3 n^{-\frac{3}{4}} u)^{c^{-1}_3 n^{\frac{3}{4}} z}} \\
  \frac{\left[ c^{c^{-1}_3 n^{-\frac{1}{4}} z}_2 \left( 1 - \frac{1}{b - x_0} c^{-1}_3 n^{-\frac{1}{4}} z \right) \right]^{c_1 n^{\frac{1}{2}} s}}{\left[ c^{c^{-1}_3 n^{-\frac{1}{4}} w}_2 \left( 1 - \frac{1}{b - x_0} c^{-1}_3 n^{-\frac{1}{4}} w \right) \right]^{c_1 n^{\frac{1}{2}} t}} \frac{e^{F_n(x_0 + c^{-1}_3 n^{-\frac{1}{4}} z)}}{e^{F_n(x_0 + c^{-1}_3 n^{-\frac{1}{4}} w)}} \frac{1}{w - z}.
\end{multline}

For $\lvert w \rvert \leq n^{\frac{1}{20}}$, we have
\begin{equation} \label{eq:approx_leant_term_Pearcey}
  \left( 1 - c_3 n^{-\frac{3}{4}} v \right)^{c^{-1}_3 n^{\frac{3}{4}} w} = e^{-vw} e^{\bigO(n^{-\frac{11}{20}} w)}
\end{equation}
and
\begin{equation} \label{eq:approx_quadartic_term_Pearcey}
  \begin{split}
    \left[ c^{c^{-1}_3 n^{-\frac{1}{4}} w}_2 \left( 1 - \frac{1}{b - x_0} c^{-1}_3 n^{-\frac{1}{4}} w \right) \right]^{c_1 n^{\frac{1}{2}} t} ={}& \left[ 1 - \frac{1}{2} c^{-1}_1 n^{-\frac{1}{2}} w^2 + \bigO(n^{-\frac{3}{4}} w^3) \right]^{c_1 n^{\frac{1}{2}} t} \\
    ={}& e^{-\frac{w^2}{2} t} e^{\bigO(n^{-\frac{1}{4}} w^3)}.
  \end{split}
\end{equation}

Using approximations \eqref{eq:approx_leant_term_Pearcey} and \eqref{eq:approx_quadartic_term_Pearcey} for $w \in \tilde{\Gamma}$ and $z \in \tilde{\Sigma}$, and noting that $F_n(x_0 + c^{-1}_3 n^{-\frac{1}{4}} w)$ is approximated by $nF(x_0 + c^{-1}_3 n^{-\frac{1}{4}} w)$ as in \eqref{eq:relation_between_F_n_and_F} and
\begin{equation}
  \begin{split}
    F(x_0 + c^{-1}_3 n^{-\frac{1}{4}} w) ={}& F(x_0) + \frac{1}{24} F^{(4)}(x_0)(c^{-1}_3 n^{-\frac{1}{4}} z)^4 + \bigO((c^{-1}_3 n^{-\frac{1}{4}} z)^5) \\
    ={}& F(x_0) + \frac{1}{4n} z^4 + \bigO(n^{-\frac{5}{4}} z^5),
  \end{split}
\end{equation}
we have
\begin{equation} \label{eq:final_approx_of_K^2}
  \begin{split}
    \tilde{K}^{(2)}_{\ess}(s, u; t, v) ={}& \frac{1}{(2 \pi i)} \oint_{\tilde{\Sigma}} dz \oint_{\tilde{\Gamma}} dw \frac{e^{uz}}{e^{vw}} \frac{e^{-\frac{s z^2}{2}}}{e^{-\frac{t w^2}{2}}} \frac{e^{\frac{z^4}{4}}}{e^{\frac{w^4}{4}}} \frac{1}{z - w} \frac{e^{\bigO(n^{-\frac{11}{20}} w)}}{e^{\bigO(n^{-\frac{11}{20}} z)}} \frac{e^{\bigO(n^{-\frac{1}{4}} z^3)}}{e^{\bigO(n^{-\frac{1}{4}} w^3)}} \frac{e^{\bigO(n^{-\frac{1}{4} z^5})}}{e^{\bigO(n^{-\frac{1}{4} w^5})}} \\
  ={}& \frac{1}{(2 \pi i)} \oint_{\tilde{\Sigma}^{\infty}} dz \oint_{\tilde{\Gamma}^{\infty}} dw  \frac{e^{-\frac{w^2}{4} + \frac{t w^2}{2} - vw}}{e^{-\frac{z^4}{4} + \frac{s z^2}{2} - ut}} \frac{1}{w - z} (1 + \bigO(n^{-\frac{1}{5}})).
  \end{split}
\end{equation}

Similar to \eqref{eq:change_of_variable_K^2_ess}, we take the change of variable \eqref{eq:change_of_variable_w_z_again} and write
\begin{equation} \label{eq:transform_of_K^2_remainder}
  K^{(2)}_{\remainder}(s, u; t, v) = c_3 n^{\frac{1}{4}} \left( \frac{1 - c_3 n^{-\frac{3}{4}} v}{1 - c_3 n^{-\frac{3}{4}} u} \right)^{n x_0} [c^{x_0}_2 (b - x_0)]^{c_1 n^{\frac{1}{2}} (s - t)} \tilde{K}^{(2)}_{\remainder}(s, u; t, v)),
\end{equation}
where $\tilde{K}^{(2)}_{\remainder}(s, u; t, v))$ is defined by a double contour integral formula like $\tilde{K}^{(2)}_{\ess}(s, u; t, v))$ in \eqref{eq:double_countour_defn_of_K^2_ess}, with the same integrand, but the integral is taken on the contour $(z, w) \in c_3 n^{\frac{1}{4}} (\Sigma - x_0) \times c_3 n^{\frac{1}{4}} (\Gamma_{a, b} - x_0) \setminus \tilde{\Sigma} \times \tilde{\Gamma}$. Recall that $\Re F(z)$ attains its unique minimum at $x_0$ on $\Sigma$ and attains its unique maximum at $\Gamma_{a, b}$. It is not difficult to check that on for $(z, w)$ on the contour for $\tilde{K}^{(2)}_{\remainder}(s, u; t, v)$, the factor $e^{F_n(x_0 + c^{-1}_3 n^{-\frac{1}{4}} z)}/e^{F_n(x_0 + c^{-1}_3 n^{-\frac{1}{4}} w)}$ dominates the other terms of the integrand and make double contour integral $\tilde{K}^{(2)}_{\remainder}(s, u; t, v)$ vanishing as $n \to \infty$. 

At last we consider $K^{(1)}(s, u; t, v)$ in \eqref{eq:rescaled_K_Laguerre:1}, the other part of the kernel $K(s, u; t, v)$ in \eqref{eq:kernel_rescaled_and_split}. Note that if $s \geq t$, then $K^{(1)}(s, u; t, v)$ vanishes, and if $s < t$, since $c_1 > 0$ and $c_2 > 1$ in \eqref{eq:formulas_of_c_123}, we have $c^{-c_1 n^{-\frac{1}{2}} s}_2 (1 + c_3 n^{-\frac{3}{4}} u) \geq c^{-c_1 n^{-\frac{1}{2}} t}_2 (1 + c_3 n^{-\frac{3}{4}} v)$ for all $u, v \in \realR$ and $n$ large enough. Below we assume the condition $s < t$ and $n$ is large enough so that
\begin{equation}
  K^{(1)}(s, u; t, v) = \frac{-1}{2 \pi i} \oint_{\Gamma^*_b} \left( \frac{1 - c_3 n^{-\frac{3}{4}} v}{1 - c_3 n^{-\frac{3}{4}} u} \right)^{nw} [c^w_2 (b - w)]^{c_1 n^{\frac{1}{2}} (s - t)} dw.
\end{equation}
Here the contour $\Gamma^*_b$, like $\Gamma_b$ in \eqref{eq:rescaled_K_Laguerre:1}, encloses the pole $b$. We use a different notation, because we want to deform the contour into a square whose left side is through $x_0$, as in Figure \ref{fig:Gamma_b_star}, so that it is of a different shape from the $\Gamma_b$ for the asymptotic analysis of $K^{(2)}(s, u; t, v)$. Note that $\Sigma_{\ess}$ is part of $\Gamma_b$. It straightforward to check by explicit calculation that the absolute value of the integrand attains its unique maximum on the contour $\Gamma^*_b$ at $x_0$. Then we write similar to \eqref{eq:K^2_divided_into_ess_and_remainder}, \eqref{eq:change_of_variable_K^2_ess} and \eqref{eq:transform_of_K^2_remainder}
\begin{figure}[ht]
  \centering
  \includegraphics{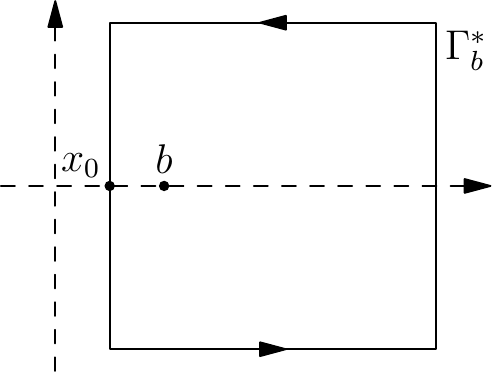}
  \caption{The square contour $\Gamma^*_b$ around $b$ whose left side is through $x_0$.}
  \label{fig:Gamma_b_star}
\end{figure}
\begin{equation} \label{eq:K^1_rescaled_and_distilled}
  \begin{split}
    K^{(1)}(s, u; t, v) = {}& K^{(1)}_{\ess}(s, u; t, v) + K^{(1)}_{\remainder}(s, u; t, v) \\
    = {}& c^{-1}_3 n^{-\frac{1}{4}} \left( \frac{1 - c_3 n^{-\frac{3}{4}} v}{1 - c_3 n^{-\frac{3}{4}} u} \right)^{n x_0} [c^{x_0}_2 (b - x_0)]^{c_1 n^{\frac{1}{2}} (s - t)} \\
    & \times \left(\tilde{K}^{(1)}_{\ess}(s, u; t, v) + \tilde{K}^{(1)}_{\remainder}(s, u; t, v) \right),
  \end{split}
\end{equation}
where (taking the change of variable $w \mapsto x_0 + c^{-1}_3 n^{-\frac{1}{4}} w$ as in \eqref{eq:change_of_variable_w_z_again})
\begin{equation}
  \tilde{K}^{(1)}_{\ess}(s, u; t, v) = \frac{-1}{2 \pi i} \oint_{\Sigma^{\ess}} \left( \frac{1 - c_3 n^{-\frac{3}{4}} v}{1 - c_3 n^{-\frac{3}{4}} u} \right)^{c^{-1}_3 n^{\frac{3}{4}} w} \left[ c^{c^{-1}_3 n^{-\frac{1}{4}} w}_2 \left( 1 - \frac{1}{b - x_0} c^{-1}_3 n^{-\frac{1}{4}} w \right) \right]^{c_1 n^{\frac{1}{2} (s - t)}} dw,
\end{equation}
and $\tilde{K}^{(1)}_{\remainder}(s, u; t, v)$ is a contour integral with the same integrand as $\tilde{K}^{(1)}_{\ess}(s, u; t, v)$ with the integral contour changed into $\Gamma_b \setminus \Sigma^{\ess}$. 

Using the approximations \eqref{eq:approx_leant_term_Pearcey} and \eqref{eq:approx_quadartic_term_Pearcey}, we have (if $s < t$ and $n$ is large enough)
\begin{equation} \label{eq:K^1_final_approx}
  \begin{split}
    \tilde{K}^{(1)}_{\ess}(s, u; t, v) ={}& \frac{1}{2 \pi i} \oint_{\tilde{\Gamma}^*_b} e^{(v - u)w - \frac{w^2}{2}(s - t)} e^{\bigO(n^{-\frac{11}{20}} w)} e^{\bigO(n^{-\frac{1}{4}} w^3)} dw \\
    ={}& \frac{1}{2 \pi i} \int^{i\infty}_{-i\infty} e^{(v - u)w - \frac{w^2}{2}(s - t)} dw (1 + \bigO(n^{-\frac{1}{4}})) \\
    ={}& \frac{1}{\sqrt{2\pi (t - s)}} e^{-\frac{(v - u)^2}{2(t - s)}} (1 + \bigO(n^{-\frac{1}{4}})).
  \end{split}
\end{equation}
The estimate of $\tilde{K}^{(1)}_{\remainder}(s, u; t, v))$ is similar to $\tilde{K}^{(2)}_{\remainder}(s, u; t, v))$, and we can show that it vanishes exponentially as $n \to \infty$.

Collecting the results in \eqref{eq:kernel_rescaled_and_split}, \eqref{eq:K^2_divided_into_ess_and_remainder}, \eqref{eq:change_of_variable_K^2_ess}, \eqref{eq:final_approx_of_K^2}, \eqref{eq:K^1_rescaled_and_distilled}, \eqref{eq:K^1_final_approx} and the vanishing properties of $\tilde{K}^{(2)}_{\remainder}(s, u; t, v)$, $\tilde{K}^{(1)}_{\remainder}(s, u; t, v)$ and $\tilde{K}^{(1)}_{\ess}(s, u; t, v)$ when $s > t$, we prove that as $s, u, t, v$ are fixed and $n \to \infty$, 
\begin{equation}
  \lim_{n \to \infty} (1 + c_3 n^{-\frac{3}{4}} u) c^{-c_1 n^{-\frac{1}{2}} (s - t)} \frac{f_n(s, u)}{f_n(t, v)} K(n_1, x; n_2, y) dy = K^P(s, u; t, v) dv.
\end{equation}
with the conjugation function
\begin{equation} \label{eq:conjugation_Pearcey}
  f_n(s, u) = (1 + c_3 n^{-\frac{3}{4}} u)^{n x_0} [c^{x_0}_2 (x_0 - b)n]^{-c_1 n^{\frac{1}{2}}s}.
\end{equation}
Since we have
\begin{equation}
  \lim_{n \to \infty} (1 + c_3 n^{-\frac{3}{4}} u) c^{-c_1 n^{-\frac{1}{2}} (s - t)} = 1 \quad \text{as $n \to \infty$ and $u$ is fixed,}
\end{equation}
and by using the change of variables \eqref{eq:rescaling_of_n_1_n_2_x_y_z_w} reversely
\begin{equation}
  f_n(s, u) = x^{nx_0} ((x_0 - b)n)^{2n - n_1},
\end{equation}
we prove Theorem \ref{thm:Pearcey_kernel}.

\subsection{Proof of Theorem \ref{thm:Pearcey_added}}

From Theorem \ref{thm:Schur_corrsp}, the sequences
\begin{equation} \label{eq:sequence_added}
  \nu^{(1)}, \dotsc, \nu^{(n)}, \dotsc, \nu^{(N)}, \quad \text{with} \quad \nu^{(n)} = (\nu^{(n)}_1, \dotsc, \nu^{(n)}_n)
\end{equation}
have the same joint distribution as the spectra of the consecutive minors of a $\GUE(N)$ matrix with external source, namely
\begin{equation}
  (M + A)_N, \quad \text{with} \quad M_N = \GUE(N) \quad \text{and} \quad A_N = \diag(\underbrace{\kappa_1, \dotsc, \kappa_1}_{r_1}, \underbrace{\kappa_2, \dotsc, \kappa_2}_{N - r_2}).
\end{equation}
The sequence \eqref{eq:sequence_added} forms an inhomogeneous Markov chain in discrete time given by \eqref{eq:transition_general}, with correlation kernel given by \eqref{eq:compact_integral_formula_of_K}. In \cite{Adler-Orantin-van_Moerbeke10}, we studied $n$ non-intersecting Brownian paths $(x_1(t), \dotsc, x_n(t))$ leaving from the origin at time $t = 0$, with $n - \lfloor pn \rfloor$ particles forced to the point $b\sqrt{n}$ and $\lfloor pn \rfloor$ particles forced to $a\sqrt{n}$ at time $t = 1$. In Theorem 1.1 of \cite{Adler-Orantin-van_Moerbeke10}, it was shown that a bifurcation (cusp) appears at the point $(x_0\sqrt{n}, t_0)$, when $n \to \infty$, and that a Pearcey process arises in the scale $c_0 \mu n^{-1/4}$, with
\begin{equation}
  \begin{split}
    q & \text{ specified by the equation } p = \frac{1}{1 + q^3}, \\
    t_0 {}& = \left( 1 + 2(a - b)^2 \frac{q^2 - q + 1}{(q + 1)^2} \right)^{-1}, \\
    x_0 {}& = \frac{(2a - b)q + (2b - a)}{q + 1} t_0, \\
    c_0 \mu {}& = \sqrt{\frac{t_0(1 - t_0)}{2}} \left( \frac{q^2 - q + 1}{q} \right)^{\frac{1}{4}}.
  \end{split}
\end{equation}
Referring to the notation of Corrollary \ref{cor:polymer_added}, we set
\begin{equation}
  b\sqrt{n} = \kappa_1 = -\sqrt{n}, \quad a\sqrt{n} = \kappa_2 = \sqrt{n}, \quad (1 - p)n = r_1 = \frac{n}{2}, \quad pn = n - r_1 = \frac{n}{2},
\end{equation}
that is, $a = -b = 1$, $p = \frac{1}{2}$, $q = 1$, yielding $x_0 = 0$, $t_0 = \frac{1}{3}$, $c_0 \mu = \frac{1}{3}$. For these values, one checks that
\begin{equation}
  \lim_{n \to \infty} \Prob^{(b\sqrt{n}, a\sqrt{n})}_n \left( \text{all } x_j(t_0) \in x_0\sqrt{n} + \frac{c_0 \mu}{n^{\frac{1}{4}}} E^c \right) = \Prob^P(\mathcal{P}(0) \cap E = \emptyset),
\end{equation}
where $\mathcal{P}(t)$ is the Pearcey process.\footnote{The value $t_0 = \frac{1}{3}$ also derives from the correspondence (see Box 1 in \cite{Adler-Orantin-van_Moerbeke10}) of the Brownian motion problem with the GUE-matrix with external source $M + A$, where the diagonal matrix $A_t = \diag(\kappa_1, \dotsc, \kappa_1, \kappa_2, \dotsc, \kappa_2) \sqrt{2t/(1 - t)}$. One notices $\sqrt{2t/(1 - t)} = 1$ only when $t = \frac{1}{3}$.}

\section{Interpretation in percolation model and Schur process} \label{sec:Schur_process}


In this section we prove Theorem \ref{thm:Schur_corrsp} for the three minor processes that we study. We also give the proof for the Wishart minor process, for completeness as well as a preparation of the GUE with external source case.

Recall the RSK correspondence that gives a bijection between $m \times n$ nonnegative integer matrices $X = [x_{ij}]$ and a pair of semi-standard tableaux. The shape of the pair of Young tableaux, which is a Young diagram, has a bijective relation to the maximal length of $l$ non-intersecting up-right paths $L^{(l)}(m, n)$ defined in \eqref{eq:defn_of_L^l(N,n)} with the weight on each site $(i,j)$ equal to the entry $x_{ij}$ of $X$. A Young diagram can be represented by a sequence of integers in descending order, and only the nonzero terms are meaningful. To be concrete, suppose the shape of the Young diagram is $\kappa^{(m, n)} = (\kappa^{(m, n)}_1, \kappa^{(m, n)}_2, \dotsc)$, then 
\begin{equation} \label{eq:relation_between_mu_and_L}
  \sum^l_{i=1} \kappa^{(m, n)}_i = L^{(l)}(m,n).
\end{equation}
Note that since the Young diagram $\kappa^{(m, n)}$ comes from the $m \times n$ matrix $X$ by RSK correspondence, $\kappa^{(m, n)}_i = 0$ for all $i > \min(m, n)$.

Below we fix $m$ and assume $1 \leq n \leq m$. By definition, $\kappa^{(m, n)}$ has infinitely many component, but since the components $\kappa^{(m, n)}_k = 0$ for $k > n$, we abuse the notation and consider $\kappa^{(m, n)}$ as an $n$-variate vector, as a discrete counterpart of the $\mu^{(m, n)}$ defined in \eqref{eq:defn_mu^mn}. From the relation \eqref{eq:relation_between_mu_and_L} to the last-passage paths, we have the interlacing property $\kappa^{(m, n-1)} \preceq \kappa^{(m, n)}$, that is, (note that as $i$ increases, our $\kappa^{(m, n)}_i$ are decreasing but $\mu^{(m, n)}_i$ are increasing)
\begin{equation} \label{eq:basic_interlacing_condition}
  0 \leq \kappa^{(m, n)}_n \leq \kappa^{(m, n - 1)}_{n - 1} \leq \kappa^{(m, n)}_{n-1} \leq \kappa^{(m, n-1)}_{n-2} \leq \kappa^{(m, n)}_{n-2} \leq \dots \leq \kappa^{(m, n)}_2 \leq \kappa^{(m, n-1)}_1 \leq \kappa^{(m, n)}_1 < \infty.
\end{equation}
In other words, $\{ \kappa^{(m, n)}_k \mid 1 \leq k \leq n \leq N \}$ form a Gelfand-Tsetlin pattern.

Let $s_1, \dotsc, s_m; t_1, \dotsc, t_n$ be positive parameters such that $s_i t_j < 1$ for all $i = 1, \dotsc, m, j = 1, \dotsc, n$ and let $w(i,j)$ be an independent random variable in geometric distribution with parameter $s_i t_j$, so that
\begin{equation} \label{eq:random_variable_w(i,j)}
  \Prob(w(i,j) = k) = (1 - s_i t_j)(s_i t_j)^k.
\end{equation}
If the entries $x_{ij}$ of the $m \times n$ matrix $X$ are random variables $w(i,j)$ defined in \eqref{eq:random_variable_w(i,j)}, the shape of the corresponding pair of the Young tableaux, which is a random Young diagram, follows the Schur measure \cite{Okounkov01}, \ie, the probability that the shape is $\kappa$ is
\begin{equation} \label{eq:Schur_measure_formula}
  P(\kappa) = \prod^m_{i=1} \prod^n_{j=1} (1-s_i t_j) s_{\kappa}(s_1, \dots, s_m) s_{\kappa}(t_1, \dots, t_n).
\end{equation}
Equivalently, if the distribution of the weight $w(i,j)$ at each $(i,j)$ site of the $m \times n$ square lattice is defined by \eqref{eq:random_variable_w(i,j)}, the maximal length of $l$ non-intersecting up-right paths $L^{(l)}(m,n)$ becomes a random variable and is determined by \eqref{eq:Schur_measure_formula} via \eqref{eq:relation_between_mu_and_L}. Note that the maximal length $L^{(l)}(m,n)$ is well defined if the weight $x_{ij}$ are real numbers instead of integers. Then we can define $\kappa^{(m, n)}_i$ conversely by \eqref{eq:relation_between_mu_and_L}. In that case $\kappa^{(m, n)} = (\kappa^{(m, n)}_1, \kappa^{(m, n)}_2, \dotsc)$ is a sequence of real numbers in descending order, but may not be interpreted as a Young diagram (\cf\ $\mu^{(m, n)}$ defined in \eqref{eq:defn_mu^mn} and Remark \ref{rmk:geometric_RSK}.)

If we fix $m$ and let $n$ vary between $1$ and $N$, we have the joint probability distribution of $\kappa^{(m, 1)}, \kappa^{(m, 2)}, \dots, \kappa^{(N, m)}$ as follows.
\begin{prop}[\cite{Dieker-Warren09}] \label{prop:Markov_Schur}
  In the $m \times N$ rectangular lattice $\{ (i,j) \mid  i = 1, \dotsc, m,\ j = 1, \dotsc, N \}$ where $N \leq m$, let the weight at each site be independent random variables given by \eqref{eq:random_variable_w(i,j)}. If $n$ runs from $1$ to $N$, and is regarded as the discrete time, the random Young diagram $\kappa^{(m, n)}$ evolves as an inhomogeneous Markov chain with transition probability
  \begin{multline} \label{eq:evolution_of_GT}
    P_{n-1,n}(\kappa^{(m, n-1)}, \kappa^{(m, n)}) = \\
    \prod^m_{i=1} (1 - s_i t_n) \frac{S_{\kappa^{(m, n)}}(s_1, \dots, s_m)}{S_{\kappa^{(m, n-1)}}(s_1, \dots, s_m)} t^{\sum^n_{i=1}\kappa^{(m, n)}_i - \sum^{n-1}_{i=1}\kappa^{(m, n-1)}_i}_n \id_{\kappa^{(m, n-1)} \preceq \kappa^{(m, n)}},
  \end{multline}
  and the initial state $\kappa^0 = (0)$. 
\end{prop}
In the special case that $q_j = 1$ for all $j$, Proposition \ref{prop:Markov_Schur} is proved in \cite{O'Connell03}. See also \cite{Forrester-Nagao11}.

Choosing special values of $p_i, q_j$ and taking limit, Proposition \ref{prop:Markov_Schur} yields the four cases of Theorem \ref{thm:Schur_corrsp}. Below we show the detail of taking limits case by case.

\paragraph{Wishart}

Let the $m = M$ in Proposition \ref{prop:Markov_Schur}. Set
\begin{equation} \label{eq:a_i_b_j_Wishart_limit}
  p_i = 1, \quad q_j = 1 + a_j/L, \quad \textnormal{for $i = 1, \dotsc, M$, $j = 1, \dotsc, N$, where $a_j < 0$},
\end{equation}
we have
\begin{equation} \label{eq:from_geometrix_to_exponential_Wishart}
  \lim_{L \to \infty} \frac{w(i,j)}{L} = u(i,j),
\end{equation}
where $u(i,j)$ are independent random variables in exponential distribution with parameter $-a_j$, as in the continuous percolation model in the Wishart case of Theorem \ref{thm:Schur_corrsp}.

Note that in this special case \eqref{eq:evolution_of_GT} can be simplified by the evaluation of Schur polynomials \cite[Section 6.1, Exercise 6]{Fulton97}
\begin{equation}
  s_{\kappa}(\underbrace{1, \dots, 1}_M) = \frac{\Delta_M(1-\kappa_1, 2-\kappa_2, \dots, M - \kappa_M)}{\Delta_M(1, 2, \dots, M)},
\end{equation}
if $\kappa_{M+1} = \kappa_{M+2} = \dots = 0$ \footnote{where $\Delta_M(x_1, \dots, x_M) = \prod^M_{i=1} \prod^M_{j=i+1} (x_j-x_i)$}. If we take the limit $L \to \infty$ in \eqref{eq:evolution_of_GT} and consider $\kappa^{(M, n)}_i = \lfloor Lx^{(M, n)}_i \rfloor$ for fixed values of $x^{(M, n)}_i$, we find that the transition probability \eqref{eq:evolution_of_GT} in Proposition \ref{prop:Markov_Schur} satisfies ($\lfloor Lx^{(M, n)} \rfloor = (\lfloor Lx^{(M, n)}_1 \rfloor, \dotsc, \lfloor Lx^{(M, n)}_n \rfloor)$)
\begin{multline} \label{eq:transition_prob_Laguerre_Schur}
  \lim_{L \to \infty} L^n P_{n-1,n}(\lfloor Lx^{(M,n-1)} \rfloor, \lfloor Lx^{(M, n)} \rfloor) = \\
  \frac{(-a_n)^M}{(M-n)!} \frac{\Delta_n(x^{(M, n)}_n, \dots, x^{(M, n)}_1)}{\Delta_{n-1}(x^{(M, n-1)}_{n-1}, \dots, x^{(M, n-1)}_1)} \frac{\prod^n_{i=1} (x^{(M, n)}_i)^{M-n} e^{a_nx^{(M, n)}_i}}{\prod^{n-1}_{i=1} (x^{(M, n-1)}_i)^{M-n+1} e^{a_nx^{(M, n-1)}_i}} \id_{x^{(M, n-1)} \preceq x^{(M, n)}}
\end{multline}
and
\begin{equation} \label{eq:initial_prob_Laguerre_Schur}
  \lim_{L \to \infty} L P(\lfloor Lx^{(M,1)}_1 \rfloor) = (-a_1)^M((M-1)!)^{-1} (x^{(M, 1)}_1)^{M-1}e^{a_1 x^{(M, 1)}_1} \id_{x^{(M, 1)}_1 \geq 0}.
\end{equation}
Now if we consider random variables $\mu^{(M, n)}_{L, i}$ defined by $\kappa^{(M, n)}_i$ as
\begin{equation}
  \kappa^{(M, n)}_i = L\mu^{(M, n)}_{L,i}, \quad \textnormal{for $1 \leq i \leq n \leq N$,}
\end{equation}
then as $L \to \infty$, $\mu^{(M, n)}_{L, i}$ converges in distribution to $\mu^{(M, n)}_i$ that is defined in the continuous percolation model by \eqref{eq:defn_mu^mn} where $u(i,j)$ are independent random variables in exponential distribution with parameter $-a_j$. Then $\mu^{(M, n)}$ constitute an inhomogeneous Markov chain as $n$ runs from $1$ to $N$, and the transition probability $P_{n - 1, n}(\mu^{(M, n - 1)}, \mu^{(M, n)}$ is $\lim_{L \to \infty} L^n P_{n-1,n}(\lfloor L\mu^{(M,n-1)} \rfloor, \lfloor L\mu^{(M, n)} \rfloor)$
given in \eqref{eq:transition_prob_Laguerre_Schur}, and the distribution of $\mu^{(M, 1)}_1$ is $\lim_{L \to \infty} L P(\lfloor \mu^{(M,1)}_1 \rfloor)$ given in \eqref{eq:initial_prob_Laguerre_Schur}. We see that the transition probability of $\mu^{(M, n)}$ coincide with the transitional probability \eqref{eq:transition_general} in the Wishart case, and the probability distribution of $\mu^{(M, 1)}_1$ is the same as the $p_1(\mu^{(M, 1)}_{1})$ in Table \ref{table:thm}. Thus we prove the Wishart case of Theorem \ref{thm:Schur_corrsp}. 

\begin{rmk} \label{rmk:geometric_RSK}
  The quantity $\mu^{(M, n)}_i$ that come from the limit RSK correspondence can also be associated to the continuous RSK correspondence. See \cite{Kirillov01}, \cite{Noumi-Yamada04} and \cite[Appendix A]{Forrester-Rains05} for more details.
\end{rmk}

\paragraph{GUE with external source}

The classical transition between the Laguerre and Gaussian weights gives that if $n$ is finite and $x = M + \sqrt{M}\xi$,
\begin{equation}
  \lim_{M \to \infty} e^{M-a\sqrt{M}} M^{-M+n} e^{-(1 - a/\sqrt{M})x} x^{M-n} = e^{-\frac{\xi^2}{2} + a\xi}.
\end{equation}
Suppose in the $M \times N$ rectangular lattice $\{ (i,j) \mid i = 1, \dots, M, j = 1, \dots, N \}$ ($N \leq M$), the weight at each site is an independent random variable $u(i,j)$ in exponential distribution with parameter $1 - a_j/\sqrt{M}$. Let $\mu^{(M, 1)}, \dotsc, \mu^{(M, N)}$ be defined by \eqref{eq:defn_mu^mn} and for $1 \leq i \leq n \leq N$, we define the random variables $\nu^{(M, n)} = (\nu^{(M, n)}_1, \dotsc, \nu^{(M, n)}_n)$ by
\begin{equation} \label{eq:GUE_percolation_change_of_variables}
  \mu^{(M, n)}_i = M + \sqrt{M} \nu^{(M, n)}_i,
\end{equation}
such that $\sum^l_{i=1} \nu^{(M, n)}_i = M^{-1/2} (L^{(l)}(M,n) - lM)$. From the property of $\mu^{(M, n)}$, we have that as $n$ runs from $1$ to $N$, the sequences $\nu^{(M, n)} = (\nu^{(M, n)}_1, \dots, \nu^{(M, n)}_n)$ evolves as an inhomogeneous Markov chain. Moreover, as $M \to \infty$, $\nu^{(M, n)}$ have a joint limit distribution. To compute the limit, consider the transition probability $P_{n - 1, n}(\mu^{(M, n - 1)}, \mu^{(M, n)})$ \eqref{eq:transition_general} in the Wishart case of the last-passage percolation model with parameters specified by $p_i = 0, q_j = 1 - a_j/\sqrt{M}$, we have, substituting $\mu^{(M, n)}$ into $\nu^{(M, n)}$ by \eqref{eq:GUE_percolation_change_of_variables} and assuming $\nu^{(M, n)} = \nu^{(n)}$ for all $M$,
\begin{multline} \label{eq:limiting_transition_prob_GUE_percolation}
  \lim_{M \to \infty} M^{n/2} P_{n-1,n}(\mu^{(M, n-1)},\mu^{(M, n)}) \\
  = \frac{e^{-\frac{a^2_n}{2}}}{\sqrt{2\pi}} \frac{\Delta_n(\nu^{(n)}_n, \dots, \nu^{(n)}_1)}{\Delta_{n-1}(\nu^{(n-1)}_{n-1}, \dots, \nu^{(n-1)}_1)} \frac{\prod^n_{i=1} e^{-\frac{(\nu^{(n)}_i)^2}{2} + a_n \nu^{(n)}_i}}{\prod^{n-1}_{i=1} e^{-\frac{(\nu^{(n-1)}_i)^2}{2} + a_n \nu^{(n-1)}_i}} \id_{\nu^{(n-1)} \preceq \nu^{(n)}}
\end{multline}
and
\begin{equation} \label{eq:initial_transition_prob_GUE_percolation}
  \lim_{M \to \infty} \sqrt{M} P(\mu^{(M, 1)}_1) = (2\pi)^{-1/2}e^{-a^2_1/2} e^{-\frac{1}{2}(\nu^{(1)}_1 - a_1)^2}.
\end{equation}
A simple scaling argument yields that as $M \to \infty$, the limiting transition probability from $\nu^{(M, n - 1)}$ to $\nu^{(M, n)}$ has the density function $\lim_{M \to \infty} M^{n/2} P_{n-1,n}(\mu^{(M, n-1)},\mu^{(M, n)})$, and the limiting distribution of $\nu^{(M, 1)}_1$ is $\lim_{M \to \infty} \sqrt{M} P(\mu^{(M, 1)}_1)$. We see that the limiting transition probability of $\nu^{(M, n)}$ coincide with the transitional probability \eqref{eq:transition_general} in the GUE with external source minor process and the limiting distribution of $\nu^{(M, 1)}_1$ is the same as the $p_1(z)$ given in Table \ref{table:thm}, and so we prove the GUE with external source case of Theorem \ref{thm:Schur_corrsp}.

  \paragraph{Jacobi-\Pineiro}
Let the $m = M'$ in Proposition \ref{prop:Markov_Schur}. Set 
\begin{equation} \label{Jacobi_sequence_of_weights}
  p_i = t^{i-1}, \quad q_j = t^{\alpha_j + 1}, \quad \textnormal{for $i = 1, \dotsc, M'$, $j = 1, \dotsc, N$,}
\end{equation}
where $t \in (0,1)$ and $\alpha_j \geq 0$. By the Jacobi-Trudi formula of Schur polynomials, for any Young diagram $\kappa = (\kappa_1, \kappa_2, \dots)$, if $\kappa_{M'+1} = \kappa_{M'+2} = \dots = 0$, then \cite[Section 6.1, Exercise 5]{Fulton97}
\begin{equation}
 s_{\kappa}(1, t, t^2 \dots, t^{M' - 1}) = \frac{\det(t^{(i-1) (\kappa_j+M'-j)})^{M'}_{i,j=1}}{\det(t^{(i-1)(M'-j)})^{M'}_{i,j=1}} = \frac{\Delta_{M'}(t^{\kappa_1+M'-1}, t^{\kappa_2+M'-2}, \dots, t^{\kappa_{M'}})}{\Delta_{M'}(t^{M'-1}, t^{M'-2}, \dots, 1)}.
\end{equation}
Thus with our special choice of $p_i, q_j$ \eqref{Jacobi_sequence_of_weights}, the transition probability \eqref{eq:evolution_of_GT} becomes
\begin{multline} \label{eq:transition_prob_JUE}
  P_{n-1,n}(\kappa^{(M', n-1)}, \kappa^{(M', n)}) = \prod^{M'}_{i=1} (1 - t^{\alpha_n+i}) \frac{\Delta_{M'}(t^{\kappa^{(M', n)}_1+M'-1}, t^{\kappa^{(M', n)}_2+M'-2}, \dots, t^{\kappa^{(M', n)}_{M'}})}{\Delta_{M'}(t^{\kappa^{(M', n-1)}_1+M'-1}, t^{\kappa^{(M', n-1)}_2+M'-2}, \dots, t^{\kappa^{(M, n-1)}_{M'}})} \\
  \times t^{(\alpha_n + 1)( \sum^n_{i=1} \kappa^{(M', n)}_i - \sum^{n - 1}_{i=1} \kappa^{(M', n - 1)}_i)} \id_{\kappa^{(M', n-1)} \preceq \kappa^{(M', n)}},
\end{multline}
where we assume $\kappa^{(M', n)}_i = \kappa^{(M', n - 1)}_j = 0$ for $i > n, j > n - 1$. Note that $\kappa^{(M', n)}_i +M'-i = M' - 1$ (\resp\ $\kappa^{(M', n-1)}_j +M'-j = M' - j$)  if $i > n$ (\resp\ $j > n-1$).

Now we assume
\begin{equation} \label{eq:t_depend_on_L}
  t = e^{-1/L},
\end{equation}
Similar to \eqref{eq:transition_prob_Laguerre_Schur}, if we take the limit $L \to \infty$ in \eqref{eq:evolution_of_GT} and consider $\kappa^{(M', n)}_i = \lfloor Lx^{(M', n)}_i \rfloor$, using the identity

\begin{equation}
  \lim_{L \to \infty} t^{\kappa^{(M', n)}_i+M'-i} = e^{-x^{(M', n)}_i}, \quad \text{for $1 \leq i \leq n \leq N$,}
\end{equation}
we find that the transition probability \eqref{eq:evolution_of_GT} in Proposition \ref{prop:Markov_Schur} (\eqref{eq:transition_prob_JUE} in our special case) satisfies
\begin{multline} \label{eq:asymp_of_Jacobi_process}
  \lim_{L \to \infty} L^n P_{n-1,n}(\lfloor Lx^{(M', n-1)} \rfloor, \lfloor Lx^{(M', n)} \rfloor) = \frac{1}{(M'-n)!} \prod^{M'}_{i=1} (\alpha_n + i) \\
  \times \frac{\Delta_n(e^{-x^{(M', n)}_1}, \dots, e^{-x^{(M', n)}_n})}{\Delta_{n-1}(e^{-x^{(M', n-1)}_1}, \dots, e^{-x^{(M', n-1)}_{n-1}})} \frac{\prod^n_{i=1} (1 - e^{-x^{(M', n)}_i})^{M'-n} e^{-(\alpha_n + 1) x^{(M', n)}_i}}{\prod^{n-1}_{i=1} (1 - e^{-x^{(M', n-1)}_i})^{M'-n+1} e^{-(\alpha_n + 1) x^{(M', n-1)}_i}}.
\end{multline}
Also we have
\begin{equation} \label{eq:asymp_of_Jacobi_process_initial}
  \lim_{L \to \infty} L P(\lfloor x^{(M',1)}_1 \rfloor) = ((M'-1)!)^{-1} \prod^{M'}_{i=1} (\alpha_1 + i) (1 -e^{x^{(M', 1)}_1})^{M'-1} e^{-(\alpha_1 + 1)x^{(M', 1)}_1}.
\end{equation}
On the other hand, since $t$ depends on $L$ by \eqref{eq:t_depend_on_L} similar to \eqref{eq:from_geometrix_to_exponential_Wishart},
\begin{equation}
 \lim_{L \to \infty} \frac{w(i,j)}{L} = u(i,j),
\end{equation}
by \eqref{Jacobi_sequence_of_weights} is a random variable in exponential distribution with parameter $\alpha_j + i$. Then like in the proof of the Wishart case of Theorem \ref{thm:Schur_corrsp}, we consider random variables $\mu^{(M', n)}_{L, i}$ defined by $\kappa^{(M', n)}_i$ as
\begin{equation}
  \kappa^{(M', n)}_i = L\mu^{(M', n)}_{L,i}, \quad \textnormal{for $1 \leq i \leq n \leq N$.}
\end{equation}
As $L \to \infty$, $\mu^{(M', n)}_{L, i}$ converges in distribution to $\mu^{(M', n)}_i$ that is defined in the continuous percolation model by \eqref{eq:defn_mu^mn} where $u(i,j)$ are independent random variables in exponential distribution with parameter $a_j + i$. Then $\mu^{(M', n)}$ constitute an inhomogeneous Markov chain as $n$ runs from $1$ to $N$, and the transition probability $P_{n - 1, n}(\mu^{(M', n - 1)}, \mu^{(M', n)}$ is $\lim_{L \to \infty} L^n P_{n-1,n}(\lfloor L\mu^{(M',n-1)} \rfloor, \lfloor L\mu^{(M', n)} \rfloor)$ given in \eqref{eq:asymp_of_Jacobi_process}, and the distribution of $\mu^{(M', 1)}_1$ is $\lim_{L \to \infty} L P(\lfloor \mu^{(M',1)}_1 \rfloor)$ given in \eqref{eq:asymp_of_Jacobi_process_initial}. To prove the  Jacobi-\Pineiro\ case of Theorem \ref{thm:Schur_corrsp}, we need one more step of change of variables, such that
\begin{equation}
  \nu^{(M', n)}_i = e^{-\mu^{(M', n)}_i}, \quad \textnormal{for $1 \leq i \leq n \leq N$.}
\end{equation}
Noting that $\nu^{(M', n)}$ constitute an inhomogeneous Markov chain and
\begin{align}
  P_{n - 1, n}(\nu^{(M', n - 1)}, \nu^{(M', n)}) = {}& P_{n - 1, n}(\mu^{(M', n - 1)}, \mu^{(M', n)}) \prod^n_{i = 1} \nu^{(M', n)}_i, \\
  P(\nu^{(M', 1)}_1) = {}& P(\mu^{(M', 1)}_1) \nu^{(M', 1)}_1,
\end{align}
we check that the transition probability of $\nu^{(M', n)}$ coincides with the transition probability \eqref{eq:transition_general} in the Jacobi-\Pineiro\ case, and the probability distribution of $\mu^{(M',1)}_1$ is the same as the $p_1(\mu^{(M', 1)}_1)$ in Table \ref{table:thm}, thus finishing the proof in that case.

\paragraph{Multiple Laguerre limit}

The classical transition between the Jacobi and Laguerre weights gives that if $\xi$ is finite and $x = \xi/M'$,
\begin{equation}
  \lim_{M' \to \infty} (M')^{\alpha} x^{\alpha}(1 - x)^{M'} = \xi^{\alpha}e^{-\xi}.
\end{equation}
Then similar to the derivation of the GUE with external source case from the Wishart case, we find the proof of the multiple Laguerre case of Theorem \ref{thm:Schur_corrsp} from that of the Jacobi-\Pineiro\ case.

Suppose in the $M' \times N$ rectangular lattice $\{ (i, j) \mid i = 1, \dotsc, M', j = 1, \dotsc, N \}$ ($N \leq M$), the weight at each site is an independent random variable $u(i, j)$ in exponential distribution with parameter $\alpha_j + i$. Let $\mu^{(M', 1)}, \dotsc, \mu^{(M', N)}$ be defined by \eqref{eq:defn_mu^mn} and for $1 \leq i \leq n \leq N$, we define the random variables  $\tilde{\nu}^{(M', n)} = (\tilde{\nu}^{(M', n)}_1, \dotsc, \tilde{\nu}^{(M', n)}_n)$ and $\nu^{(M', n)} = (\nu^{(M', n)}_1, \dotsc, \nu^{(M', n)}_n)$ by
\begin{equation} \label{eq:relation_between_mu_and_nu_JUE_later}
  \tilde{\nu}^{(M', n)}_i = e^{- \mu^{(M', n)}_i}, \quad \nu^{(M', n)}_i = M' \tilde{\nu}^{(M', n)}_i = M' e^{- \mu^{(M', n)}_i},
\end{equation}
such that $\sum^l_{i=1} \log \nu^{(M', n)}_i = - L^{(l)}(M',n) + l\log
M'$. From the property of $\mu^{(M', n)}$, we have that as $n$ runs
from $1$ to $N$, the sequences  $\tilde{\nu}^{(M', n)} =
(\tilde{\nu}^{(M', n)}_1, \dotsc, \tilde{\nu}^{(M', n)}_n)$ and
$\nu^{(M', n)} = (\nu^{(M', n)}_n, \dotsc, \nu^{(M', n)}_n)$ evolves
as an inhomogeneous Markov chain. Moreover, as $M' \to \infty$,
$\nu^{(M', n)}$ have a joint limit distribution. To compute the joint
limit distribution of $\nu^{(M', n)}$, we note that from the result
obtained above in the Jacobi-\Pineiro\ limit, the distribution
function of $\tilde{\nu}^{(M', n)}$ is given by
\eqref{eq:transition_general} in the Jacobi-\Pineiro\ case, with
$\mu^{(M', n - 1)}, \mu^{(M', n)}$ substituted by $\tilde{\nu}^{(M', n
  - 1)}, \tilde{\nu}^{(M', n)}$. Similar to
\eqref{eq:limiting_transition_prob_GUE_percolation}, if we fix
$\nu^{(M', n)} = \nu^{(n)}$ for all $M'$ and take $M' \to \infty$, then 
\begin{multline} \label{eq:limiting_trans_prpb_LUE_perco}
    \lim_{M' \to \infty} (M')^{-n} P_{n-1,n}(\tilde{\nu}^{(M', n-1)}, \tilde{\nu}^{(M', n)}) = \\
    \frac{1}{\alpha_n !} \frac{\Delta_n(\nu^{(n)}_1, \dots, \nu^{(n)}_n)}{\Delta_{n-1}(\nu^{(n-1)}_1, \dots, \lambda^{(n-1)}_{n-1})} \frac{\prod^n_{i=1} (\nu^{(n)}_i)^{\alpha_n} e^{-\nu^{(n)}_i}}{\prod^{n-1}_{i=1} (\nu^{(n-1)}_i)^{(\alpha_n + 1)} e^{-\nu^{(n-1)}_i}} \id_{\nu^{(n-1)} \preceq \nu^{(n)}}
 \end{multline}
 and similar to \eqref{eq:initial_transition_prob_GUE_percolation}
 \begin{equation}
   \lim_{M' \to \infty} (M')^{-1} P(\tilde{\nu}^{(M', 1)}_1) = (\alpha_1 !)^{-1} (\nu^1_1)^{\alpha_1} e^{-\nu^{(1)}_1} \id_{\nu^{(1)}_1 \geq 0}.
 \end{equation}
 Compare the limiting transition probability \eqref{eq:limiting_trans_prpb_LUE_perco} with the transition probability \eqref{eq:transition_general} in the multiple Laguerre minor process and $\lim_{M' \to \infty} (M')^{-1} P(\tilde{\nu}^{(M', 1)}_1)$ with $p_1(z)$ given in Table \ref{table:thm}, we prove the multiple Laguerre case of Theorem \ref{thm:Schur_corrsp}.
 
\paragraph{Proof of Corollary \ref{cor:polymer_added}}

The statement is based on adapting the arguments of \cite{Glynn-Whitt91}, which uses Donsker's invariance principle and induction and on the following central limit lemma for exponentlally distributed variables:
\begin{lem}
  Let the \iid\ exponentially distributed random variables $X_i$ have mean and standard deviation $\mu = \sigma = (1 - \kappa/\sqrt{M})^{-1}$. Then, we have the following convergence in distribution to a standard Brownian motion $B(t)$:
  \begin{equation}
    \lim_{M \to \infty} \frac{1}{\sqrt{M}} \left( \sum^{\lfloor Mt \rfloor}_{i = 1} X^{(l)}_i - \lfloor Mt \rfloor \right) = B(t) + \kappa t.
  \end{equation}
\end{lem}
\appendix

\section{Contour Integral Representations of multiple Laguerre polynomials of the first kind, and Jacobi-\Pineiro\ Polynomials} \label{sec:appendix_MOPS}

The multiple orthogonal polynomials are multiple weight generalizations of orthogonal polynomials, such that the single weight that defines the orthogonality is replaced by a sequence of weights. Among the orthogonal polynomials, the Hermite, Laguerre and Jacobi polynomials are the most well known, and are called \emph{very classical polynomials}. The three kinds of very classical orthogonal polynomials have various multiple weight generalizations. With our scope being restricted to the so called AT systems, there are four corresponding \emph{very classical multiple orthogonal polynomials}. To be concrete, the Hermite polynomials give rise to the multiple Hermite polynomials, the Laguerre polynomials have two generalizations, namely the multiple Laguerre polynomials of the first and second kinds, and the multiple weight generalization of the Jacobi polynomials are called Jacobi-\Pineiro\ polynomials. The algebraic properties of these four kinds of multiple orthogonal polynomials are summarized in \cite{Coussement-Van_Assche01}, and the relation among them is illustrated in the diagram in the beginning of Section 3 of \cite{Coussement-Van_Assche01}.

The multiple Hermite polynomials and the multiple Laguerre polynomials of the first kind (simply called the multiple Laguerre polynomials in some random matrix theory literature) are closely related to the random matrix models called GUE with external source and LUE with external source (better known as complex Wishart ensemble) respectively \cite{Bleher-Kuijlaars05}. The random matrix models corresponding to the other two kinds of very classical multiple orthogonal polynomials, as far as we know, are first studied in this paper. For the sake of asymptotic analysis, it is desired to have contour integral formulas for these multiple orthogonal polynomials. To our limited knowledge, these contour integral formulas have not been written down explicitly in literature. We state and prove the formulas in this appendix. 

For comparison and reference, we first state the contour integral formulas of multiple Hermite polynomials and multiple Laguerre polynomials of the second kind, both obtained in \cite{Bleher-Kuijlaars05}. Unlike orthogonal polynomials with respect to a single weight, there are type I and type II multiple orthogonal polynomials with respect to a sequence of multiple weights. First we consider multiple Hermite polynomials. Let $a_1, \dotsc, a_n$ be a sequence or distinct real constants. The $n$-th multiple Hermite polynomial of type II, denoted as $P_{(a_1, \dotsc, a_n)}(x)$, is an $n$-th degree monic polynomial satisfying the orthogonality condition with respect to the multiple weights $e^{-\frac{x^2}{2} + a_1 x}, \dotsc, e^{-\frac{x^2}{2} + a_n x}$
\begin{equation}
  \int^{\infty}_{-\infty} P_{(a_1, \dotsc, a_n)}(x) e^{-\frac{x^2}{2} + a_k x} dx = 0, \quad k = 1, \dotsc, n.
\end{equation}
The $(n - 1)$-th multiple Hermite polynomial of type I, denoted as $Q_{(a_1, \dotsc, a_n)}(x)$, is not a polynomial (see Remark \ref{rmk:nomenclature_of_type_I_MOP} below) but the linear combination of the weights $e^{-\frac{x^2}{2} + a_1 x}, \dotsc, e^{-\frac{x^2}{2} + a_n x}$ satisfying the orthogonality condition 
\begin{gather}
  \int^{\infty}_{-\infty} Q_{(a_1, \dots, a_n)}(x) x^k dx = 0, \quad k = 0, \dotsc, n-2, \\
  \int^{\infty}_{-\infty} Q_{(a_1, \dots, a_n)}(x) P_{(a_1, \dotsc, a_{n - 1})}(x) dx = 1.
\end{gather}
There are contour integral formulas for the multiple Hermite polynomials of both types I and II.
\begin{prop}[\cite{Bleher-Kuijlaars05} Theorems 2.1 and 2.3] \label{prop:mult_Hermite}
  \begin{equation} \label{eq:mult_Hermite_type_II}
    P_{(a_1, \dotsc, a_n)}(x) = \frac{1}{\sqrt{2\pi} i} \int_{C + i\realR \uparrow} e^{\frac{1}{2}(s - x)^2} \prod^n_{k = 1} (s - a_k) ds,
  \end{equation}
  where the contour $C + i\realR \uparrow$ is the upward vertical line $\{ z = C + iy \mid y \in \realR \}$, and
  \begin{equation}
    Q_{(a_1, \dotsc, a_n)}(x) = \frac{1}{\sqrt{2\pi}2\pi i} \oint_{\Gamma_a} e^{-\frac{1}{2}(t - x)^2} \prod^n_{k = 1} \frac{1}{t - a_k} dt,
  \end{equation}
  where $\Gamma_a$ encloses all poles $a_1, \dotsc, a_n$ counterclockwise.
\end{prop}
\begin{rmk} \label{rmk:nomenclature_of_type_I_MOP}
  Our definition of the multiple orthogonal polynomials of type I is an abuse of language, but it is essentially equivalent to the standard definition. (Our $Q_{(a_1, \dotsc, a_n)}(x)$ corresponds to the $Q_{\vec{n}}(x)$ in Formula (1.6) in \cite{Bleher-Kuijlaars05}.) This remark also applies to other multiple orthogonal polynomials of type I defined below.
\end{rmk}
\begin{rmk} \label{rmk:multiple_Hermite}
  Our Proposition \ref{prop:mult_Hermite} is only the generic case of the theorems in \cite{Bleher-Kuijlaars05} in the sense that our $P_{(a_1, \dotsc, a_n)}$ and $Q_{(a_1, \dotsc, a_n)}$ are the $P_{\vec{n}}$ and $Q_{\vec{n}}$ with $\vec{n} = (1, \dotsc, 1)$ in \cite{Bleher-Kuijlaars05}. to recover the degenerate cases, it suffices to take the limit that some $a_k$ become identical.
\end{rmk}

The multiple Laguerre polynomials are defined similarly. Let $p > -1$ be a real constant and $a_1, \dotsc, a_n < 0$ be a sequence of distinct negative real constants. The $n$-th multiple Laguerre polynomial of the second kind, type II with respect to the multiple weights $x^p e^{a_1 x}, \dotsc, x^p e^{a_n x}$, denoted as $P_{(a_1, \dotsc, a_n; p)}(x)$, is an $n$-th degree monic polynomial satisfying the orthogonality condition with respect to the multiple weights
\begin{equation}
  \int^{\infty}_0 P_{(a_1, \dotsc, a_n; p)}(x) x^p e^{a_k x} dx = 0, \quad k = 1, \dotsc, n.
\end{equation}
The $(n - 1)$-th multiple Laguerre polynomial of the second kind, type I, denoted as $Q_{(a_1, \dotsc, a_n; p)}(x)$, is the linear combination of $x^p e^{a_1 x}, \dotsc, x^p e^{a_n x}$ satisfying the orthogonality condition
\begin{gather}
  \int^{\infty}_0 Q_{(a_1, \dots, a_n; p)}(x) x^k dx = 0, \quad k = 0, \dotsc, n-2, \\
  \int^{\infty}_0 Q_{(a_1, \dots, a_n; p)}(x) P_{(a_1, \dotsc, a_{n - 1}; p)}(x) dx = 1.
\end{gather}
There are contour integral formulas for the multiple Laguerre polynomials of the second kind, both types I and II.
\begin{prop}[\cite{Bleher-Kuijlaars05} Theorems 3.1 and 3.2] \label{prop:multiple_Laguerre_2nd_kind}
  Suppose $p$ is a nonnegative integer. For $x \in (0, \infty)$, 
  \begin{equation} \label{eq:mLaguerre_2nd_I}
    P_{(a_1, \dotsc, a_n; p)}(x) = (-1)^p \frac{\Gamma(n + p + 1)}{\prod^n_{k = 1} a_k} \frac{x^{-p}}{2\pi i} \oint_{\Sigma} e^{-xs} s^{-n - p - 1} \prod^n_{k = 1} (s - a_k) ds,
  \end{equation}
  where $\Sigma$ is a contour enclosing $0$ counterclockwise, and
  \begin{equation} \label{eq:mLaguerre_2nd_II}
    Q_{(a_1, \dots, a_n; p)}(x) = (-1)^p \frac{\prod^n_{k = 1} a_k}{\Gamma(n + p)} \frac{x^p}{2 \pi i} \oint_{\Gamma_a} e^{xt} t^{n + p - 1} \prod^n_{k = 1} \frac{1}{t - a_k} dt,
  \end{equation}
  where $\Gamma_a$ is a contour enclosing poles $a_1, \dotsc, a_n$ counterclockwise.
\end{prop}
Note that our contour integral formulas \eqref{eq:mLaguerre_2nd_I} and \eqref{eq:mLaguerre_2nd_II} are different from \cite[Formulas (3.5) and (3.10)]{Bleher-Kuijlaars05} by change of variables. An analogue of Remark \ref{rmk:multiple_Hermite} applies to Proposition \ref{prop:multiple_Laguerre_2nd_kind}. 

\begin{rmk}
  Although the multiple Laguerre polynomials are well defined for all real $p > -1$, the statement and the proof given in \cite{Bleher-Kuijlaars05} of the proposition are only valid for integer $p$. When $p$ is not an integer, the expression \eqref{eq:mLaguerre_2nd_I} is valid if we let $\Sigma$ be the deformed Hankel contour, but there is no simple way to generalize the contour integral formula in \eqref{eq:mLaguerre_2nd_II} for noninteger $p$.
\end{rmk}


Now we state the result for multiple Laguerre polynomials of the first kind. Let $\alpha_1, \dots, \alpha_n > -1$ be a sequence of distinct real constants. The $n$-th multiple Laguerre polynomial of the first kind, type II, denoted as $P_{(\alpha_1, \dots, \alpha_n)}(x)$, is an $n$-th degree monic polynomial satisfying orthogonality condition
\begin{equation} \label{eq:orthogonal_relation_Laguerre}
  \int^{\infty}_0 P_{(\alpha_1, \dots, \alpha_n)}(x) x^{\alpha_k}e^{-x} dx = 0, \quad k = 1, \dots, n.
\end{equation}
The $(n-1)$-th multiple Laguerre polynomial of the first kind of type I is the linear combination of $x^{\alpha_1}e^{-x}, \dots, x^{\alpha_n}e^{-x}$ satisfying the orthogonality conditions
\begin{gather}
  \int^{\infty}_0 Q_{(\alpha_1, \dots, \alpha_n)}(x) x^k dx = 0, \quad k = 0, \dots, n-2, \label{eq:orthogonality_of_Laguerre_kind1_type1} \\
  \int^{\infty}_0 Q_{(\alpha_1, \dots, \alpha_n)}(x) P_{(\alpha_1, \dots, \alpha_{n-1})}(x) dx = 1. \label{eq:orthogonality_of_Laguerre_kind1_type1:2}
\end{gather}

There are contour integral formulas for the multiple Laguerre polynomials of the first kind, both types I and II.
\begin{thm} \label{thm:Laguerre}
  For $x \in (0, \infty)$,
  \begin{equation} \label{eq:contour_integral_of_P_Laguerre}
    P_{(\alpha_1, \dots, \alpha_n)}(x) = \frac{e^x}{2\pi i} \oint_{\Sigma} \frac{\Gamma(z+1) \prod^n_{k=1} (z - \alpha_k)}{x^{z+1}} dz,
  \end{equation}
  where the contour $\Sigma$ is the deformed Hankel contour, and
  \begin{equation} \label{eq:contour_integral_of_Q_Laguerre}
    Q_{(\alpha_1, \dots, \alpha_n)}(x) = \frac{e^{-x}}{2\pi i} \oint_{\Gamma_{\alpha}} \frac{x^z}{\Gamma(z+1) \prod^n_{k=1} (z - \alpha_k)} dz,
  \end{equation}
  where the contour $\Gamma_{\alpha}$ encloses the poles $\alpha_1, \dotsc, \alpha_n$ counterclockwise.
\end{thm}

At last we state the contour integral formula for Jacobi-\Pineiro\ polynomials. Let $\beta > -1$ be a real constant, and $\alpha_1, \alpha_2, \dots, \alpha_n > -1$ be distinct real constants. The $n$-th Jacobi-\Pineiro\ polynomial of type II $P_{(\alpha_1, \dots, \alpha_n; \beta)}(x)$ is an $n$-th degree monic polynomial satisfying the orthogonality condition
\begin{equation} \label{eq:orthogonal_with_P}
  \int^1_0 P_{(\alpha_1, \dots, \alpha_n; \beta)}(x) x^{\alpha_k}(1-x)^{\beta} dx = 0, \quad k = 1, \dots, n.
\end{equation}
The $(n-1)$-th Jacobi-\Pineiro\ polynomial of type I, denoted as $Q_{(\alpha_1, \dots, \alpha_n; \beta)}(x)$, is the linear combination of $x^{\alpha_k}(1-x)^{\beta}$, $k = 1, \dots, n$, that satisfies the orthogonal conditions
\begin{gather}
  \int^1_0 Q_{(\alpha_1, \dots, \alpha_n; \beta)}(x) x^k dx = 0, \quad k = 0, \dots, n-2, \label{eq:vanishing_defn_of_Q} \\
  \int^1_0 Q_{(\alpha_1, \dots, \alpha_n; \beta)}(x) P_{(\alpha_1, \dots, \alpha_{n-1}; \beta)}(x) dx = 1. \label{eq:non-vanishing_defn_of_Q}
\end{gather}

There are contour integral formulas for the Jacobi-\Pineiro\ polynomials of both types I and II.
\begin{thm} \label{thm:main_thm}
  For $x \in (0, 1)$,
  \begin{equation} \label{eq:contour_integral_of_P}
    P_{(\alpha_1, \dots, \alpha_n; \beta)}(x) = \frac{\Gamma(n+1+\beta)}{\prod^n_{k=1}(n+1+\beta+\alpha_k)} \frac{(1-x)^{-\beta}}{2\pi i} \oint_{\Sigma} \frac{x^{-z-1} \Gamma(z+1) \prod^n_{k=1} (z-\alpha_k)}{\Gamma(z+n+2+\beta)} dz,
  \end{equation}
  where the contour $\Sigma$ is the deformed Hankel contour, and
  \begin{equation} \label{eq:contour_integral_of_Q}
    Q_{(\alpha_1, \dots, \alpha_n; \beta)}(x) = \frac{\prod^n_{k=1} (n+\beta+\alpha_k)}{\Gamma(n+\beta)} \frac{(1-x)^{\beta}}{2\pi i} \oint_{\Gamma_{\alpha}} \frac{x^z \Gamma(z+n+\beta)}{\Gamma(z+1) \prod^n_{k=1}(z-\alpha_k)} dz,
  \end{equation}
  where the contour $\Gamma_{\alpha}$ encloses the poles $\alpha_1, \dotsc, \alpha_n$ counterclockwise, but does not enclose other poles, if there are any.

  We note that if $\beta \in \intZ$, the contour $\Sigma$ can be a closed contour since the integrand has only finitely many poles, and $\Gamma_{\alpha}$ can be any contour enclosing $\alpha_1, \dotsc, \alpha_n$ since the integrand has no other poles.
  
    In the case $n = 1$ and $\beta + \alpha_1 = -1$, on the right-hand side of \eqref{eq:contour_integral_of_Q} the constant factor vanishes while the contour integral blows up. However, \eqref{eq:contour_integral_of_Q} still holds via \lHopital's rule as $\alpha_1 \to \beta - 1$.
\end{thm}


Before the proof of Theorems \ref{thm:Laguerre} and \ref{thm:main_thm}, we recall the following results about gamma function.

\begin{lem} \label{lem:Gamma_ratio}
  Let $C$ and $\epsilon$ be positive constants, and suppose $\Re z < 0$, $\lvert \Im z \rvert < C$.
  \begin{enumerate}[label=(\alph{*})]
  \item \label{enu:lem:Gamma_ratio_a}
    If $\lvert z + k \rvert > \epsilon$ for all $k = 0, 1, 2, \dots$, we have
    \begin{equation}
      \Gamma(z) = x^{-x-\frac{1}{2}}e^x \bigO(1),
    \end{equation}
    and the $\bigO(1)$ factor is bounded uniformly in $z$.
  \item \label{enu:lem:Gamma_ratio_b}
    If $a, b \in \realR$ and $\lvert z + a + k \rvert > \epsilon$, $\lvert z + b + k \rvert > \epsilon$ for all $k = 0, 1, 2, \dots$, we have
    \begin{equation}
      \frac{\Gamma(z+a)}{\Gamma(z+b)} = x^{a-b}  \bigO(1),
    \end{equation}
    and the $\bigO(1)$ factor is bounded uniformly in $z$.
  \end{enumerate}
\end{lem}
\begin{proof}
  Applying the Euler reflection formula (see \cite[Theorem 1.2.1]{Andrews-Askey-Roy99})
  \begin{equation}
    \Gamma(z) = \frac{\pi}{\sin(\pi z) \Gamma(1-z)},
  \end{equation}
  we can derive the desired asymptotics of $\Gamma(z)$ (\resp\ $\Gamma(z+a)$ and $\Gamma(z+b)$) from those of $\Gamma(1-z)$ (\resp\ $\Gamma(1-z-a)$ and $\Gamma(1-z-b)$). Note that $\Re z < 0$ implies $\Re(1-z) > 1$. Apply the asymptotic formula (\cite[Corollary 1.4.3]{Andrews-Askey-Roy99})
  \begin{equation}
    \Gamma(w) \sim \sqrt{2\pi}w^{w-\frac{1}{2}}e^{-w}, \quad \text{for $\lvert \arg w \rvert \leq \pi - \delta$, where $\delta \in (0, \pi)$.}
  \end{equation}
  we obtain the asymptotics of $\Gamma(1-z)$, and hence prove \ref{lem:Gamma_ratio}\ref{enu:lem:Gamma_ratio_a}. The proof of \ref{lem:Gamma_ratio}\ref{enu:lem:Gamma_ratio_b} is similar.
\end{proof}

In the proof of the theorems, we use the contour $\Sigma_{a, b}$ to realize the deformed Hankel contour $\Sigma$, defined as
\begin{equation} \label{eq:concrete_defn_of_Sigma}
  \Sigma_{a, b} = \Sigma^1_{a, b} \cup \Sigma^2_{a, b} \cup \Sigma^3_{a, b}, \quad \text{where} \quad
  \left\{
    \begin{aligned}
      \Sigma^1_{a, b} = {}& \{ z = t - ai \mid -\infty < t \leq b \}, \\
      \Sigma^2_{a, b} = {}& \{ z = b + it \mid -a \leq t \leq a \}, \\
      \Sigma^3_{a, b} = {}& \{ z = ai - t \mid -a \leq t < \infty \}.
    \end{aligned}
  \right.
\end{equation}

\paragraph{Proof of Theorem \ref{thm:Laguerre}}

First we prove \eqref{eq:contour_integral_of_Q_Laguerre}. A simple residue calculation shows that
\begin{equation}
  \res_{z = \alpha_k} \frac{x^z}{\Gamma(z+1) \prod^n_{k=1} (z - \alpha_k)} dz
\end{equation}
is a constant multiple of $x^{\alpha_k}$. Thus the right-hand side of \eqref{eq:contour_integral_of_Q_Laguerre} is a linear combination of $x^{\alpha_k}e^{-x}$, $k = 1, \dots, n$.

On the other hand, for any $m \in \intZ_+$,
\begin{equation} \label{eq:inner_product_of_Q_with_x^m_Laguerre}
  \begin{aligned}
    & \int^{\infty}_0 \left( \frac{e^{-x}}{2\pi i} \oint_{\Gamma_{\alpha}} \frac{x^z}{\Gamma(z+1) \prod^n_{j=1} (z - \alpha_j)} dz \right) x^m dx \\
    = {}& \frac{1}{2\pi i} \oint_{\Gamma_{\alpha}} \left( \int^{\infty}_0 x^{z+m} e^{-x} dx \right) \frac{1}{\Gamma(z+1) \prod^n_{j=1} (z - \alpha_j)} dz \\
    = {}& \frac{1}{2\pi i} \oint_{\Gamma_{\alpha}} \frac{\Gamma(z+m+1)}{\Gamma(z+1) \prod^n_{j=1} (z - \alpha_j)} dz \\
    = {}& \frac{1}{2\pi i} \oint_{\Gamma_{\alpha}} F_m(z) dz, \quad \text{where} \quad F_m(x) = \frac{\prod^m_{j=1} (z+j)}{\prod^n_{j=1} (z - \alpha_j)}.
  \end{aligned}
\end{equation}
For $m = 0, 1, \dotsc, n - 2$, $F_m(z) = \bigO(z^{m - n}) = \bigO(z^{-2})$ as $z \to \infty$. Deforming $\Gamma_{\alpha}$ into a large enough circle cntered at the origin, we see that $\frac{1}{2\pi i} \oint_{\Gamma_{\alpha}} F_m(z) dz$ vanishes, and prove that the right-hand side of \eqref{eq:contour_integral_of_Q_Laguerre} satisfies \eqref{eq:orthogonality_of_Laguerre_kind1_type1}. If $m = n - 1$, we have $F_m(z) = z^{-1} + \bigO(z^{-2})$. With $\Gamma_{\alpha}$ deformed into a large circle centered at the origin, we see that $\frac{1}{2\pi i} \oint_{\Gamma_{\alpha}} F_m(z) dz = 1$. This result, together with the vanishing result obtained above for $m = 0, \dotsc, n - 2$, yields that the right-hand side of \eqref{eq:contour_integral_of_Q_Laguerre} satisfies \eqref{eq:orthogonality_of_Laguerre_kind1_type1:2}.


Next we prove \eqref{eq:contour_integral_of_P_Laguerre}. To show that the contour integral over $\Sigma$ in \eqref{eq:contour_integral_of_P_Laguerre} is well defined, we apply Lemma \ref{lem:Gamma_ratio}\ref{enu:lem:Gamma_ratio_a}, and find that for $x \neq 0$, the integrand of the contour integral in \eqref{eq:contour_integral_of_P_Laguerre} decays rapidly as $z$ moves along the deformed Hankel contour and $\Re z \to -\infty$. Thus $Q_{(\alpha_1, \dots, \alpha_n)}(x)$ is pointwise defined for $x \neq 0$.

Since $\Gamma(x)$ has poles $z = 0, -1, -2, \dots$, the poles of the integrand in \eqref{eq:contour_integral_of_P_Laguerre} are $z = -1, -2, \dots$. At the pole $z = -j$, we denote the residue
\begin{equation}
  r_j = \res_{z=-j} \frac{\Gamma(z+1) \prod^n_{k=1} (z-\alpha_k)}{x^{z+1}} = \frac{(-1)^n \prod^n_{k=1} (j + \alpha_k)}{(j-1)!} (-x)^{j-1}.
\end{equation}
Thus we can formally apply the residue theorem, and express the contour integral over $\Sigma$ in \eqref{eq:contour_integral_of_P_Laguerre} as $\sum^{\infty}_{j = 1} r_j$, a power series in $x$. To make the argument rigorous, we write
\begin{equation} \label{eq:formula_of_residue_r_j}
  \frac{1}{2\pi i} \oint_{\Sigma} \frac{\Gamma(z+1) \prod^n_{k=1} (z-\alpha_k)}{x^{z+1}} dz = \frac{1}{2\pi i} \oint_{\Sigma_{1, -k - \frac{1}{2}}} \frac{\Gamma(z+1) \prod^n_{k=1} (z-\alpha_k)}{x^{z+1}} dz + \sum^k_{j=1} r_j,
\end{equation}
where $\Sigma_{1, -k - \frac{1}{2}}$ is defined in \eqref{eq:concrete_defn_of_Sigma}. If $x \in (0, \infty)$,  by Lemma \ref{lem:Gamma_ratio}\ref{enu:lem:Gamma_ratio_a}, we have that as $k \to \infty$, the contour integral over $\Sigma_{1, -k - \frac{1}{2}}$ vanishes for any $x \neq 0$. From \eqref{eq:formula_of_residue_r_j}, we find that $\sum^{\infty}_{j=1} r_j$ converges for any $x \in \compC$. Thus the contour integral over $\Sigma$ in \eqref{eq:contour_integral_of_P_Laguerre} is a power series in $x$ for all $x \in (0, \infty)$.

To show that the right-hand side of \eqref{eq:contour_integral_of_P_Laguerre} is a polynomial in $x$ of degree $n-1$, we apply the identity that for $k \in \intZ$,
\begin{equation} \label{eq:contour_integral_equal_to_exp_x^k}
  \frac{1}{2\pi i} \oint_{\Sigma} \frac{\Gamma(z+k+1)}{x^{z+1}} dz = x^k e^{-x}.
\end{equation}
To prove \eqref{eq:contour_integral_equal_to_exp_x^k}, we note that the contour integral in \eqref{eq:contour_integral_equal_to_exp_x^k} is similar in form to that in \eqref{eq:contour_integral_of_P_Laguerre}.  By arguments like above, we express the contour integral in \eqref{eq:contour_integral_equal_to_exp_x^k} in power series and have
\begin{equation} \label{eq:proof_of_lemma:eval_contour_a}
  \frac{1}{2\pi i} \oint_{\Sigma} \frac{\Gamma(z+k+1)}{x^{z+1}} dz =  \sum^{\infty}_{j=0} \res_{z = -j-k-1} \frac{\Gamma(z+k+1)}{x^{z+1}} = \sum^{\infty}_{j=0} \frac{(-1)^j x^{j+k}}{j!} = x^ke^{-x}.
\end{equation}

From the recurrence formula of the gamma function, we find that there are coefficients $c_0, \dots, c_n$ such that
\begin{equation} \label{eq:constants_c_k_Gamma_func}
  \frac{\Gamma(z+1) \prod^n_{k=1} (z-\alpha_k)}{x^{z+1}} = \sum^n_{k=0} c_k \frac{\Gamma(z+k+1)}{x^{z+1}}, \quad \text{where} \quad c_n = 1.
\end{equation}
Thus
\begin{equation}
  \frac{e^x}{2\pi i} \oint_{\Sigma} \frac{\Gamma(z+1) \prod^n_{j=1} (z - \alpha_k)}{x^{z+1}} dz = \sum^n_{k=0} c_k x^k
\end{equation}
is a monic polynomial of degree $n$.

We need to show that the right-hand side of \eqref{eq:contour_integral_of_P_Laguerre} satisfies the orthogonality relation \eqref{eq:orthogonal_relation_Laguerre}. Given any $M > 0$, we have for any $m = 1, \dots, n$
\begin{equation} \label{eq:divide_the_Laguerre_orthogonal_into_several}
  \begin{split}
    & \int^M_0 \left( \frac{1}{2\pi i} \oint_{\Sigma} \frac{\Gamma(z+1) \prod^n_{k=1} (z-\alpha_k)}{x^{z+1}} dz \right) x^{\alpha_m} dx \\
    = {}& \frac{1}{2\pi i} \oint_{\Sigma} \left( \int^M_0 x^{\alpha_m-z-1} dx \right) \Gamma(z+1) \prod^n_{k=1} (z-\alpha_k) dz \\
    = {}& \frac{-1}{2\pi i} \oint_{\Sigma} M^{\alpha_m-z} \Gamma(z+1) \prod_{\substack{k=1, \dots, n \\ k \neq m}} (z-\alpha_k) dz \\
    = {}& \sum^{n-1}_{j=1} \frac{c'_j}{2\pi i} \oint_{\Sigma} M^{\alpha_m-z} \Gamma(z+j) dz,
  \end{split}
\end{equation}
where the constants $c'_1, \dotsc, c'_{n - 1}$ are similar to the constants $c_1, \dotsc, c_n$ in \eqref{eq:constants_c_k_Gamma_func} such that $-\Gamma(z+1) \prod_{\substack{k=1, \dots, n \\ k \neq m}} (z-\alpha_k) = \sum^{n-1}_{j=1} c'_j \Gamma(z+j)$.
By \eqref{eq:contour_integral_equal_to_exp_x^k}, we have
\begin{equation} \label{eq:evaluation_for_m_and_M}
  \frac{1}{2\pi i} \oint_{\Sigma} M^{\alpha_m-z} \Gamma(z+j) = M^{\alpha_m+j} e^{-M}.
\end{equation}
Substituting \eqref{eq:evaluation_for_m_and_M} into \eqref{eq:divide_the_Laguerre_orthogonal_into_several}, we verify the orthogonality condition by
\begin{equation} \label{eq:M_to_infty_vanishing}
  \lim_{M \to \infty} \int^M_0 \left( \oint_{\Sigma} \frac{\Gamma(z+1) \prod^n_{k=1} (z-\alpha_k)}{x^{z+1}} dz \right) x^{\alpha_m} dx = 0.
\end{equation}

\paragraph{Proof of Theorem \ref{thm:main_thm}}

First we prove \eqref{eq:contour_integral_of_Q}. Without loss of generality, we assume $\beta + \alpha_1 \neq 1$ if $n=1$. A simple residue calculation shows that
\begin{equation}
  \res_{z=\alpha_k} \frac{x^z \Gamma(z+n+\beta)}{\Gamma(z+1) \prod^n_{k=1} (x-\alpha_k)}
\end{equation}
is a constant multiple of $x^{\alpha_k}$. Thus the right-hand side of \eqref{eq:contour_integral_of_Q} is the linear combination of $x^{\alpha_k}(1-x)^{\beta}$, $k = 1, \dots, n$.

To prove the orthogonality condition \eqref{eq:vanishing_defn_of_Q}, we note that for any $m = 0, 1, \dots, n-2$,
\begin{equation} \label{eq:inner_product_of_Q_with_x^m}
  \begin{aligned}
    & \int^1_0 \left( \frac{(1-x)^{\beta}}{2\pi i} \oint_{\Gamma_{\alpha}} \frac{x^z \Gamma(z+n+\beta)}{\Gamma(z+1) \prod^n_{k=1}(z-\alpha_k)} dz \right) x^m dx \\
    = {}& \frac{1}{2 \pi i} \oint_{\Gamma_{\alpha}} \left( \int^1_0 x^{z+m}(1-x)^{\beta} dx \right) \frac{\Gamma(z+n+\beta)}{\Gamma(z+1) \prod^n_{k=1}(z-\alpha_k)} dz \\
    = {}& \frac{\Gamma(\beta+1)}{2 \pi i}  \oint_{\Gamma_{\alpha}} \frac{\Gamma(z+m+1) \Gamma(z+n+\beta)}{\Gamma(z+m+2+\beta) \Gamma(z+1)} \frac{dz}{\prod^n_{k=1}(z-\alpha_k)} \\
    = {}& \frac{\Gamma(\beta+1)}{2 \pi i}  \oint_{\Gamma_{\alpha}} G_m(z) dz, \qquad \text{where} \quad G_m(z) = \frac{\prod^m_{j=1} (z+j) \prod^{n-1}_{l=m+2} (z+l+\beta)}{\prod^n_{k=1}(z-\alpha_k)}.
  \end{aligned}
\end{equation}
Similar to \eqref{eq:inner_product_of_Q_with_x^m_Laguerre}, for $m = 0, 1, \dotsc, n - 2$, $G_m(z) = \bigO(z^{m - n}) = \bigO(z^{-2})$ as $z \to \infty$. By deforming $\Gamma_{\alpha}$ into a large circle, we have that $\frac{1}{2\pi i} \oint_{\Gamma_{\alpha}} G_m(z) dz$ vanishes, and prove that the right-hand side of \eqref{eq:contour_integral_of_Q} satisfies \eqref{eq:vanishing_defn_of_Q}.

To verify the orthogonality condition \eqref{eq:non-vanishing_defn_of_Q}, we have similar to \eqref{eq:inner_product_of_Q_with_x^m} 
\begin{multline} \label{eq:inner_of_Q_and_x^n-1}
    \int^1_0 \left( \frac{(1-x)^{\beta}}{2\pi i} \oint_{\Gamma_{\alpha}} \frac{x^z \Gamma(z+n+\beta)}{\Gamma(z+1) \prod^n_{k=1}(z-\alpha_k)} dz \right) x^{n-1} dx \\
    = \Gamma(\beta+1) \oint_{\Gamma_{\alpha}} G_{n - 1}(z) dz, \quad \text{where} \quad G_{n - 1}(z) = \frac{\prod^{n-1}_{j=1} (z+j)}{(z+n+\beta) \prod^n_{k=1}(z-\alpha_k)}.
\end{multline}
If we replace the contour $\Gamma_{\alpha}$ into a large circle $\tilde{\Gamma}$, by the estimate $G_{n - 1}(z) = \bigO(z^{-2})$ we have
\begin{equation} \label{eq:vanishing_of_Q_inner_x^n-1_big_circle}
  \frac{1}{2\pi i} \oint_{\tilde{\Gamma}} G_{n - 1}(z) dz = 0,
\end{equation}
Note that the contour $\tilde{\Gamma}$ encloses one more pole $z = -(n+\beta)$ than the contour $\Gamma_{\alpha}$. We have by \eqref{eq:vanishing_of_Q_inner_x^n-1_big_circle} that
\begin{equation} \label{eq:residue_formula_of_Q_inner_x^n-1}
  \frac{1}{2\pi i} \oint_{\Gamma_{\alpha}} G_{n - 1} dz = -\res_{z = -(n+\beta)} \frac{\prod^{n-1}_{j=1} (z+j)}{(z+n+\beta) \prod^n_{k=1}(z-\alpha_k)}  = \frac{\prod^{n-1}_{j=1} (j+\beta)}{\prod^n_{k=1} (n+\beta+\alpha_k)}.
\end{equation}
By \eqref{eq:inner_of_Q_and_x^n-1}, \eqref{eq:residue_formula_of_Q_inner_x^n-1} and the vanishing of $\frac{1}{2\pi i} \oint_{\Gamma_{\alpha}} G_m(z) dz$ for $m = 0, 1, \dotsc, n - 1$, we verify that the right-hand side of \eqref{eq:contour_integral_of_Q} satisfies \eqref{eq:non-vanishing_defn_of_Q}.


Next we prove \eqref{eq:contour_integral_of_P}. To show that the contour integral over $\Sigma$ in \eqref{eq:contour_integral_of_P} is well defined, we apply Lemma \ref{lem:Gamma_ratio}\ref{enu:lem:Gamma_ratio_b}, and find that for $x \neq 0$, we see that the integrand in \eqref{eq:contour_integral_of_P} decays rapidly as $z$ moves along the deformed Hankel contour and $\Re z \to -\infty$. Thus $Q_{(\alpha_1, \dots, \alpha_n; \beta)}(x)$ is pointwise defined for $x \neq 0$.

Since $\Gamma(z)$ has no zero and has poles at $z = 0, -1, -2, \dots$, the poles of the integrand in \eqref{eq:contour_integral_of_P} are $z = -1, -2, \dots$. At the pole $z = -j$, we denote the residue
\begin{equation} \label{eq:formula_of_residue_R_j}
  \begin{split}
    R_j = {}& \res_{z=-j} \frac{x^{-z-1} \Gamma(z+1) \prod^n_{k=1} (z-\alpha_k)}{\Gamma(z+n+2+\beta)} \\
    = {}& \res_{z=-j} \frac{x^{-z-1} \Gamma(z+j+1) \prod^n_{k=1} (z-\alpha_k)}{\Gamma(z+n+2+\beta) \prod^j_{l=1} (z+l)} \\
    = {}& \frac{(-1)^{n-j}}{\Gamma(n+2+\beta)} \frac{\prod^n_{k=1} (j + \alpha_k) \prod^j_{l=1} (n-l+2 + \beta)}{(j-1)!} x^{j-1}.
  \end{split}
\end{equation}
Thus we can formally apply the residue theorem, and express the contour integral over $\Sigma$ in \eqref{eq:contour_integral_of_P} as $\sum^{\infty}_{j = 1} R_j$, a power series in $x$. To make the argument rigorous, we mimick the argument for the contour integral in \eqref{eq:contour_integral_of_P_Laguerre} and write
\begin{equation}
  \frac{1}{2\pi i} \oint_{\Sigma} \frac{x^{-z-1} \Gamma(z+1) \prod^n_{k=1} (z-\alpha_k)}{\Gamma(z+n+2+\beta)} dz = \frac{1}{2\pi i} \oint_{\Sigma_{1, -k - \frac{1}{2}}} \frac{x^{-z-1} \Gamma(z+1) \prod^n_{k=1} (z-\alpha_k)}{\Gamma(z+n+2+\beta)} dz + \sum^k_{j=1} R_j,
\end{equation}
where $\Sigma_{1, -k - \frac{1}{2}}$ is the same as in \eqref{eq:formula_of_residue_r_j}. If $x \in [0,1]$, by Lemma \ref{lem:Gamma_ratio}\ref{enu:lem:Gamma_ratio_b}, we have that as $k \to \infty$, the contour integral over $\Sigma_{1, -k - \frac{1}{2}}$ vanishes for any $x \neq 0$, and from \eqref{eq:formula_of_residue_r_j}, we find that $\sum^{\infty}_{j=1} r_j$ converges for any $x \in \compC$. Thus the contour integral over $\Sigma$ in \eqref{eq:contour_integral_of_P_Laguerre} is a power series in $x$ for all $x \in (0, 1)$.

To show that the right-hand side of \eqref{eq:contour_integral_of_P} is a polynomial in $x$ of degree $n-1$, we apply the identity that for $0 \leq x \leq 1$ and $k \in \intZ$,
\begin{equation} \label{eq:expansion_of_monomial_(1-x)^beta}
  \frac{1}{2\pi i} \oint_{\Sigma} \frac{x^{-z-1} \Gamma(z+k+1)}{\Gamma(z+k+2+\beta)} dz = \frac{1}{\Gamma(1+\beta)} x^k(1-x)^{\beta}.
\end{equation}
The proof of \eqref{eq:expansion_of_monomial_(1-x)^beta} is similar to that of \eqref{eq:contour_integral_equal_to_exp_x^k}, Noting the similarity between the contour integral in \eqref{eq:expansion_of_monomial_(1-x)^beta} and that in \eqref{eq:contour_integral_of_P} by arguments like above we have
\begin{equation} \label{eq:expansion_of_monomial_(1-x)^beta_again}
  \begin{split}
    \frac{1}{2\pi i} \oint_{\Sigma} \frac{x^{-z-1} \Gamma(z+k+1)}{\Gamma(z+k+2+\beta)} dz = {}& \sum^{\infty}_{j=0} \res_{z=-j-k-1} \frac{x^{-z-1} \Gamma(z+j+k+2)}{\Gamma(z+k+2+\beta) \prod^{j+1}_{l=1} (z+k+l)} \\
    = {}& \sum^{\infty}_{j=0} \frac{(-1)^jx^{j+k}}{\Gamma(-j+2+\beta) j!} \\
    = {}& \sum^{\infty}_{j=0} \frac{\prod^{j-1}_{l=0}(\beta-l)}{\Gamma(1+\beta) j!}(-1)^jx^{j+k}.
  \end{split}
\end{equation}
Comparing \eqref{eq:expansion_of_monomial_(1-x)^beta_again} with the expansion
\begin{equation}
  x^k(1-x)^{\beta} = \sum^{\infty}_{j=0} \frac{\prod^{j-1}_{l=0}(\beta-l)}{j!}(-1)^jx^{j+k},
\end{equation}
we prove \eqref{eq:expansion_of_monomial_(1-x)^beta}.

Similar to \eqref{eq:constants_c_k_Gamma_func}, there are coefficients $c_0, c_1, \dots, c_n$ such that
\begin{equation} \label{eq:linear_expansion_of_contour_rep_of_Q}
  \frac{\Gamma(z+1) \prod^n_{k=1} (z-\alpha_k)}{\Gamma(z+n+2+\beta)} = \sum^n_{k=0} c_k \frac{\Gamma(z+k+1)}{\Gamma(z+k+2+\beta)}, \quad \text{where} \quad c_n = \frac{\prod^n_{k=1}(n+1+\beta+\alpha_k)}{\prod^n_{k=1} (k+\beta)}.
\end{equation}
 In \eqref{eq:linear_expansion_of_contour_rep_of_Q}, the value of $c_n$ is obtained, when $\beta \not\in \intZ$, by substituting $z = -(n+1+\beta)$, so that \eqref{eq:linear_expansion_of_contour_rep_of_Q} becomes
\begin{equation} \label{eq:linear_expansion_of_contour_Q_specialization}
  \Gamma(-n-\beta) \prod^n_{k=1} (-n-1-\beta-\alpha_k) = c_n \Gamma(-\beta),
\end{equation}
and further by \lHopital's rule when $\beta \in \intZ$. Summarizing the results obtained, we have
\begin{equation}
  \frac{(1-x)^{-\beta}}{2\pi i} \oint_{\Sigma} \frac{x^{-z-1} \Gamma(z+1) \prod^n_{k=1} (z-\alpha_k)}{\Gamma(z+n+2+\beta)} dz = \sum^n_{k=0} \frac{c_k}{\Gamma(1+\beta)} x^k,
\end{equation}
and hence prove that the right-hand side of \eqref{eq:contour_integral_of_P} is a monic polynomial of degree $n$.

We need to show that the orthogonal relation \eqref{eq:orthogonal_with_P} holds. For $m = 1, \dotsc, n$,
\begin{equation} \label{eq:contour_integral_for_orth_of_P}
  \begin{split}
    & \int^1_0 \left( \oint_{\Sigma} \frac{x^{-z-1} \Gamma(z+1) \prod^n_{k=1} (z-\alpha_k)}{\Gamma(z+n+2+\beta)} dz \right) x^{\alpha_m} dx \\
    = {}& \lim_{\epsilon \to 0_+} \int^{1-\epsilon}_{\epsilon} \left( \oint_{\Sigma} \frac{x^{-z-1} \Gamma(z+1) \prod^n_{k=1} (z-\alpha_k)}{\Gamma(z+n+2+\beta)} dz \right) x^{\alpha_m} dx \\
    = {}& \oint_{\Sigma} \left( \int^{1-\epsilon}_{\epsilon} x^{-z-1}x^{\alpha_m} dx \right) \frac{\Gamma(z+1) \prod^n_{k=1} (z-\alpha_k)}{\Gamma(z+n+2+\beta)} dz \\
    = {}& - \oint_{\Sigma} \frac{\Gamma(z+1) \prod_{\substack{k=1, \dots, n \\ k \neq m}} (z-\alpha_k)}{\Gamma(z+n+2+\beta)} dz.
  \end{split}
\end{equation}
From the formula \eqref{eq:contour_integral_of_P}, we see that if we can show the last contour integral in \eqref{eq:contour_integral_for_orth_of_P} vanishes, then the proof is done. To see that, we take $a = b = C$, a large positive number, in the definition of the contour $\Sigma$ in \eqref{eq:concrete_defn_of_Sigma}. By Lemma \ref{lem:Gamma_ratio}\ref{enu:lem:Gamma_ratio_b}, the integrand in the last contour integral in \eqref{eq:contour_integral_for_orth_of_P} is $\bigO(\lvert z \rvert^{-(2+\beta)})$. Thus the contour integral itself should be $\bigO(C^{-(1+\beta)})$. Let $C \to \infty$, we prove that the contour integral vanishes.

\section{Construction of contours} \label{sec:constr_of_contours}

First we construct $\Sigma$ of Figure \ref{fig:Sigma_and_Gamma}. Let $\Sigma$ be of the shape (see Figure \ref{fig:real_Sigma})
\begin{equation} \label{eq:explicit_formula_of_Sigma}
  \Sigma := \Sigma_1 \cup \Sigma_2 \cup \Sigma_3, \quad \text{where} \quad
  \begin{cases}
    \Sigma_1 := x_0 + (1 - i)\sigma + t, & t \leq -\sigma, \\
    \Sigma_2 := x_0 + it, & -\sigma \leq t \leq \sigma, \\
    \Sigma_3 := x_0 + (1 + i)\sigma - t, & t \geq \sigma.
  \end{cases}
\end{equation}
Note that for a large enough $\sigma$, we have the following

\paragraph{Conditions for $\sigma$}

\begin{enumerate}
\item
  For $z \in \Sigma_1 \cup \Sigma_3$, $\Re F(z)$ decreases as $z$ moves to the left.
\item
  For $t \in (\sigma, \infty) \cup (-\infty, -\sigma)$, $\Re F(it) < \Re F(x_0)$.
\end{enumerate}
We assume the parameter $\sigma$ in \eqref{eq:explicit_formula_of_Sigma} satisfy the two conditions above.

To show that $\Re F(z)$ attains its maximum on $\Sigma$ at $x_0$, it then suffices to show that $\Re F(z)$ decreases as $z$ moves along $\Sigma_2$ upward (\resp\ downward) from $x_0$ to $x_0 \pm i\sigma$. We need only to consider the case $z \in \Sigma_2 \cap \compC_+$ due to the symmetry of $\Re F(z)$ about the real axis. To this end, since $F'(x_0) = F''(x_0) = 0$, it suffices to show that for all $t > 0$,
\begin{equation}
  \begin{split}
    \frac{\partial^2 (\Re F(x_0 + it))}{\partial t^2} = -\Re F''(x_0 + it) ={}& -\Re \left[ \frac{1}{x_0 + it} - \frac{1}{(x_0 - a + it)^2} - \frac{1}{x_0 - b + it)^2} \right] \\
    ={}& -\frac{x_0}{x^2_0 + t^2} + \frac{(x_0 - a)^2 - t^2}{((x_0 - a)^2 + t^2)^2} + \frac{(x_0 - b)^2 - t^2}{((x_0 - b)^2 + t^2)^2} \\
    < {}& 0.
  \end{split}
\end{equation}
Using the expressions of $x_0$, $x_0 - a$ and $x_0 - b$ in \eqref{eq:expression_of_u_v_x_0} and \eqref{eq:expression_of_a_b_in_r} and taking the substitution $t = \frac{(1 + r^2)^2}{2(1 + r^3)}s$, we have
\begin{equation}
  \Re F''(x_0 + it) = \frac{4(1 + r^3)^2}{(1 + r^2)^3} G(s), 
\end{equation}
where
\begin{equation}
  G(s) = \frac{(1 + r^2)^3}{(1 + r^2)^2 + 4(1 + r^3)^2 s^2} - \frac{1 - s^2}{(1 + s^2)^2} - r^2 \frac{1 - r^2 s^2}{(1 + r^2 s^2)^2}.
\end{equation}
We need to show $G(s) \geq 0$ when $s > 0$. Since $G(0) = 0$, it suffices to show $G'(s) > 0$ for $s > 0$. Writing
\begin{multline}
  G'(s) = \frac{s^2}{[(1 + r^2)^2 + 4(1 + r^3)^2 s^2](1 + s^2)^2} \\
  \times \left[ (3(1 + r^2)^2 - 4(1 + r^3)^2) + ((1 + r^2)^2 + 4(1 + r^3)^2)s^2 \right].
\end{multline}
We check that the coefficient
\begin{equation}
  3(1 + r^2)^2 - 4(1 + r^3)^2 = (-2r^3 + \sqrt{3}r^2 + \sqrt{3} - 2)(\sqrt{3}(1 + r^2) + 2(1 + r^2))
\end{equation}
is positive for $r \in (-1, r_0)$, hence $G'(s) > 0$ for $s > 0$ given that $r \in (-1, r_0)$. Therefore our construction of $\Sigma$ is valid.

Next we want to construct $\Gamma_b$ and $\Gamma_a$ of Figure \ref{fig:Sigma_and_Gamma} such that $\Re F(z)$ attains its minimum on $\Gamma_b$ (\resp\ $\Gamma_a$) near $x_0$. Since $F(\overline{z}) = \overline{F(z)}$, we consider only the parts of $\Gamma_b$ and $\Gamma_a$ on the upper half plane, and let the lower parts of $\Gamma_b$ and $\Gamma_a$ be obtained by reflection. We construct first $\Gamma^{\prelim}_b$ (\resp\ $\Gamma^{\prelim}_a$) as an approximation of $\Gamma_b \cap \compC_+$ (\resp\ $\Gamma_a \cap \compC_+$), and then deform $\Gamma^{\prelim}_b$ (\resp\ $\Gamma^{\prelim}_a$) locally around $x_0$ to obtain $\Gamma_b \cap \compC_+$ (\resp\ $\Gamma_a \cap \compC_+$).

In our construction of the contours, we need a technical result that $F'(z)$ has no zeros in the first quadrant (except for $x_0$ on the boundary). To see it, we consider the contour integral $\frac{1}{2\pi i} \oint_{C_{\quadrant}} F''(z)/F'(z) dz$, where $C_{\quadrant}$ is a quarter-circle contour in the first quadrant, centered at $0$ and locally deformed at $a$, $b$ and $x_0$ to exclude the singularities, see Figure \ref{fig:C_quad}. The value of this contour integral is the number of zeros of $F'(z)$ enclosed by $C_{\quadrant}$. By direct calculation, we see that the contour integral vanishes as the radius of $C_{\quadrant}$ approaches $\infty$, which implies that $F'(z)$ has no zero in the first quadrant.
\begin{figure}[htb]
  \begin{minipage}[t]{0.45\linewidth}
  \centering
  \includegraphics{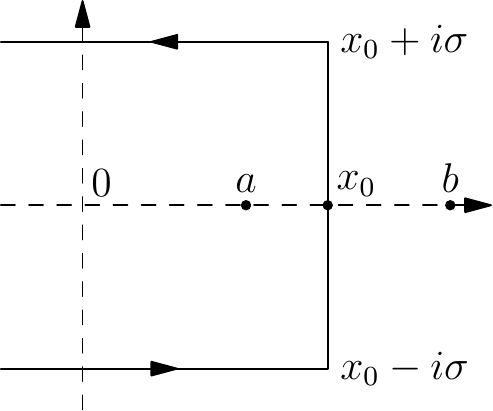}
  \caption{The actual shape of $\Sigma$.}
  \label{fig:real_Sigma}
  \end{minipage}
  \hspace{\stretch{1}}
  \begin{minipage}[t]{0.45\linewidth}
    \centering
    \includegraphics{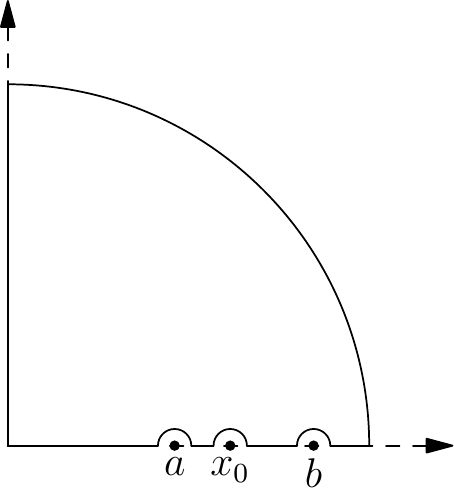}
    \caption{The contour $C_{\quadrant}$.}
    \label{fig:C_quad}
  \end{minipage}
\end{figure}

First we construct $\Gamma^{\prelim}_b$. By the local property of $F(z)$ around $x_0$, we have that the direction $\frac{\pi}{4}$ is a steepest-ascent direction of $\Re F(z)$ at $x_0$. we construct the gradient flow line of the vector field $\nabla\Re F(z)$ on the complex plane, (with the complex plane identified with $\realR^2$,) starting from $x_0$ with initial direction $\frac{\pi}{4}$. The value of $\Re F(z)$ increases as $z$ moves along the flow line. Consider the region
\begin{equation}
  B = \{ z \in \compC_+ \mid 0 < \Im z < N_1,\ 0 < \Re z < N_2  \text{ and $z$ is to the right of $\Sigma$,} \} 
\end{equation}
Here we choose $N_1$ and $N_2$ such that they satisfy the following

\paragraph{Conditions for $N_1$ and $N_2$}

\begin{enumerate}
\item \label{enu:condition_of_N_1_N_2:1}
  As $z$ moves along the horizontal line $\realR + iN_1$ to the right, $\Re F(z)$ is increasing.
\item \label{enu:condition_of_N_1_N_2:2}
  For $z$ in the vertical line between $N_2$ and $N_2 + iN_1$, $\Re F(z) > \Re F(x_0)$.
\end{enumerate}
It is not difficult to check that if $N_1$ is large enough, then Condition \ref{enu:condition_of_N_1_N_2:1} is satisfied, and for any fixed $N_1$, if $N_2$ is large enough, then Condition \ref{enu:condition_of_N_1_N_2:2} is satisfied. So the region $B$ is well defined. We denote by $\Gamma_B$ the part of the gradient flow line in $B$ and connected to $x_0$. Since along the gradient flow line $\Re F(z)$ is eventually unbounded, $\Gamma_B$ has to hit the boundary of $B$. The value of $\Re F(z)$ on $\Sigma$ is less than $\Re F(x_0)$, so $\Gamma_B$ cannot hit $\Sigma$. For the same reason it cannot hit the line segment between $x_0$ and $b$, or the imaginary axis. Then there are only two cases to consider:
\begin{enumerate}
\item 
  $\Gamma_B$ hits the ray $(b, +\infty)$: We define $\Gamma^{\prelim}_b$ as $\Gamma_B$, as shown in Figure \ref{fig:small_Gamma_b}.
\item 
  $\Gamma_B$ hits the top or the right side of $B$: We define $\Gamma^{\prelim}_b$ as the combination of $\Gamma_B$ and the right side (and possibly also the top) of $B$ between its intersection with $\Gamma_B$ and its intersection with the real axis, as shown in Figure \ref{fig:big_Gamma_b}.
\end{enumerate}
\begin{figure}[h]
  \begin{minipage}[t]{0.45\linewidth}
    \centering
    \includegraphics{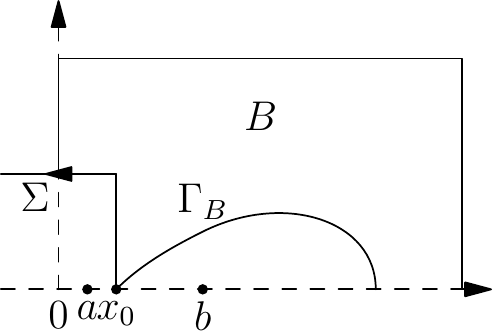}
    \caption{One possible construction of $\Gamma^{\prelim}_b$ when $\Gamma_B$ hits the bottom of $B$.}
    \label{fig:small_Gamma_b}
  \end{minipage}
  \hspace{\stretch{1}}
  \begin{minipage}[t]{0.45\linewidth}
    \centering
    \includegraphics{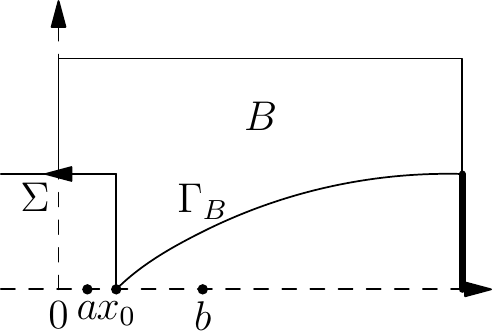}
    \caption{Another possible construction of $\Gamma^{\prelim}_b$ when $\Gamma_B$ hits the right side of $B$, and then $\Gamma^{\prelim}_b$ is the union of $\Gamma_B$ and the part of the right side of $B$ (in boldface) between $\Gamma_B$ and the real axis.}
    \label{fig:big_Gamma_b}
  \end{minipage}
\end{figure}
It is clear that in either case, $\Re F(z)$ has $x_0$ as the unique minimum on $\Gamma^{\prelim}_b$.

The construction of $\Gamma^{\prelim}_a$ is done in a similar way. Denote
\begin{equation}
  A = \{ z \in \compC_+ \mid \text{$z$ is to the left of $\Sigma$ and $0 < \Im z < \sigma$} \},
\end{equation}
and construct the gradient flow line of $\nabla \Re F(z)$ starting from $x_0$ with initial steepest ascent direction $\frac{3\pi}{4}$. We denote the part of the flow line in $A$ and connected to $x_0$ by $\Gamma_A$. By the same argument as for $\Gamma_B$, we have the following two possible cases:
\begin{enumerate}
\item \label{enu:Gamma_A:1}
  $\Gamma_A$ hits $(0, a)$: we define $\Gamma^{\prelim}_a$ as $\Gamma_A$, as shown in Figure \ref{fig:big_Gamma_a}.
\item \label{enu:Gamma_A:2}
  $\Gamma_A$ hits the left side of $A$, \ie, the vertical line between $0$ and $i\sigma$: we define $\Gamma^{\prelim}_a$ as the union of $\Gamma_A$ and the vertical line between $0$ and the intersection of $\partial A$ and $\Gamma_A$, as shown in Figure \ref{fig:small_Gamma_a}.
\end{enumerate}
\begin{figure}[h]
  \begin{minipage}[t]{0.45\linewidth}
    \centering
    \includegraphics{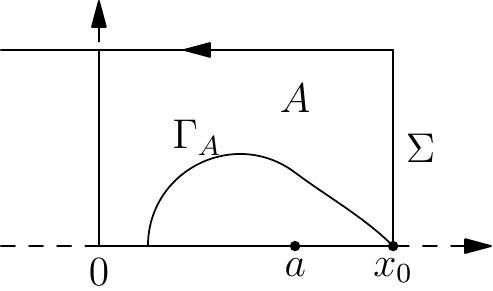}
    \caption{One possible construction of $\Gamma^{\prelim}_a$ when $\Gamma_A$ hits $(0, a)$.}
    \label{fig:big_Gamma_a}
  \end{minipage}
  \hspace{\stretch{1}}
  \begin{minipage}[t]{0.45\linewidth}
    \centering
    \includegraphics{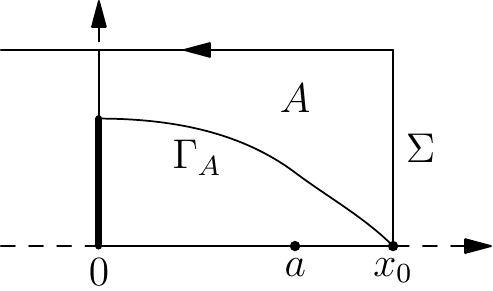}
    \caption{Another possible construction of $\Gamma^{\prelim}_a$ when $\Gamma_A$ hits the left side of $B$, and then $\Gamma^{\prelim}_a$ is the union of $\Gamma_B$ and the vertical line (in boldface) between $\Gamma_A$ and the real axis.}
    \label{fig:small_Gamma_a}
  \end{minipage}
\end{figure}
Note that in case \ref{enu:Gamma_A:1}, $\Re F(z)$ has $x_0$ as the unique minimum on $\Gamma^{\prelim}_a$ due to the property of the flow line $\Gamma_A$, but in case \ref{enu:Gamma_A:2}, this is not always true. If $a$ is big enough, or equivalently the parameter $r \in (-1, r_0)$ defined in \eqref{eq:defn_of_r} is away from $r_0$, we have that for $z = it$ ($t > 0$) on the positive imaginary axis
\begin{equation} \label{eq:ineq_depending_on_a}
  \frac{\partial(\Re F(it))}{\partial t} = -\Im F'(it) = \frac{1}{a} \frac{1}{\frac{a}{t} + \frac{t}{a}} + \frac{1}{b} \frac{1}{\frac{b}{t} + \frac{t}{b}} - \frac{\pi}{2} < 0,
\end{equation}
and then $\Re F(z)$ is decreasing as $z = it$ moves from $0$ upward along the imaginary axis. Hence the definition of $\Gamma^{\prelim}_a$ is valid if \eqref{eq:ineq_depending_on_a} holds for all $t > 0$. Numerical result shows that \eqref{eq:ineq_depending_on_a} holds if $a > 0.3436$, or equivalently $r < -0.7150$.

\begin{rmk} \label{rmk:large_enough_a}
  Although more sophisticated construction of $\Gamma^{\prelim}_a$ allow smaller $a$, there is an essential difficulty if we let $a \to 0^+$. Since we require $\Gamma_a$ (and so $\Gamma^{\prelim}_a$) not to intersect with $(-\infty, 0)$, there is $z_0 \in \Gamma^{\prelim}_a \cap (0, a)$. We need $\Re F(z_0) > \Re F(x_0)$. However, it is not difficult to show that if $a$ is very close to $0$, or equivalently, $r$ is very close to $r_0$,
  \begin{equation} \label{eq:lower_bound_of_a}
    \Re F(x) < \Re F(x_0), \quad \text{for all $x \in (0, a)$.}
  \end{equation}
  and then the construction of $\Gamma^{\prelim}_a$ is impossible. Numerical result shows that \eqref{eq:lower_bound_of_a} holds if $a < 0.1184$, or equivalently $r > -0.6827$. To consider the case that $a$ is very close to $0$, we need to allow $\Gamma_a$ to cross $(-\infty, 0)$, and radically change the contour integral formula. We do not pursue it in this paper.
\end{rmk}

The contours $\Gamma^{\prelim}_a$ and $\Gamma^{\prelim}_b$ with their reflections in the lower half plane are still not satisfactory contours, because
\begin{enumerate*}[label*=(\arabic*)]
  \item
    they intersect with $\Sigma$ at $x_0$,
  \item
    they are not described by explicit formulas around $x_0$, the point around which we do steepest-descent analysis.
\end{enumerate*}
To solve these defects, we first choose a small enough $\epsilon > 0$ and deform the part of $\Gamma^{\prelim}_a$ (\resp\ $\Gamma^{\prelim}_b$) that is within distance $\epsilon$ to $x_0$ into the straight line segment between $x_0$ to $x_0 + e^{\frac{3\pi i}{4}}\epsilon$ (\resp\ between $x_0$ to $x_0 + e^{\frac{\pi i}{4}}\epsilon$) such that $\Re F(z)$ still increases as $z$ moves along the contours away from $x_0$. Then we deform the line segment between $x_0$ and $x_0 + e^{\frac{3\pi i}{4}} c^{-1}_3n^{-\frac{1}{4}}$ (\resp\ the line segment between $x_0$ and $x_0 + e^{\frac{\pi i}{4}} c^{-1}_3n^{-\frac{1}{4}}$) into the vertical line segment between $x_0 - \frac{1}{\sqrt{2}} c^{-1}n^{-\frac{1}{4}}$ and $x_0 + e^{\frac{3\pi i}{4}} c^{-1}_3n^{-\frac{1}{4}}$ (\resp\ the vertical line segment between $x_0 + \frac{1}{\sqrt{2}} c^{-1}_3n^{-\frac{1}{4}}$ and $x_0 + e^{\frac{\pi i}{4}} c^{-1}_3n^{-\frac{1}{4}}$. Then we obtain $\Gamma_a \cap \compC_+$ (\resp\ $\Gamma_b \cap \compC_+$) and the final contour $\Gamma_a$ (\resp\ $\Gamma_b$ is obtained by reflection about the real axis.

\bibliographystyle{abbrv}
\bibliography{bibliography.bib}

\def\cydot{\leavevmode\raise.4ex\hbox{.}}
\begin{thebibliography}{10}

\bibitem{Adler-Nordenstam-van_Moerbeke10a}
M.~Adler, E.~Nordenstam, and P.~van Moerbeke.
\newblock Consecutive minors for {D}yson's {B}rownian motions, 2010.
\newblock arXiv:1007.0220.

\bibitem{Adler-Nordenstam-van_Moerbeke10}
M.~Adler, E.~Nordenstam, and P.~van Moerbeke.
\newblock The {D}yson {B}rownian minor process, 2010.
\newblock arXiv:1006.2956.

\bibitem{Adler-Orantin-van_Moerbeke10}
M.~Adler, N.~Orantin, and P.~van Moerbeke.
\newblock Universality for the {P}earcey process.
\newblock {\em Phys. D}, 239(12):924--941, 2010.

\bibitem{Amir-Corwin-Quastel11}
G.~Amir, I.~Corwin, and J.~Quastel.
\newblock Probability distribution of the free energy of the continuum directed
  random polymer in {$1+1$} dimensions.
\newblock {\em Comm. Pure Appl. Math.}, 64(4):466--537, 2011.

\bibitem{Andrews-Askey-Roy99}
G.~E. Andrews, R.~Askey, and R.~Roy.
\newblock {\em Special functions}, volume~71 of {\em Encyclopedia of
  Mathematics and its Applications}.
\newblock Cambridge University Press, Cambridge, 1999.

\bibitem{Baryshnikov01}
Y.~Baryshnikov.
\newblock G{UE}s and queues.
\newblock {\em Probab. Theory Related Fields}, 119(2):256--274, 2001.

\bibitem{Bleher-Kuijlaars05}
P.~M. Bleher and A.~B.~J. Kuijlaars.
\newblock Integral representations for multiple {H}ermite and multiple
  {L}aguerre polynomials.
\newblock {\em Ann. Inst. Fourier (Grenoble)}, 55(6):2001--2014, 2005.

\bibitem{Borodin-Ferrari08a}
A.~Borodin and P.~L. Ferrari.
\newblock Anisotropic growth of random surfaces in 2+1 dimensions, 2008.
\newblock arXiv:0804.3035.

\bibitem{Borodin-Ferrari08}
A.~Borodin and P.~L. Ferrari.
\newblock Large time asymptotics of growth models on space-like paths. {I}.
  {P}ush{ASEP}.
\newblock {\em Electron. J. Probab.}, 13:no. 50, 1380--1418, 2008.

\bibitem{Borodin-Ferrari-Prahofer-Sasamoto07}
A.~Borodin, P.~L. Ferrari, M.~Pr{\"a}hofer, and T.~Sasamoto.
\newblock Fluctuation properties of the {TASEP} with periodic initial
  configuration.
\newblock {\em J. Stat. Phys.}, 129(5-6):1055--1080, 2007.

\bibitem{Borodin-Peche08}
A.~Borodin and S.~P{\'e}ch{\'e}.
\newblock Airy kernel with two sets of parameters in directed percolation and
  random matrix theory.
\newblock {\em J. Stat. Phys.}, 132(2):275--290, 2008.

\bibitem{Defosseux10}
M.~Defosseux.
\newblock Orbit measures, random matrix theory and interlaced determinantal
  processes.
\newblock {\em Ann. Inst. Henri Poincar\'e Probab. Stat.}, 46(1):209--249,
  2010.

\bibitem{Dieker-Warren09}
A.~B. Dieker and J.~Warren.
\newblock On the largest-eigenvalue process for generalized {W}ishart random
  matrices.
\newblock {\em ALEA Lat. Am. J. Probab. Math. Stat.}, 6:369--376, 2009.

\bibitem{Ferrari-Frings12}
P.~Ferrari and R.~Frings.
\newblock Perturbed {G}{U}{E} minor process and {W}arren's process with drifts,
  2012.
\newblock arXiv:1212.5534.

\bibitem{Ferrari-Frings10}
P.~L. Ferrari and R.~Frings.
\newblock On the partial connection between random matrices and interacting
  particle systems.
\newblock {\em J. Stat. Phys.}, 141(4):613--637, 2010.

\bibitem{Forrester-Nagao11}
P.~J. Forrester and T.~Nagao.
\newblock Determinantal correlations for classical projection processes.
\newblock {\em J. Stat. Mech. Theory Exp.}, (8):P08011, 28, 2011.

\bibitem{Forrester-Nordenstam09}
P.~J. Forrester and E.~Nordenstam.
\newblock The anti-symmetric {GUE} minor process.
\newblock {\em Mosc. Math. J.}, 9(4):749--774, 934, 2009.

\bibitem{Forrester-Rains05}
P.~J. Forrester and E.~M. Rains.
\newblock Interpretations of some parameter dependent generalizations of
  classical matrix ensembles.
\newblock {\em Probab. Theory Related Fields}, 131(1):1--61, 2005.

\bibitem{Forrester-Rains06}
P.~J. Forrester and E.~M. Rains.
\newblock Jacobians and rank 1 perturbations relating to unitary {H}essenberg
  matrices.
\newblock {\em Int. Math. Res. Not.}, pages Art. ID 48306, 36, 2006.

\bibitem{Fulton97}
W.~Fulton.
\newblock {\em Young tableaux}, volume~35 of {\em London Mathematical Society
  Student Texts}.
\newblock Cambridge University Press, Cambridge, 1997.
\newblock With applications to representation theory and geometry.

\bibitem{Glynn-Whitt91}
P.~W. Glynn and W.~Whitt.
\newblock Departures from many queues in series.
\newblock {\em Ann. Appl. Probab.}, 1(4):546--572, 1991.

\bibitem{Gravner-Tracy-Widom02}
J.~Gravner, C.~A. Tracy, and H.~Widom.
\newblock Fluctuations in the composite regime of a disordered growth model.
\newblock {\em Comm. Math. Phys.}, 229(3):433--458, 2002.

\bibitem{Johansson03}
K.~Johansson.
\newblock Discrete polynuclear growth and determinantal processes.
\newblock {\em Comm. Math. Phys.}, 242(1-2):277--329, 2003.

\bibitem{Johansson-Nordenstam06}
K.~Johansson and E.~Nordenstam.
\newblock Eigenvalues of {GUE} minors.
\newblock {\em Electron. J. Probab.}, 11:no. 50, 1342--1371, 2006.

\bibitem{Johansson-Nordenstam06a}
K.~Johansson and E.~Nordenstam.
\newblock Erratum to: ``{E}igenvalues of {GUE} minors'' [{E}lectron. {J}.
  {P}robab. {\bf 11} (2006), no. 50, 1342--1371; mr2268547].
\newblock {\em Electron. J. Probab.}, 12:1048--1051 (electronic), 2007.

\bibitem{Kirillov01}
A.~N. Kirillov.
\newblock Introduction to tropical combinatorics.
\newblock In {\em Physics and combinatorics, 2000 ({N}agoya)}, pages 82--150.
  World Sci. Publ., River Edge, NJ, 2001.

\bibitem{Noumi-Yamada04}
M.~Noumi and Y.~Yamada.
\newblock Tropical {R}obinson-{S}chensted-{K}nuth correspondence and birational
  {W}eyl group actions.
\newblock In {\em Representation theory of algebraic groups and quantum
  groups}, volume~40 of {\em Adv. Stud. Pure Math.}, pages 371--442. Math. Soc.
  Japan, Tokyo, 2004.

\bibitem{O'Connell03}
N.~O'Connell.
\newblock Conditioned random walks and the {RSK} correspondence.
\newblock {\em J. Phys. A}, 36(12):3049--3066, 2003.
\newblock Random matrix theory.

\bibitem{O'Connell12}
N.~O'Connell.
\newblock Directed polymers and the quantum {T}oda lattice.
\newblock {\em Ann. Probab.}, 40(2):437--458, 2012.

\bibitem{Okounkov01}
A.~Okounkov.
\newblock Infinite wedge and random partitions.
\newblock {\em Selecta Math. (N.S.)}, 7(1):57--81, 2001.

\bibitem{Coussement-Van_Assche01}
W.~Van~Assche and E.~Coussement.
\newblock Some classical multiple orthogonal polynomials.
\newblock {\em J. Comput. Appl. Math.}, 127(1-2):317--347, 2001.
\newblock Numerical analysis 2000, Vol. V, Quadrature and orthogonal
  polynomials.

\bibitem{Warren07}
J.~Warren.
\newblock Dyson's {B}rownian motions, intertwining and interlacing.
\newblock {\em Electron. J. Probab.}, 12:no. 19, 573--590, 2007.

\end{thebibliography}

\end{document}